\documentclass[12pt]{amsart}
\usepackage{amssymb,latexsym,graphicx,amscd}
\input xy
\xyoption{all}

\numberwithin{equation}{section}

\newtheorem{theorem}{Theorem}[section]
\newtheorem{proposition}[theorem]{Proposition}

\newtheorem{corollary}[theorem]{Corollary}
\newtheorem{lemma}[theorem]{Lemma}

\theoremstyle{definition}

\newtheorem{remark}[theorem]{Remark}
\newtheorem{example}[theorem]{Example}
\newtheorem{definition}[theorem]{Definition}

\newtheorem{question}[theorem]{Question}

\newcommand{\Aut}{\operatorname{Aut}}
\newcommand{\Mat}{\operatorname{Mat}}

\newcommand{\deff}{\operatorname{Def}}

\newcommand{\Z}{{\mathbb Z}}

\newcommand{\Tr}{\operatorname{Tr}}

\renewcommand{\eqref}[1]{{\rm (\ref{#1})}}

\begin{document}

\title[Quivers with potentials~I]
{Quivers with potentials and their representations~I: Mutations}

\author{Harm Derksen}
\address{\noindent Department of Mathematics, University of Michigan,
Ann Arbor, MI 48109, USA} \email{hderksen@umich.edu}

\author{Jerzy Weyman}
\address{\noindent Department of Mathematics, Northeastern University,
 Boston, MA 02115}
\email{j.weyman@neu.edu}

\author{Andrei Zelevinsky}
\address{\noindent Department of Mathematics, Northeastern University,
 Boston, MA 02115}
\email{andrei@neu.edu}


\begin{abstract}
We study quivers with relations given by non-commutative analogs
of Jacobian ideals in the complete path algebra. This framework
allows us to give a representation-theoretic interpretation of
quiver mutations at arbitrary vertices. This gives a far-reaching
generalization of Bernstein-Gelfand-Ponomarev reflection functors.
The motivations for this work come from several sources:
superpotentials in physics, Calabi-Yau algebras, cluster algebras.
\end{abstract}

\date{April 18, 2007; revised July 10, 2007 and March 12, 2008}

 \thanks{Research of H.~D. supported
by the NSF grant DMS-0349019.
Research of J.~W. supported
by the NSF grant DMS-0600229.
Research of A.~Z. supported by the
NSF grant DMS-0500534 and by a Humboldt Research Award.}

\maketitle
\tableofcontents

\section{Introduction}

The main objects of study in this paper are \emph{quivers with
potentials} (QPs for short). Roughly speaking, a QP is a
quiver~$Q$ together with an element $S$ of the path algebra of~$Q$
such that~$S$ is a linear combination of cyclic paths. We
associate to~$S$ the two-sided ideal~$J(S)$ in the path algebra
generated by the (noncommutative) partial derivatives of $S$ with
respect to the arrows of $Q$. We refer to $J(S)$ as the
\emph{Jacobian ideal}, and to the quotient of the path algebra
modulo $J(S)$ as the \emph{Jacobian algebra}. They appeared in
physicists' work on \emph{superpotentials} in the context of the
Seiberg duality in mirror symmetry (see e.g., \cite{dm,bd,b}).
Since in some of their work the superpotentials are required to
satisfy some form of Serre duality, we prefer not to use this
terminology, and just refer to~$S$ as a \emph{potential}; another
reason for this is that we are working with the completed path
algebra, so our potentials are possibly infinite linear
combinations of cyclic paths.
The Jacobian algebras also play an important role in the recent
work on Calabi-Yau algebras \cite{boc,gi,ir,kr}.

In this paper we introduce and study \emph{mutations} for QPs and their
(decorated) representations.
In the context of Calabi-Yau algebras, the mutations were
discussed in \cite{ir} but our approach is much more elementary
and down-to-earth.
Namely, we develop the setup that directly extends to QPs the
Bernstein-Gelfand-Ponomarev reflection functors
\cite{bgp} and their ``decorated" version \cite{mrz}.

The original motivation for our study comes from the theory of
cluster algebras introduced and studied in a series of papers
\cite{ca1,ca2,ca3,ca4}.
In this paper, we deal only with the underlying combinatorics of
this theory embodied in skew-symmetrizable integer matrices and
their mutations.
Furthermore, we restrict our attention to \emph{skew-symmetric}
integer matrices.
Such matrices can be encoded by quivers without loops and oriented $2$-cycles.
Namely, a skew-symmetric integer $n \times n$ matrix $B = (b_{i,j})$
corresponds to a quiver $Q(B)$ with vertices $1, \dots, n$,
and $b_{i,j}$ arrows from $j$ to $i$ whenever $b_{i,j} > 0$.
For every vertex~$k$, the \emph{mutation}
at~$k$ transforms~$B$ into another skew-symmetric integer
$n \times n$ matrix $\mu_k(B) = \overline B = (\overline b_{i,j})$.
The formula for $\overline b_{i,j}$ is given below in
\eqref{eq:B-mutation}.
It is well-known (see Proposition~\ref{pr:A-B-mutation} below)
that the quiver $Q(\overline B)$ can be obtained from
$Q(B)$ by the following three-step procedure:
\begin{enumerate}
\item[{\bf Step 1.}] For every incoming arrow $a:j \to k$ and every outgoing
arrow $b:k \to i$, create a ``composite" arrow $[ba]:j \to i$;
thus, whenever $b_{i,k}, b_{k,j} > 0$, we create $b_{i,k} b_{k,j}$
new arrows from $j$ to $i$.
\item[{\bf Step 2.}] Reverse all arrows at~$k$; that is, replace each arrow
$a:j \to k$ with $a^\star: k \to j$, and $b:k \to i$ with
$b^\star: i \to k$.
\item[{\bf Step 3.}] Remove any maximal disjoint collection
of oriented $2$-cycles (that can appear as a result
of creating new arrows in Step~$1$).
\end{enumerate}

In the case where~$k$ is a source or a sink of $Q(B)$, the first
and last steps of the above procedure are not applicable, so
$Q(\overline B)$ is obtained from $Q(B)$ by just reversing all the
arrows at~$k$.
In this situation, J.Bernstein, I.~Gelfand, and V.~Ponomarev
\cite{bgp} introduced the \emph{reflection functor} at~$k$ sending
representations of a quiver $Q(B)$ (without relations) into
representations of $Q(\overline B)$.
A modification of these functors acting on  \emph{decorated
representations} was introduced in \cite{mrz} to establish a link
between cluster algebras and quiver representations (the
definition of decorated representations for general QPs
is given below in Section~\ref{sec:reps}).

The elementary approach of \cite{mrz} has not been further pursued until now,
giving way to a more sophisticated approach via cluster categories and
cluster-tilted algebras developed in \cite{BMRRT,BMR1,BMR2,BMR3,CC,CCS,CK1,CK2}
and many other publications.
Most of the results in these papers are for the quivers obtained by mutations from hereditary algebras
(i.e., quivers without oriented cycles and without relations).
In this paper we return to the more elementary point of view of
\cite{mrz} and propose an alternative approach (which is in fact more
general, since we do not impose any restrictions on quivers in question).
In this approach, the mutations at arbitrary vertices (not just sources or sinks)
are defined for QPs and their decorated representations.
The construction for QPs is carried out in
Section~\ref{sec:mutations}, and for their representations in
Section~\ref{sec:reps}.
It turns out to be rather delicate and requires a
lot of technical preparation.
The first two steps of the above mutation procedure extend to QPs
in a relatively straightforward way, but Step 3 presents a real challenge:
we need to accompany the removal of oriented $2$-cycles from a
quiver with a suitable modification of the potential, leaving the
corresponding Jacobian algebra unchanged.
Our main device in dealing with this difficulty is
Theorem~\ref{th:trivial-reduced-splitting}, which is the crucial
technical result of the paper.
Roughly speaking, Theorem~\ref{th:trivial-reduced-splitting}
asserts that every potential~$S$ can be transformed by an
automorphism of the path algebra into the sum of two potentials
$S_{\rm triv}$ and $S_{\rm red}$ on the disjoint sets of arrows,
where the \emph{trivial} part $S_{\rm triv}$
is a linear combination of cyclic $2$-paths, while
the \emph{reduced} part $S_{\rm red}$ involves only cyclic $d$-paths
with $d \geq 3$.
Furthermore, the Jacobian algebra of $S_{\rm red}$ is isomorphic
to that of~$S$.

Several comments on this result are in order.
First, our arguments heavily depend on the setup using completed
path algebras, thus allowing potentials to involve infinite sums of cyclic paths.
Second, the reduction $S \mapsto S_{\rm red}$ is not given by a
canonical procedure.
As a consequence, our construction of mutations for QPs and their representations
is not functorial in any obvious sense.
On the positive side, we prove that every mutation is a
well-defined transformation on the right-equivalence
classes of QPs (and their representations), where, roughly speaking, two QPs are
\emph{right-equivalent} if they can be obtained from each other by an
automorphism of the path algebra (for more precise definitions
see Definitions~\ref{def:AS-isomorphism} and
\ref{def:right-equiv-reps}).

Finally, it is important to keep in mind that, even with the help
of Theorem~\ref{th:trivial-reduced-splitting},
in order to get rid of all oriented $2$-cycles in the
mutated QP, one needs to impose some ``genericity" conditions on
the initial potential~$S$.
These conditions are studied in Section~\ref{sec:generic}.
They are not very explicit in general, but
we introduce an important class or \emph{rigid} QPs
(see  Definitions~\ref{def:deff-space}
and \ref{def:rigid-QP}) for which the absence of oriented $2$-cycles
after any sequence of QP mutation is guaranteed.

We now describe the contents of the paper in more detail.
In Section~\ref{sec:path-algebras} we introduce an algebraic setup
for dealing with quivers and their path algebras.
We fix a base field $K$, and encode a quiver with
the vertex set $Q_0$ and the arrow set $Q_1$ by its \emph{vertex
span} $R = K^{Q_0}$ and \emph{arrow span} $A=K^{Q_1}$.
Thus, $R$ is a finite-dimensional commutative $K$-algebra, and~$A$
is a finite-dimensional $R$-bimodule.
We then introduce the \emph{path algebra}
$$R\langle A\rangle=\bigoplus_{d=0}^\infty A^d,$$
and, more importantly for our purposes, the \emph{complete path algebra}
$$R\langle\langle A \rangle\rangle=\prod_{d=0}^\infty A^d;$$
here $A^d$ stands for the $d$-fold tensor power of~$A$ as an
$R$-bimodule.
We view $R\langle\langle A \rangle\rangle$ as a topological
algebra via the ${\mathfrak m}$-adic topology, where
${\mathfrak m}$ is the two-sided ideal generated by~$A$.

In Section~\ref{sec:potentials} we introduce some of our main objects of study:
potentials and their Jacobian ideals.
It is natural to view potentials as elements of the
\emph{trace space}
$R\langle\langle A \rangle\rangle/\{R\langle\langle A
\rangle\rangle, R\langle\langle A \rangle\rangle\}$,
where $\{R\langle\langle A
\rangle\rangle, R\langle\langle A \rangle\rangle\}$
is the closure of the vector subspace
in~$R\langle\langle A \rangle\rangle$ spanned by all commutators.
It is more convenient for us to define a potential~$S$ as an element
of the cyclic part of $R\langle\langle A \rangle\rangle$; for all practical purposes, $S$ can be
replaced by a \emph{cyclically equivalent} potential, that is, the
one with the same image in the trace space.
To define the Jacobian ideal $J(S)$ and derive its basic properties,
we develop the formalism of \emph{cyclic derivatives}, in
particular, establishing ``cyclic" versions of the Leibniz rule and the chain rule.
The main result of Section~\ref{sec:potentials} is Proposition~\ref{pr:automorphism-respects-jacobian}
that asserts that any isomorphism~$\varphi$ of path algebras sends
$J(S)$ to $J(\varphi(S))$.
Note that cyclic derivatives for general non-commutative algebras
were introduced in \cite{rss}, and the results we present can be
easily deduced from those in loc.cit.
For the convenience of the reader, we
present complete independent proofs.
Victor Ginzburg informed us that in the context of path algebras
of quivers, cyclic derivatives were introduced and studied in
\cite{blb,gi1}, and that Proposition~\ref{pr:automorphism-respects-jacobian}
is a consequence of the geometric interpretation of $J(S)$
given in \cite[Definition~5.1.1, Lemma~5.1.3]{gi}.

In Section~\ref{sec:QP} we introduce quivers with potentials
(QPs) and define the right-equivalence relation on them, which plays
an important role in the paper.
We then state and prove the key technical result of the paper: Splitting
Theorem~\ref{th:trivial-reduced-splitting}, already discussed above.
The proof is elementary but pretty involved; it
uses in an essential way the topology in a complete path algebra.
In order not to interrupt the argument, we move to the Appendix
our treatment of the topological properties needed for the proof of one
of the technical lemmas.

In Section~\ref{sec:mutations} we finally introduce the mutations of QPs.
Using Theorem~\ref{th:trivial-reduced-splitting}, we prove that
the mutation at an arbitrary vertex is a well-defined involution on
the set of right-equivalence classes of reduced QPs
(Theorem~\ref{th:mutation-involutive}).

In Section~\ref{sec:mut-invariants}, we study some mutation
invariants of QPs.
In particular, we show that mutations preserve the class of QPs
with finite-dimensional Jacobian algebras
(Corollary~\ref{cor:mutation-preserves-fin-dim}).
Another important property of QPs preserved by mutations is
\emph{rigidity} (Corollary~\ref{cor:mutation-preserves-rigidity}),
which was already mentioned above.
For the precise definition of rigid QPs see Definitions~\ref{def:deff-space}
and \ref{def:rigid-QP} below; intuitively,
a QP is rigid if its right-equivalence class is invariant under
infinitesimal deformations.

In Section~\ref{sec:generic}, we introduce and study \emph{nondegenerate} QPs,
that is, those to which one can apply an arbitrary sequence of mutations without creating
oriented $2$-cycles.
In Corollary~\ref{cor:nondegenerate-exists} we show that nondegeneracy is guaranteed
by non-vanishing of countably many nonzero polynomial functions on
the space of potentials.
In particular, if the base field~$K$ is uncountable,
a nondegenerate QP exists for every underlying quiver.

Section~\ref{sec:QP-matrix-mutations} contains some examples of
rigid and non-rigid potentials and some further
results illustrating the importance of rigidity.
A simple but important Proposition~\ref{pr:rigidity-no-2-cycles}
asserts that rigid QPs have no oriented $2$-cycles.
Combining this with the fact that rigidity is preserved by mutations,
we conclude that every rigid QP is nondegenerate.
Using a result by Keller-Reiten \cite{kr}, we show in
Example~\ref{ex:triangular-grid} that
the class of rigid  QPs (as well as the class of QPs with
finite-dimensional Jacobian algebras) is strictly greater
than the class of QPs mutation-equivalent to acyclic ones.
On the other hand, Example~\ref{ex:double-triangle}
exhibits an underlying quiver without oriented $2$-cycles
that does not admit a rigid QP; thus, the class of nondegenerate QPs
is strictly greater than the class of rigid ones.

In Section~\ref{sec:Dynkin-rigid}, we consider quivers that are
mutation-equivalent to a Dynkin quiver.
For every such underlying quiver, we compute explicitly
the corresponding rigid QP (Proposition~\ref{pr:finite-type-potential}).
Comparing this result with the description of cluster-tilted
algebras obtained in \cite{CCS,BMR2}, we conclude in
Corollary~\ref{cor:jacobian-cluster-tilted}
that in the case in question, every cluster-tilted algebra can be
identified with the Jacobian algebra of the corresponding rigid QP.
Thus, Jacobian algebras  can be viewed as generalizations of cluster-tilted algebras.

In Section~\ref{sec:reps} we introduce decorated representations
of QPs (QP-representations, for short)
and their right-equivalence (Definitions~\ref{def:dec-rep}
and \ref{def:right-equiv-reps}).
We then present a representation-theoretic extension of
Splitting Theorem~\ref{th:trivial-reduced-splitting}
by defining the reduced part of a QP-representation $\mathcal M$
(Definition~\ref{def:reduced-part-rep}) and proving that, up to
right-equivalence, it is determined by the right-equivalence class
of $\mathcal M$ (Proposition~\ref{pr:rep-splitting}).
We use this result to introduce mutations of QP-representations
and to prove a representation-theoretic extension of
Theorem~\ref{th:mutation-involutive}: the mutation at every vertex
is an involution on the set of right-equivalence classes
of reduced QP-representations
(Theorem~\ref{th:rep-mutation-involutive}).

Some examples of QP-representations and their mutations are given
in Section~\ref{sec:band}.
All these examples treat quivers with three vertices.
In particular, we describe the effect of mutations on a special
family of \emph{band representations} coming from the theory of
string algebras \cite{BR,F}.

The concluding Section~\ref{sec:open-problems} contains some open problems about QPs and their
representations that we find essential for better understanding of the theory.

In the forthcoming continuation of this paper, we plan to discuss
applications of QP-representations and their mutations to the structure of
the corresponding cluster algebras.

\section{Quivers and path algebras}
\label{sec:path-algebras}

A \emph{quiver} $Q=(Q_0,Q_1,h,t)$ consists of a pair of finite
sets $Q_0$ (\emph{vertices}) and $Q_1$ (\emph{arrows}) supplied
with two maps $h:Q_1 \to Q_0$ (\emph{head}) and $t:Q_1 \to Q_0$
(\emph{tail}).
It is represented as a directed graph with the set of vertices
$Q_0$ and directed edges $a: ta \rightarrow ha$ for $a \in Q_1$.
Note that this definition allows the underlying graph to have
multiple edges and (multiple) loops.

We fix a field $K$, and associate to a quiver $Q$ two vector
spaces $R = K^{Q_0}$ and $A=K^{Q_1}$ consisting of $K$-valued
functions on $Q_0$ and $Q_1$, respectively.
We will sometimes refer to $R$ as the \emph{vertex span} of~$Q$,
and to~$A$ as the \emph{arrow span} of~$Q$.
The space $R$ is a commutative algebra under the pointwise multiplication of functions.
The space $A$ is an $R$-bimodule, with the bimodule structure defined as follows:
if $e \in R$ and $f \in A$ then $(e\cdot f)(a)=e(ha)f(a)$
and $(f\cdot e)(a)=f(a)e(ta)$ for all $a\in Q_1$.

We denote by~$Q^\star$ the \emph{dual} or \emph{opposite} quiver~$Q^\star$
obtained by reversing the arrows in~$Q$ (i.e., replacing
$Q=(Q_0,Q_1,h,t)$ with $Q^\star = (Q_0,Q_1,t,h)$).
The corresponding arrow span is naturally identified with
the dual bimodule $A^\star$ (the dual vector space of $A$
with the standard $R$-bimodule structure).

For a given vertex set $Q_0$ with the vertex span~$R$, every finite-dimensional
$R$-bimodule~$B$ is the arrow span of some quiver on~$Q_0$.
To see this, consider the elements $e_i \in R$ for $i \in Q_0$
given by $e_i(j)=\delta_{i,j}$ (the Kronecker delta symbol).
They form a basis of idempotents of $R$, hence every
$R$-bimodule~$B$ has a direct sum decomposition
$$B = \bigoplus_{i, j \in Q_0} B_{i,j},$$
where $B_{i,j}= e_i B e_j \subseteq B$ for every $i,j\in Q_0$.
If~$B$ is finite-dimensional, we can identify the (finite) set of arrows
$Q_1$ with a $K$-basis in $B$ which is the union of bases in all components $B_{i,j}$;
under this identification, every $a \in Q_1 \cap B_{i,j}$ has
$h(a) = i$ and $t(a) = j$.

It is convenient to represent an $R$-bimodule $B$ by a matrix of vector spaces
$(B_{i,j})$ whose rows and columns are labeled by vertices.
In this model, the left (resp. right) action of an element
$c = \sum_i c_i e_i \in R$ is given by the left (resp. right) multiplication
by the diagonal matrix with diagonal entries $c_i$.
And the tensor product over $R$ is given by a usual matrix
multiplication: if $B = \bigoplus_{i, j} B_{i,j}$ and
$C = \bigoplus_{i, j} C_{i,j}$, then
$$(B \otimes_R C)_{i,j} = \bigoplus_{k} (B_{i,k} \otimes C_{k,j}).$$

Returning to a quiver~$Q$ with the arrow span~$A$, for each nonnegative integer~$d$,
let $A^d$  denote the $R$-bimodule
$$
A^d
= \underbrace{A\otimes_R A\otimes_R \cdots \otimes_R A}_d,
$$
with the convention $A^0=R$.

\begin{definition}
\label{def:path-algebra}
The \emph{path algebra} of $Q$ is
defined as
the (graded) tensor algebra
$$
R\langle A\rangle=\bigoplus_{d=0}^\infty A^d.
$$
For each $i, j \in Q_0$, the component
$R\langle A\rangle_{i,j} = e_i R\langle A\rangle e_j$
is called the \emph{space of paths} from $j$ to $i$.
\end{definition}

As above, we identify the set of arrows $Q_1$ with some basis of $A$
consisting of homogeneous elements, that is, each
$a \in Q_1$ belongs to some component $A_{i,j}$.
Then for every $d \geq 1$,  the products $a_1  \cdots a_d$ such
that all $a_k$ belong to $Q_1$, and $t(a_k) = h(a_{k+1})$ for
$1 \leq k < d$, form a $K$-basis of $A^d$.
We call this basis the \emph{path basis} of~$A^d$ associated
to~$Q_1$.
For $d=0$, we call $\{e_i \mid i \in Q_0\}$ the path basis of
$A^0 = R$.
We refer to the union of path bases for all~$d$ as the path basis
of $R\langle A\rangle$.
The elements of the path basis will be sometimes referred to
simply as \emph{paths}.
We depict $a_1  \cdots a_d$ as a path of length~$d$ starting in
the vertex $t(a_d)$ and ending in $h(a_1)$.
Note that the product $(a_1  \cdots a_d)(a_{d+1}  \cdots a_{d+k})$
of two paths is~$0$ unless $t(a_d) = h(a_{d+1})$, in which case
the product is given by concatenation of paths.
This description implies the following:
\begin{equation}
\label{eq:nonzero-products}
\text{If $0 \neq p \in A^k e_i$ and $0 \neq q \in e_i A^\ell$
for some vertex~$i$
then $pq \neq 0$.}
\end{equation}

\begin{definition}
\label{def:complete-path-algebra}
The \emph{complete path algebra} of $Q$ is defined as
$$
R\langle\langle A \rangle\rangle=\prod_{d=0}^\infty A^d.
$$
Thus, the elements of $R\langle\langle A \rangle\rangle$ are
(possibly infinite) $K$-linear combinations of the elements of a
path basis in $R\langle A \rangle$; and the multiplication in
$R\langle\langle A \rangle\rangle$ naturally extends the multiplication in
$R\langle A \rangle$.
\end{definition}

Note that, if the quiver $Q$ is \emph{acyclic} (that is, has no oriented cycles),
then $A^d = \{0\}$ for $d \gg 0$, hence in this case $R\langle\langle A \rangle\rangle
= R\langle A \rangle$, and this algebra is finite-dimensional.

\begin{example}
Consider the quiver $Q=(Q_0,Q_1)$ with $Q_0=\{1\}$ and $Q_1=\{a\}$
with $a:1\to 1$. This is the loop quiver:
$$
\xymatrix{
1\ar@(ul,dl)[]_{a}
}\ .
$$
In this case $R=K^{Q_0}=K$, and $A=K^{Q_1}= Ka$.
We have $R\langle A\rangle=K[a]$, and $R\langle\langle A\rangle\rangle=K[[a]]$,
the algebra of formal power series.
\end{example}

Let ${\mathfrak m} = {\mathfrak m}(A)$ denote the (two-sided) ideal
of $R\langle\langle A\rangle\rangle$ given by
\begin{equation}
\label{eq:max-ideal}
{\mathfrak m}  = {\mathfrak m}(A) = \prod_{d=1}^\infty A^d.
\end{equation}
Thus the powers of~${\mathfrak m}$ are given by
$${\mathfrak m}^n = \prod_{d=n}^\infty A^d.$$
We view $R\langle\langle A\rangle\rangle$ as a topological $K$-algebra
via the \emph{${\mathfrak m}$-adic topology} having
the powers of ${\mathfrak m}$ as a basic system of open
neighborhoods of~$0$.
Thus, the closure of any subset $U \subseteq R\langle\langle A\rangle\rangle$
is given by
\begin{equation}
\label{eq:closure}
\overline U = \bigcap_{n=0}^\infty (U + {\mathfrak m}^n).
\end{equation}
It is clear that $R\langle A\rangle$ is a  dense subalgebra
of $R\langle\langle A\rangle\rangle$.

In dealing with $R\langle\langle A\rangle\rangle$,
the following fact is quite useful: every
(non-commutative) formal power series over~$R$
in a finite number of variables can be evaluated at arbitrary elements of ${\mathfrak m}$
to obtain a well-defined element of $R\langle\langle A\rangle\rangle$.
To illustrate, let us show that ${\mathfrak m}$ is the unique maximal
two-sided ideal of $R\langle\langle A\rangle \rangle$
having zero intersection with $R = A^0$.
Indeed, it is enough to show that any element
$x \in R\langle\langle A\rangle \rangle - {\mathfrak m}$
generates an ideal having nonzero intersection with $R$.
Let $x = c + y$ with $c$ a nonzero element of $R$,
and $y \in {\mathfrak m}$.
Multiplying $x$ on both sides by suitable elements of $R$, we can
assume that $c = e_i$ for some $i \in Q_0$, and $e_i y = y e_i = y$.
But then $z = e_i - y + y^2 - y^3 + \cdots$ is a well-defined element
of $R\langle\langle A\rangle \rangle$, and we have $xz = e_i$,
proving our claim.

This characterization of ${\mathfrak m}$ implies that it
is invariant under any algebra automorphism~$\varphi$
of $R\langle\langle A\rangle \rangle$ such that $\varphi|_R$ is
the identity.
Thus,~$\varphi$ is continuous, i.e., is an automorphism of
$R\langle\langle A\rangle \rangle$ as a topological algebra.

The same argument shows that, more generally, if $A$ and $A'$ are finite-dimensional
$R$-bimodules then any algebra homomorphism 
$\varphi: R\langle\langle A\rangle \rangle \to R\langle\langle A'\rangle \rangle$
such that $\varphi|_R = {\rm id}$, sends ${\mathfrak m}(A)$ into ${\mathfrak m}(A')$, hence is
continuous.
Thus,~$\varphi$ is uniquely determined by its
restriction to $A^1 = A$, which is a $R$-bimodule homomorphism
$A \to {\mathfrak m}(A') = A' \oplus {\mathfrak m}(A')^2$.
We write $\varphi|_A = (\varphi^{(1)}, \varphi^{(2)})$, where
$\varphi^{(1)}:A \to A'$ and $\varphi^{(2)}:A \to {\mathfrak m}(A')^2$ are
$R$-bimodule homomorphisms.

\begin{proposition}
\label{pr:automorphisms}
Any pair $(\varphi^{(1)}, \varphi^{(2)})$
of $R$-bimodule homomorphisms
$\varphi^{(1)}: A \to A'$ and $\varphi^{(2)}: A \to {\mathfrak m}(A')^2$
gives rise to a unique homomorphism of topological algebras
$\varphi: R\langle\langle A\rangle \rangle \to R\langle\langle A'\rangle \rangle$
such that $\varphi|_R = {\rm id}$, and $\varphi|_A = (\varphi^{(1)}, \varphi^{(2)})$.
Furthermore, $\varphi$ is an isomorphism
if and only if $\varphi^{(1)}$ is a $R$-bimodule isomorphism $A \to A'$.
\end{proposition}

\begin{proof}
The uniqueness of $\varphi$ is clear.
To show the existence, let $Q_1= \{a_1, \dots, a_N\} \subset A = A^1$ be the set of arrows in~$A$.
Writing an element $x \in R\langle\langle A\rangle\rangle$ as an
infinite $K$-linear combination of the elements of the
corresponding path basis in $R\langle A\rangle$, we express~$x$ as
a (non-commutative) formal power series $F(a_1, \dots, a_N)$.
Therefore, a desired algebra homomorphism can be obtained by setting
$\varphi(x) = F(\varphi^{(1)}(a_1) + \varphi^{(2)}(a_1), \dots,
\varphi^{(1)}(a_N) + \varphi^{(2)}(a_N))$.

If $\varphi$ is an isomorphism then $\varphi^{(1)}: A \to A'$ is
clearly an isomorphism of $R$-bimodules.
To show the converse implication, we can identify $A$ and $A'$
with the help of~$\varphi^{(1)}$, and so assume that $\varphi^{(1)}$ is
the identity automorphism of~$A$.
Then the (infinite) matrix that expresses $\varphi$ as a
$K$-linear map in the path basis of $R\langle\langle A\rangle\rangle$
is lower-triangular with all the diagonal entries equal to~$1$
(we order the basis elements so that their degrees weakly increase).
Clearly, such a matrix is invertible, completing the proof of
Proposition~\ref{pr:automorphisms}.
\end{proof}

\begin{definition}
\label{def:automorphisms}
Let $\varphi$ be the automorphism of $R\langle\langle A\rangle \rangle$
corresponding to a pair $(\varphi^{(1)}, \varphi^{(2)})$ as in
Proposition~\ref{pr:automorphisms}.
If~$\varphi^{(2)} = 0$, then we call $\varphi$ a \emph{change of arrows}.
If $\varphi^{(1)}$ is the identity automorphism of~$A$, we say that~$\varphi$
is a \emph{unitriangular} automorphism; furthermore, we
say that $\varphi$ is of \emph{depth} $d \geq 1$, if $\varphi^{(2)}(A)
\subset {\mathfrak m}^{d+1}$.
\end{definition}

The following property of unitriangular automorphisms is immediate
from the definitions:
\begin{align}
\label{eq:unitriangular-d}
&\text{If~$\varphi$ is an unitriangular automorphism of $R\langle\langle A
\rangle\rangle$ of depth~$d$,}\\
\nonumber
&\text{then
$\varphi(u) - u \in  {\mathfrak m}^{n+d}$ for $u \in {\mathfrak m}^{n}$.}
\end{align}

\section{Potentials and their Jacobian ideals}
\label{sec:potentials}

In this section we introduce some of our main objects of study:
potentials and their Jacobian ideals in the
complete path algebra $R\langle\langle A\rangle\rangle$ given by
Definition~\ref{def:complete-path-algebra}.
We fix a path basis in $R\langle A\rangle$;
recall that it consists of the elements $e_i \in R = A^0$ together with the products
$a_1  \cdots a_d$ (paths) such that all $a_k$ belong to $Q_1$, and $t(a_k) = h(a_{k+1})$ for
$1 \leq k < d$.
Then each space $A^d$ has a direct $R$-bimodule
decomposition $A^d = \bigoplus_{i,j \in Q_0} A^d_{i,j}$, where
the component $A^d_{i,j}$ is spanned by the paths $a_1  \cdots a_d$
with $h(a_1) = i$ and $t(a_d) = j$.

\pagebreak

\begin{definition}\
\label{def:cyclic-stuff}
\begin{itemize}
\item For each $d \geq 1$, we define the \emph{cyclic part} of $A^d$ as the
sub-$R$-bimodule $A^d_{\rm cyc} = \bigoplus_{i \in Q_0} A^d_{i,i}$.
Thus, $A^d_{\rm cyc}$ is the span of all paths $a_1  \cdots a_d$
with $h(a_1) = t(a_d)$; we call such paths \emph{cyclic}.
\item We define a closed vector subspace
$R\langle\langle A\rangle\rangle_{\rm cyc} \subseteq R\langle\langle A\rangle\rangle$
by setting
$$R\langle\langle A\rangle\rangle_{\rm cyc} =
\prod_{d=1}^\infty A^d_{\rm cyc},$$
and call the elements of $R\langle\langle A\rangle\rangle_{\rm cyc}$
\emph{potentials}.
\item For every $\xi \in A^\star$, we define the \emph{cyclic
derivative} $\partial_\xi$ as the continuous $K$-linear map
$R\langle\langle A\rangle\rangle_{\rm cyc} \to R\langle\langle A\rangle\rangle$
acting on paths by
\begin{equation}
\label{eq:cyclic-derivative}
\partial_\xi (a_1 \cdots a_d) =
\sum_{k=1}^d \xi(a_k) a_{k+1} \cdots a_d a_1 \cdots a_{k-1}.
\end{equation}
\item
For every potential~$S$, we
define its \emph{Jacobian ideal} $J(S)$ as the closure of
the (two-sided) ideal in $R\langle\langle A\rangle\rangle$
generated by the elements $\partial_\xi(S)$ for all $\xi \in A^\star$
(see \eqref{eq:closure}); clearly, $J(S)$ is a two-sided
ideal in $R\langle\langle A\rangle\rangle$.
\item We call the quotient $R\langle\langle A\rangle\rangle/J(S)$
the \emph{Jacobian algebra} of~$S$, and denote
it by ${\mathcal P}(Q,S)$ or ${\mathcal P}(A,S)$.
\end{itemize}
\end{definition}

An easy check shows that a cyclic derivative
$\partial_\xi: R\langle\langle A\rangle\rangle_{\rm cyc}
\to R\langle\langle A\rangle\rangle$
does not depend on the choice of a path basis.
Furthermore, cyclic derivatives do not distinguish between the
potentials that are equivalent in the following sense.

\begin{definition}
\label{def:cyclic-equivalence}
Two potentials $S$ and $S'$ are \emph{cyclically equivalent}
if $S - S'$ lies in the closure of the span of
all elements of the form $a_1 \cdots a_d - a_2 \cdots a_d a_1$,
where $a_1 \cdots a_d$ is a cyclic path.
\end{definition}

The following proposition is immediate from \eqref{eq:cyclic-derivative}.

\begin{proposition}
\label{pr:cyclic-equivalence}
If two potentials $S$ and $S'$ are cyclically equivalent,
then $\partial_\xi (S) = \partial_\xi(S')$ for all $\xi \in A^\star$,
hence
$J(S) = J(S')$ and ${\mathcal P}(A,S) = {\mathcal P}(A,S')$.
\end{proposition}

It is easy to see that the definition of cyclical equivalence does
not depend on the choice of a path basis.
In fact, it can be given in more invariant terms as follows.

\begin{definition}
\label{def:trace-space}
For any topological $K$-algebra $U$, its \emph{trace space}
$\Tr(U)$ is defined as
$\Tr (U) = U/\{U,U\}$, where $\{U,U\}$ is the closure of the vector subspace
in~$U$ spanned by all commutators.
We denote by $\pi = \pi_U: U \to \Tr(U)$  the canonical projection.
\end{definition}

The following proposition is a direct consequence of the definitions.

\begin{proposition}
\label{pr:cyclical-equiv-thru-trace}
Two potentials $S$ and $S'$ are cyclically
equivalent if and only if
$\pi_{R\langle\langle A\rangle\rangle}(S) =
\pi_{R\langle\langle A\rangle\rangle}(S')$.
Thus, the Jacobian ideal and the Jacobian algebra of a
potential~$S$ depend only on the image of~$S$ in
$\Tr(R\langle\langle A\rangle\rangle)$.
\end{proposition}

Recall that we identify the set of arrows $Q_1$ with a $K$-basis in $A = A^1$.
For $a \in Q_1$, we will use the notation $\partial_a$ for the
cyclic derivative $\partial_{a^\star}$, where
$Q_1^\star = \{a^\star \mid a \in Q_1\}$ is the dual basis of $Q_1$ in $A^\star$.

\begin{example}
\label{ex:two-arrows}
Consider the quiver $Q=(Q_0,Q_1)$ with $Q_0=\{1,2\}$
and $Q_1=\{a,b\}$, where $a:1\to2$ and $b:2\to1$:
$$
\xymatrix{
1\ar@<.5ex>[r]^a & 2\ar@<.5ex>[l]^b}\ .
$$
The vertex and arrow spans of~$Q$ are given by
 $R=K^{Q_0}=Ke_1 \oplus Ke_2$, and $A=K^{Q_1}= Ka\oplus Kb$.
The paths in $R\langle\langle A\rangle\rangle$
are $e_1$, $e_2$ and all products of the generators $a$ and $b$ in which the
factors $a$ and $b$ alternate.
The potentials are (possibly infinite) linear combinations
of the elements $(ab)^n$ and $(ba)^n$ for all $n \geq 1$.
Using \eqref{eq:cyclic-derivative}, we obtain
$$\partial_a ((ab)^n) = \partial_a ((ba)^n)= nb (ab)^{n-1}, \quad
\partial_b ((ab)^n) = \partial_b ((ba)^n)= na (ba)^{n-1}
\quad (n \geq 1).$$
Up to cyclical equivalence, every potential can be written in the form
$$
\sum_{n=1}^\infty \alpha_n (ab)^n \quad (\alpha_n \in K).
$$
\end{example}

\medskip

Returning to the general theory,
it is clear that every algebra homomorphism
$$\varphi:R\langle\langle A\rangle\rangle\to R\langle\langle A'\rangle \rangle$$
such that $\varphi|_R = {\rm id}$,
sends potentials to potentials.

\begin{proposition}
\label{pr:automorphism-respects-jacobian}
Every algebra isomorphism
$$\varphi:R\langle\langle A\rangle\rangle\to R\langle\langle A'\rangle \rangle$$
such that $\varphi|_R = {\rm id}$,
sends
$J(S)$ onto $J(\varphi(S))$, inducing an isomorphism of Jacobian algebras
${\mathcal P}(A,S) \to {\mathcal P}(A',\varphi(S))$.
\end{proposition}

\begin{proof}
We start by developing some ``differential calculus" for cyclic derivatives.
We need a few pieces of notation.
We set
$$R\langle\langle A\rangle\rangle \widehat \otimes R\langle\langle A\rangle\rangle
= \prod_{d,e \geq 0} (A^d \otimes A^e)$$
(the tensor product on the right is over the base field~$K$),
and view this space as a topological vector space with
a basic system of open neighborhoods of~$0$ formed by the sets
$
\prod_{d+e \geq n} (A^d \otimes A^e)$ for all $n \geq 0$;
thus, $R\langle A\rangle \otimes R\langle A\rangle$ is dense in
$R\langle\langle A\rangle\rangle \widehat \otimes R\langle\langle A\rangle \rangle$.
Now, for every $\xi \in A^\star$, we define a continuous $K$-linear map
$$
\Delta_\xi: R\langle\langle A\rangle\rangle\to
R\langle\langle A\rangle\rangle \widehat \otimes R\langle\langle A\rangle\rangle
$$
by setting $\Delta_\xi (e) = 0$ for $e \in R = A^0$, and
\begin{equation}
\label{eq:delta-xi}
\Delta_\xi(a_1 \cdots a_d) = \sum_{k=1}^d \xi(a_k) a_1 \cdots
a_{k-1} \otimes a_{k+1} \cdots a_d
\end{equation}
for any path $a_1 \cdots a_d$ of length $d \geq 1$.
Note that $\Delta_\xi$ does not depend on the choice of a path basis.
We will use the same convention as for cyclic derivatives: for $a \in Q_1$,
we write $\Delta_a$ instead of $\Delta_{a^\star}$.
For instance, in the situation of Example~\ref{ex:two-arrows}, we have
$$\Delta_a ((ab)^n) = \sum_{k=1}^n (ab)^{k-1} \otimes b (ab)^{n-k}, \quad
\Delta_b ((ab)^n) = \sum_{k=1}^n (ab)^{k-1}a \otimes (ab)^{n-k}.$$

Next, we denote by $(f,g) \mapsto f \square g$
a continuous $K$-bilinear map
$(R\langle\langle A\rangle\rangle \widehat \otimes
R\langle \langle A\rangle\rangle) \times R\langle\langle A\rangle\rangle
\to R\langle\langle A\rangle\rangle$
given by
\begin{equation}
\label{eq:square}
(u \otimes v) \square g= v g u
\end{equation}
for $u, v \in R\langle A\rangle$.
We are now ready to state the Leibniz rule.

\begin{lemma}[Cyclic Leibniz rule]
Let $f \in R\langle\langle A\rangle\rangle_{i,j}$ and
$g \in R\langle\langle A\rangle\rangle_{j,i}$ for some
vertices $i$ and $j$.
Then for every $\xi \in A^\star$, we have
\begin{equation}
\label{eq:leibniz}
\partial_\xi (fg)=\Delta_\xi(f)\square g + \Delta_\xi(g)\square f.
\end{equation}
More generally, for any finite sequence of vertices
$i_1, \dots, i_d, i_{d+1} = i_1$ and for any $f_1, \dots f_d$ such that
$f_k \in R\langle\langle A\rangle\rangle_{i_k,i_{k+1}}$,
we have
\begin{equation}
\label{eq:leibniz-several}
\partial_\xi (f_1 \cdots f_d)= \sum_{k=1}^d
\Delta_\xi(f_k)\square (f_{k+1}\cdots f_d f_1 \cdots f_{k-1}).
\end{equation}
\end{lemma}

\begin{proof}
It is enough to check \eqref{eq:leibniz} in the case where
$f = a_1 \cdots a_d$ and $g = a_{d+1} \cdots a_{d+s}$ are two
paths such that $t(a_d) = h(a_{d+1})$ and $t(a_{d+s}) = h(a_1)$.
Using \eqref{eq:cyclic-derivative}, we obtain
$$\partial_\xi (fg) =
\sum_{k=1}^{d+s} \xi(a_k) a_{k+1} \cdots a_{d+s} a_1 \cdots a_{k-1}.$$
Comparing this expression with \eqref{eq:delta-xi} and
\eqref{eq:square}, we see that the part of the last sum where~$k$
runs from~$1$ to~$d$ (resp.  from~$d+1$ to~$d+s$) is equal to
$\Delta_\xi(f)\square g$ (resp. to $\Delta_\xi(g)\square f$),
proving \eqref{eq:leibniz}.
The identity \eqref{eq:leibniz-several} follows from
\eqref{eq:leibniz} by induction on~$d$.
\end{proof}

\begin{lemma}[Cyclic chain rule]
\label{lem:chain-rule}
Suppose that $\varphi:R\langle\langle A\rangle\rangle\to R\langle\langle
A'\rangle\rangle$ is an algebra homomorphism as in
Proposition~\ref{pr:automorphisms}.
Then, for every potential $S \in R\langle\langle A\rangle\rangle)_{\rm cyc}$
and $\xi \in A'^\star$, we have:
\begin{equation}
\label{eq:chain-rule}
\partial_{\xi}(\varphi(S)) =
\sum_{a \in Q_1} \Delta_\xi (\varphi(a)) \square
\varphi(\partial_{a}(S)).
\end{equation}
\end{lemma}

\begin{proof}
It suffices to treat the case where $S = a_{1} \cdots a_{d}$ is a cyclic path.
Applying \eqref{eq:leibniz-several} and \eqref{eq:cyclic-derivative}, we obtain
\begin{align*}
\partial_\xi(\varphi(S)) &=
\sum_{k=1}^d
\Delta_\xi(\varphi(a_k)) \square (\varphi(a_{k+1} \cdots
a_d a_1 \cdots a_{k-1}))\\
&= \sum_{a \in Q_1} \Delta_\xi (\varphi(a)) \square
\varphi(\sum_{k: a_k = a} a_{k+1} \cdots
a_d a_1 \cdots a_{k-1})\\
&=
\sum_{a \in Q_1} \Delta_\xi (\varphi(a)) \square
\varphi(\partial_{a}(S)),
\end{align*}
as desired.
\end{proof}

Now we are ready to prove
Proposition~\ref{pr:automorphism-respects-jacobian}.
By Lemma~\ref{lem:chain-rule}, for every
$\xi \in A'^\star$, the element
$\partial_\xi(\varphi(S))$ lies in the ideal generated by
the elements $\varphi(\partial_{a}(S))$ for $a \in Q_1$, hence, it
lies in $\varphi(J(S))$.
Thus, we have the inclusion
$$
J(\varphi(S)) \subseteq \varphi(J(S)).
$$
We can also apply this to the inverse isomorphism $\varphi^{-1}$
and the potential $\varphi(S)$:
$$
J(S)=J(\varphi^{-1}(\varphi(S))) \subseteq \varphi^{-1}(J(\varphi(S)).
$$
Applying $\varphi$ to both sides yields
$$
\varphi(J(S))\subseteq J(\varphi(S)),
$$
completing the proof.
\end{proof}

\section{Quivers with potentials}
\label{sec:QP}

We now introduce our main objects of study.

\begin{definition}
\label{def:QP}
Suppose $Q$ is a quiver with the arrow span~$A$, and
$S \in  R \langle \langle A \rangle \rangle_{\rm cyc}$ is a potential.
We say that a pair $(Q,S)$ (or $(A,S)$) is a \emph{quiver with potential}
(QP for short) if it satisfies the following two conditions:
\begin{equation}
\label{eq:no-loops}
\text{The quiver~$Q$ has no loops, i.e., $A_{i,i} = 0$ for all $i \in Q_0$.}
\end{equation}
\begin{equation}
\label{eq:no-cyclic-repeats}
\text{No two cyclically equivalent cyclic paths appear in the
decomposition of~$S$.}
\end{equation}
\end{definition}

In view of \eqref{eq:no-loops}, every potential~$S$ belongs to
${\mathfrak m}(A)^2$; and condition \eqref{eq:no-cyclic-repeats}
excludes, for instance, any non-zero potential~$S$ cyclically
equivalent to~$0$.

\begin{definition}
\label{def:AS-isomorphism}
Let $(A,S)$ and $(A',S')$ be QPs 
on the same vertex set~$Q_0$.
By a \emph{right-equivalence} between $(A,S)$ and $(A',S')$ we mean
an algebra isomorphism
$\varphi: R \langle \langle A \rangle \rangle \to
R \langle \langle A' \rangle \rangle$ such that $\varphi|_R = {\rm id}$, and
$\varphi(S)$ is
cyclically equivalent to~$S'$ (see
Definition~\ref{def:cyclic-equivalence}).
\end{definition}

In view of Proposition~\ref{pr:cyclical-equiv-thru-trace},
any algebra homomorphism
$R \langle \langle A \rangle \rangle \to
R \langle \langle A' \rangle \rangle$ such that $\varphi|_R = {\rm id}$,
sends cyclically equivalent potentials to cyclically
equivalent ones.
It follows that right-equivalences of QPs 
have the expected properties: the composition of two right-equivalences,
as well as the inverse of a right-equivalence, is again a right-equivalence.
Note also that an isomorphism $\varphi: R \langle \langle A \rangle \rangle \to
R \langle \langle A' \rangle \rangle$ induces an isomorphism of
$R$-bimodules $A$ and $A'$ (cf. Proposition~\ref{pr:automorphisms}), so in dealing with
right-equivalent QPs 
we can assume without loss of generality that $A = A'$.

In view of Propositions~\ref{pr:cyclic-equivalence}
and \ref{pr:automorphism-respects-jacobian},
any right-equivalence of QPs 
$(A,S) \cong (A',S')$ induces
an isomorphism of the Jacobian ideals $J(S) \cong J(S')$
and of the Jacobian algebras ${\mathcal P}(A,S) \cong {\mathcal P}(A',S')$.

For every two QPs 
$(A,S)$ and $(A',S')$ (on the same set of vertices~$Q_0$), we can form their \emph{direct sum}
$(A,S) \oplus (A',S') = (A \oplus A',S + S')$; it is well-defined since both complete path
algebras $R\langle\langle A\rangle\rangle$ and $R\langle\langle A'\rangle\rangle$
have canonical embeddings into $R\langle\langle A \oplus
A'\rangle\rangle$ as closed $R$-subalgebras.

We start our analysis of QPs 
with the case $S \in A^2$.
In this case,  $J(S)$ is the closure of the ideal generated by
the subspace
\begin{equation}
\label{eq:partial-S}
\partial S = \{\partial_\xi(S) \mid \xi \in A^\star\} \subseteq A.
\end{equation}

\begin{definition}
\label{def:trivial}
We say that a QP 
$(A,S)$ is \emph{trivial}
if $S \in A^2$, and $\partial S = A$, or, equivalently,
${\mathcal P}(A,S) = R$.
\end{definition}

The following description of trivial QPs 
is seen by standard linear algebra.

\begin{proposition}
\label{pr:trivial-potential}
A QP 
$(A,S)$ with $S \in A^2$ is trivial
if and only if the set of arrows $Q_1$
consists of $2N$ distinct arrows
$a_1, b_1, \dots, a_N, b_N$ such that
each $a_k b_k$ is a cyclic $2$-path, and
there is a change of arrows~$\varphi$
(see Definition~\ref{def:automorphisms}) such that
$\varphi(S)$ is cyclically equivalent to $a_1 b_1 + \cdots + a_N b_N$.
\end{proposition}

Returning to general QPs, 
we now show that taking direct sums with trivial ones
does not affect the Jacobian algebra.

\begin{proposition}
\label{pr:jacobian-algebra-invariant}
If $(A,S)$ is an arbitrary QP, 
and $(C,T)$ is a trivial one, then the canonical embedding
$R\langle\langle A \rangle\rangle \to R\langle\langle A \oplus
C \rangle\rangle$
induces an isomorphism of Jacobian algebras
${\mathcal P}(A, S) \to {\mathcal P}(A \oplus C, S + T)$.
\end{proposition}

\begin{proof}
Let~$L$ denote the closure of the two-sided ideal in
$R \langle \langle A \oplus C \rangle \rangle$ generated by $C$;
thus, $L$ is the set of all (possibly infinite) linear combinations of paths,
each of which contains at least one arrow from $C$.
The definitions readily imply that
$$R \langle \langle A \oplus C \rangle \rangle =
R \langle \langle A \rangle \rangle \oplus L,$$
and
$$J(S+T) = J(S) \oplus L$$
(in the last equality, $J(S)$ is understood as the Jacobian ideal
of~$S$ in $R \langle \langle A \rangle \rangle$).
Therefore,
\begin{align*}
{\mathcal P}(A \oplus C, S + T) &=
R \langle \langle A \oplus C \rangle \rangle/J(S+T) =
(R \langle \langle A \rangle \rangle \oplus L)/ (J(S) \oplus L)\\
& \cong R \langle \langle A \rangle \rangle/ J(S) =
{\mathcal P}(A, S),
\end{align*}
as desired.
\end{proof}

For an arbitrary QP $(A,S)$, we denote by $S^{(2)} \in A^2$ the degree~$2$
homogeneous component of~$S$.
We call $(A,S)$ \emph{reduced}
if $S^{(2)} = 0$, i.e., $S \in {\mathfrak m}(A)^3$.
We define the \emph{trivial} and \emph{reduced} arrow spans of $(A,S)$
as the finite-dimensional $R$-bimodules given by
\begin{equation}
\label{eq:A-triv}
A_{\rm triv} = A_{\rm triv}(S) = \partial S^{(2)}, \quad
A_{\rm red} = A_{\rm red}(S) = A/\partial S^{(2)} \ .
\end{equation}
(see \eqref{eq:partial-S}).


The following statement will play a crucial role in later sections.

\begin{theorem}[Splitting Theorem]
\label{th:trivial-reduced-splitting}
For every QP $(A,S)$ with the trivial arrow span $A_{\rm triv}$
and the reduced arrow span $A_{\rm red}$, there exist a trivial
QP $(A_{\rm triv}, S_{\rm triv})$ and a reduced QP
$(A_{\rm red}, S_{\rm red})$ such that $(A,S)$
is right-equivalent to the direct sum
$(A_{\rm triv}, S_{\rm triv}) \oplus (A_{\rm red}, S_{\rm red})$.
Furthermore, the right-equivalence class of
each of the QPs $(A_{\rm triv}, S_{\rm triv})$ and
$(A_{\rm red}, S_{\rm red})$
is determined by the right-equivalence class of $(A,S)$.
\end{theorem}


Let us first prove the existence of a desired right-equivalence
\begin{equation}
\label{eq:decomposition}
(A,S) \cong (A_{\rm triv}, S_{\rm triv}) \oplus (A_{\rm red}, S_{\rm red}).
\end{equation}
There is nothing to prove if $(A,S)$ is reduced, so let us assume that $S^{(2)} \neq 0$.
Using Proposition~\ref{pr:trivial-potential} and replacing~$S$
if necessary by a cyclically equivalent
potential, we can assume that~$S$ is of the form
\begin{equation}
\label{eq:S-normal-form}
S = \sum_{k=1}^N (a_k b_k + a_k u_k + v_k b_k) + S',
\end{equation}
where each $a_k b_k$ is a cyclic $2$-path, the arrows $a_1, b_1, \dots, a_N, b_N$
form a basis of $A_{\rm triv}$, the elements $u_k$ and $v_k$ belong to ${\mathfrak m}^2$,
and the potential $S' \in {\mathfrak m}^3$
is a linear combination of cyclic paths containing none of the arrows~$a_k$ or~$b_k$.
The existence of a right-equivalence \eqref{eq:decomposition}
becomes a consequence of the following lemma.

\begin{lemma}
\label{lem:killing-u-v}
For every potential~$S$ of the form
\eqref{eq:S-normal-form}, there exists
a unitriangular automorphism $\varphi$
of $R\langle\langle A \rangle\rangle$ such that
$\varphi(S)$ is cyclically equivalent to a
potential of the form \eqref{eq:S-normal-form} with $u_k = v_k = 0$ for all~$k$.
\end{lemma}

We say that a potential~$S$ is \emph{$d$-split} if it is
of the form \eqref{eq:S-normal-form} with $u_k, v_k \in {\mathfrak m}^{d+1}$ for all~$k$.
To prove Lemma~\ref{lem:killing-u-v}, we first show the following.

\begin{lemma}
\label{lem:u-v-d}
Suppose a potential~$S$ is $d$-split for some $d \geq 1$.
There exists a unitriangular automorphism $\varphi$
of $R\langle\langle A \rangle\rangle$ having depth $d$ and such that
$\varphi(S)$ is cyclically equivalent to a $2d$-split
potential $\tilde S$ with $\varphi(S) - \tilde S \ \in {\mathfrak m}^{2d+2}$.
\end{lemma}

\begin{proof}
Let us write~$S$ in the form \eqref{eq:S-normal-form} with
$u_k, v_k \in {\mathfrak m}^{d+1}$.
Let $\varphi$ be the unitriangular automorphism
of $R\langle\langle A \rangle\rangle$
acting on arrows as follows:
$$\varphi(a_k) =a_k -v_k, \,\, \varphi(b_k) = b_k-u_k, \,\, \varphi(c) = c
\quad (c \in Q_1 - \{a_1, b_1, \dots, a_N, b_N\}).$$
Then $\varphi$ is of depth~$d$, so by \eqref{eq:unitriangular-d}, for each~$k$, we have
$$\varphi(u_k) = u_k + u'_k, \,\, \varphi(v_k) = v_k + v'_k \quad
(u'_k, v'_k \in {\mathfrak m}^{2d+1}).$$
Therefore, we obtain
\begin{align*}
\varphi(S) &= \sum_k((a_k-v_k)(b_k-u_k) + (a_k-v_k)(u_k+u'_k) + (v_k+v'_k)(b_k-u_k)) + S'\\
& = \sum_k(a_k b_k + a_k u'_k + v'_k b_k) + S_1  + S',
\end{align*}
where
$$S_1 = -\sum_k (v_k u_k + v_k u'_k + v'_k u_k) \in {\mathfrak m}^{2d+2}.$$
In view of Definition~\ref{def:cyclic-equivalence},
$S_1$ is cyclically equivalent to a potential of the form
$\sum_k(a_k u''_k + v''_k b_k) + S''$, where $u''_k, v''_k \in {\mathfrak m}^{2d+1}$,
and $S''$ is a linear combination of cyclic paths containing none of the~$a_k$ or~$b_k$.
Furthermore, we have
$$S_1 - S'' - \sum_k (a_k u''_k + v''_k b_k) \in {\mathfrak m}^{2d+2}.$$
We see that the desired potential $\tilde S$ can be chosen as
$$\tilde S = \sum_k (a_k b_k + a_k(u'_k+u''_k) + (v'_k+v''_k)b_k) + S' + S'',$$
completing the proof of Lemma~\ref{lem:u-v-d}.
\end{proof}

\begin{proof}[Proof of Lemma~\ref{lem:killing-u-v}]
Starting with a potential~$S$ of the form
\eqref{eq:S-normal-form} and using repeatedly
Lemma~\ref{lem:u-v-d}, we construct a sequence of
potentials~$S_1, S_2, \dots$, and a sequence of unitriangular automorphisms
$\varphi_1, \varphi_2, \dots$, with the following properties:

\begin{enumerate}
\item $S_1 = S$.
\item $S_d$ is $2^{d-1}$-split.
\item $\varphi_d$ is of depth $2^{d-1}$.
\item $\varphi_d (S_d)$ is cyclically equivalent to
$S_{d+1}$, and $\varphi_d(S_d) - S_{d+1} \ \in {\mathfrak m}^{2^d+2}$.
\end{enumerate}

By property (3), setting
\begin{equation}
\label{eq:limit-automorphism}
\varphi = \lim_{n \to \infty} \varphi_n  \varphi_{n-1}
 \cdots  \varphi_1,
\end{equation}
we obtain a well defined unitriangular automorphism $\varphi$ of
$R\langle\langle A \rangle\rangle$; indeed, in view of
\eqref{eq:unitriangular-d}, for any $u \in R\langle\langle A
\rangle\rangle$, if we write $\varphi_n \varphi_{n-1}
 \cdots  \varphi_1(u)$ as $\sum_{d=0}^\infty
u_n^{(d)}$ with $u_n^{(d)} \in A^d$, then each homogeneous
component $u_n^{(d)}$ stabilizes as $n \to \infty$.

We claim that this automorphism $\varphi$ satisfies the required properties
in Lemma~\ref{lem:killing-u-v}.
To see this, for $d \geq 1$, denote
$C_d = \varphi_d(S_d) - S_{d+1}$.
By (4), $C_d \in \{R\langle\langle A \rangle\rangle, R\langle\langle A
\rangle\rangle\} \cap {\mathfrak m}^{2^d+2}$
(recall from Definition~\ref{def:trace-space} that
$\{R\langle\langle A \rangle\rangle, R\langle\langle A \rangle\rangle\}$
denotes the closure of the vector subspace
in~$R\langle\langle A \rangle\rangle$ spanned by all commutators).
Using (1), it is easy to see that
$$\varphi_n  \varphi_{n-1}
 \cdots  \varphi_1(S) = S_{n+1} +
\sum_{d=1}^n \varphi_n  \varphi_{n-1}
 \cdots  \varphi_{d+1}(C_d)$$
for every $n \geq 1$; passing to the limit as $n \to \infty$
yields
$$\varphi(S) = \lim_{n \to \infty} S_n +
\varphi(\sum_{d=1}^\infty (\varphi_d \cdots
\varphi_1)^{-1}(C_d))$$
(the convergence of the series on the right is clear since
any automorphism of $R\langle\langle A \rangle\rangle$ preserves
the powers of ${\mathfrak m}$).
We conclude that $\varphi(S)$ is cyclically equivalent to
$\lim_{n \to \infty} S_n$.
In view of (2), the latter element is
of the form \eqref{eq:S-normal-form} with $u_k = v_k = 0$ for all~$k$.
This completes the proofs of Lemma~\ref{lem:killing-u-v} and
of the existence of a right-equivalence \eqref{eq:decomposition}.
\end{proof}

The above argument makes it clear that the right-equivalence class of $(A_{\rm triv}, S_{\rm triv})$
is determined by the right-equivalence class of $(A,S)$ .
To prove Theorem~\ref{th:trivial-reduced-splitting},
it remains to show that the same is true for $(A_{\rm red}, S_{\rm red})$.
Changing notation a little bit, we need to prove the following.

\begin{proposition}
\label{pr:cancel-trivial}
Let $(A,S)$ and $(A,S')$ be reduced QPs, and $(C,T)$ a trivial QP.
If $(A \oplus C, S+T)$ is right-equivalent to $(A \oplus C, S'+T)$ then
$(A, S)$ is right-equivalent to $(A, S')$.
\end{proposition}

We deduce Proposition~\ref{pr:cancel-trivial} from the following
result of  independent interest.

\begin{proposition}
\label{pr:up-to-square}
Let $(A,S)$ and $(A,S')$ be reduced QPs such that $S'-S \in J(S)^2$.
Then we have:
\begin{enumerate}
\item $J(S') = J(S)$.
\item $(A, S)$ is right-equivalent to $(A, S')$.
More precisely, there exists an algebra automorphism
$\varphi$ of $R \langle \langle A \rangle \rangle$ such that
$\varphi(S)$ is cyclically equivalent to~$S'$, and
$\varphi(u) - u \in J(S)$ for all
$u \in R \langle \langle A \rangle \rangle$.
\end{enumerate}
\end{proposition}

\begin{proof}
(1) Since $(A,S)$ is reduced, we have $J(S) \subseteq {\mathfrak m}^2$.
As an easy consequence of the cyclic Leibniz rule
\eqref{eq:leibniz}, we see that
$$\partial_\xi (J(S)^2)_{\rm cyc}\subseteq
{\mathfrak m}J(S)+J(S){\mathfrak m}$$
for any $\xi \in A^\star$.
It follows that
\begin{equation}
\label{eq:diff-of-partials}
\partial_\xi S'- \partial_\xi S
\in
{\mathfrak m}J(S)+J(S){\mathfrak m},
\end{equation}
implying that $J(S')\subseteq J(S)$.

To show the reverse inclusion, note that \eqref{eq:diff-of-partials}
also implies that
$$J(S) \subseteq J(S') + ({\mathfrak m}J(S)+J(S){\mathfrak m}).$$
Applying the same inclusion to each of the terms $J(S)$ on the
right, we obtain
$$J(S) \subseteq J(S') +
({\mathfrak m}^2 J(S)+ {\mathfrak m}J(S){\mathfrak m}
+ J(S){\mathfrak m}^2).$$
Continuing in the same way, we get
$$J(S) \subseteq J(S') + \sum_{k=0}^n
{\mathfrak m}^k J(S){\mathfrak m}^{n-k}
\subseteq J(S') + {\mathfrak m}^{n+2}$$
for any $n \geq 1$.
Remembering the definition of topology in
$R\langle \langle A \rangle \rangle$ (see \eqref{eq:closure})
and the fact that $J(S')$ is closed,
we conclude that $J(S) \subseteq J(S')$, finishing the proof of
part (1) of Proposition~\ref{pr:up-to-square}.

(2) Let $Q_1 = \{a_1, \dots, a_N\}$ be the set of arrows (that is,
a basis of~$A$).
Then a unitriangular automorphism $\varphi$ of $R\langle \langle A \rangle \rangle$
is specified by a $N$-tuple of elements $b_1, \dots, b_N \in {\mathfrak m}^2$ such that
\begin{equation}
\label{eq:phi-on-arrows}
\varphi(a_k) = a_k + b_k \,\, (k = 1, \dots, N).
\end{equation}

\begin{lemma}
\label{lem:noncommutative-Taylor}
Let $(A,S)$ be a reduced QP, and let $\varphi$ be a unitriangular automorphism
of $R\langle \langle A \rangle \rangle$ given by \eqref{eq:phi-on-arrows}.
Then the potential
$\varphi(S) - S - \sum_{k=1}^N b_k \partial_{a_k}S$
is cyclically equivalent to an element of
${\mathfrak m} I^2$, where $I$
is the closure of the ideal in $R\langle\langle A \rangle\rangle$
generated by $b_1,\dots,b_N$.
\end{lemma}

\begin{proof}
First consider the case where $S = a_{k_1} \cdots a_{k_d}$ is a
cyclic path of length $d \geq 3$.
Then $\varphi(S) = (a_{k_1}+ b_{k_1})\cdots (a_{k_d}+ b_{k_d})$.
Expanding this product, we see that the term that contains no
factors $b_{k_j}$ is equal to~$S$, while the sum of the terms that
contain exactly one factor $b_{k_j}$ is easily seen to be cyclically equivalent to
$\sum_{k=1}^N b_k \partial_{a_k}S$
(cf.\eqref{eq:cyclic-derivative}), and the rest of the terms are
cyclically equivalent to elements of
$\sum_{k=1}^N {\mathfrak m}
({\mathfrak m}^{d-1} \cap I) b_k$.

Writing a general potential $S \in {\mathfrak m}^3$ as a linear
combination of cyclic paths,we see that
$\varphi(S) - S - \sum_{k=1}^N b_k \partial_{a_k}S$
is cyclically equivalent to
$\sum_{k=1}^N c_k b_k$, where each $c_k$ is of the form
$$c_k = \sum_{\ell=1}^N  a_\ell \sum_{d=3}^\infty c_{k \ell}^{(d)}$$
with $c_{k \ell}^{(d)} \in {\mathfrak m}^{d-1} \cap I$.
Since $I$ is closed, each $c_k$ is a well-defined element of ${\mathfrak m}I$, implying the
assertion of Lemma~\ref{lem:noncommutative-Taylor}.
\end{proof}

We will also need one more lemma whose proof will be given in
Section~\ref{sec:topological-appendix}.

\begin{lemma}
\label{lem:tr-IJ}
Let $I$ be a closed ideal of $R\langle\langle
A\rangle\rangle$, and $J$ be the closure of an ideal generated by
finitely many elements $f_1, f_2, \dots, f_N$,
which are bi-homogeneous with respect to the vertex bigrading.
Then every potential belonging to the ideal $IJ$ is cyclically equivalent to
an element of the form
$\sum_{k=1}^N b_k f_k$,
where all $b_k$ belong to $I$.
\end{lemma}

To prove part (2) of
Proposition~\ref{pr:up-to-square},
we construct a sequence of $N$-tuples
$$(b_{1n}, \dots, b_{Nn}) \quad (n \geq 1)$$
of elements of ${\mathfrak m}^2$ and the corresponding unitriangular
automorphisms $\varphi_n$ of $R\langle \langle A \rangle \rangle$
(so that $\varphi_n(a_k) = a_k + b_{kn}$
for $k = 1, \dots, N$)
such that, for all $n \geq 1$, we have

\begin{enumerate}
\item $b_{kn} \in {\mathfrak m}^{n+1} \cap J(S)$
for $k = 1, \dots, N$.
\item $S'$ is cyclically equivalent to
$\varphi_0 \varphi_1 \cdots \varphi_{n-1}
(S + \sum_{k=1}^N b_{kn} \partial_{a_k}S)$
(with the convention that $\varphi_0$ is the identity
automorphism).
\end{enumerate}

We proceed by induction on~$n$.
In the basic case $n=1$, the existence of an $N$-tuple
$(b_{11}, \dots, b_{N1})$ with desired properties
follows from Lemma~\ref{lem:tr-IJ} applied to
$I = J = J(S)$ and $f_k = \partial_{a_k}S$
(note that $J(S) \subseteq {\mathfrak m}^{2}$,
since $(A,S)$ is assumed to be reduced).

Now assume that, for some $n \geq 1$, we have already defined
the elements $b_{k\ell}$ for $k=1, \dots, N$ and $\ell =1, \dots, n$,
satisfying (1) and (2).
Applying Lemma~\ref{lem:noncommutative-Taylor} to $b_k =
b_{kn}$ (so that $\varphi = \varphi_n$), we obtain that
$\varphi_n(S) - (S + \sum_{k=1}^N b_{kn} \partial_{a_k}S)$
is cyclically equivalent to an element of
${\mathfrak m}({\mathfrak m}^{n+1} \cap J(S))^2$.
We have
$${\mathfrak m}({\mathfrak m}^{n+1} \cap J(S))^2
\subseteq ({\mathfrak m}^{n+2} \cap J(S))J(S).$$
This implies in particular that $\varphi_n(S) - S$
is cyclically equivalent to an element of $J(S)^2$.
Combining Proposition~\ref{pr:automorphism-respects-jacobian}
with the already proved part (1) of
Proposition~\ref{pr:up-to-square}, we conclude that
$\varphi_n (J(S))  = J(\varphi_n(S)) = J(S)$.
It follows that $\varphi_n(S) - (S + \sum_{k=1}^N b_{kn} \partial_{a_k}S)$
is cyclically equivalent to an element of
$\varphi_n(({\mathfrak m}^{n+2} \cap J(S))J(S))$.

Applying Lemma~\ref{lem:tr-IJ} to
$I = {\mathfrak m}^{n+2} \cap J(S)$, $J = J(S)$ and $f_k = \partial_{a_k}S$,
we see that every potential in $({\mathfrak m}^{n+2} \cap J(S))J(S)$
is cyclically equivalent to a potential of the form
$$\sum_{k=1}^N b_{k,n+1} \partial_{a_k}S$$
for some $b_{k,n+1} \in {\mathfrak m}^{n+2} \cap J(S)$.
It follows that
$S + \sum_{k=1}^N b_{kn} \partial_{a_k}S$
is cyclically equivalent to
$\varphi_n(S + \sum_{k=1}^N b_{k,n+1} \partial_{a_k}S)$.
Thus, conditions (1) and (2) get satisfied with $n$ replaced by $n+1$,
completing our inductive step.

In view of condition (1),
$\lim_{n \to \infty} \varphi_1 \cdots \varphi_n$ is a
well-defined automorphism~$\varphi$ of $R\langle \langle A \rangle \rangle$
such that $\varphi(u) - u \in J(S)$ for all
$u \in R \langle \langle A \rangle \rangle$.
Passing to the limit $n \to \infty$ in condition (2), we conclude that
$S'$ is cyclically equivalent to $\varphi(S)$, completing the proof
of part (2) of Proposition~\ref{pr:up-to-square}.
\end{proof}

\begin{proof}[Proof of Proposition~\ref{pr:cancel-trivial}]
We abbreviate $J = J(S)$ and $J' = J(S')$
(understood as the Jacobian ideals of $S$ and $S'$ in
$R\langle \langle A \rangle \rangle$).
As in Proposition~\ref{pr:jacobian-algebra-invariant},
let~$L$ denote the closure of the two-sided ideal in
$R \langle \langle A \oplus C \rangle \rangle$ generated by $C$.
Then we have
\begin{equation}
\label{eq:L-splits}
R \langle \langle A \oplus C \rangle \rangle =
R \langle \langle A \rangle \rangle \oplus L, \quad
J(S+T) = J \oplus L, \quad
J(S'+T) = J' \oplus L.
\end{equation}

Let $\varphi$ be an automorphism
of $R\langle \langle A \oplus C\rangle \rangle$,
such that $\varphi(S+T)$ is cyclically equivalent to $S'+T$.
In view of \eqref{eq:L-splits} and
Proposition~\ref{pr:automorphism-respects-jacobian}, we have
\begin{equation}
\label{eq:J'L-to-JL}
\varphi(J \oplus L) = J' \oplus L.
\end{equation}

Let $\psi: R\langle \langle A \rangle \rangle \to R\langle \langle A \rangle \rangle$
denote the restriction to $R\langle \langle A \rangle \rangle$
of the composition $p \varphi$, where $p$ is the projection of
$R\langle \langle A \oplus C \rangle \rangle$ onto $R\langle \langle A \rangle \rangle$ along $L$.
In view of Proposition~\ref{pr:up-to-square}, it suffices to show
the following:
\begin{eqnarray}
\label{eq:psi-does-the-job}
&\text{$\psi$ is an automorphism
of $R\langle \langle A \rangle \rangle$
such that}\\
\nonumber
&\text{$S' - \psi(S)$ is cyclically equivalent to an element of $\psi(J^2)$}
\end{eqnarray}
(indeed, assuming \eqref{eq:psi-does-the-job} and using
Proposition~\ref{pr:automorphism-respects-jacobian}, we see that
$\psi(J^2) = J(\psi(S))^2$, hence one can apply Proposition~\ref{pr:up-to-square}
to potentials $S'$ and $\psi(S)$).

Clearly, $\psi$ is an algebra homomorphism, so can be represented by a pair
$(\psi^{(1)}, \psi^{(2)})$ as in Proposition~\ref{pr:automorphisms}.
To show that $\psi$ is an automorphism of $R\langle \langle A \rangle
\rangle$, it suffices to show that $\psi^{(1)}$ is an $R$-bimodule
automorphism of $A$.
By the definition, if we write the $R$-bimodule
automorphism~$\varphi^{(1)}$ of $A \oplus C$ as a matrix
$$\begin{pmatrix}
\varphi_{AA} & \varphi_{AC} \cr
\varphi_{CA} & \varphi_{CC} \cr
\end{pmatrix},$$
then $\psi^{(1)} = \varphi_{AA}$.
Since
$$\varphi(C) \subset \varphi(J \oplus L) = J' \oplus L
\subseteq {\mathfrak m}(A)^2 \oplus L,$$
it follows that $\varphi_{AC} = 0$, implying that $\psi^{(1)} = \varphi_{AA}$
is an $R$-bimodule automorphism of $A$, and that
$\psi$ is an automorphism of $R\langle \langle A \rangle \rangle$.

Since $S'+T$ is cyclically equivalent to $\varphi(S+T)$, the same is true for the potentials
obtained from them by applying the projection~$p$; it follows
that $S' - \psi(S)$ is cyclically equivalent to $p \varphi(T)$.
Since $T \in C^2$, the claim that $S' - \psi(S)$
is cyclically equivalent to an element of $\psi(J^2)$
follows from the fact that
$p \varphi(L) \subseteq \psi(J)$, or, equivalently, that
$\varphi(L) \subseteq \varphi(J) + L$.
Applying the inverse automorphism $\varphi^{-1}$ to both
sides, it suffices to show that $L \subseteq J + \varphi^{-1}(L)$.
Using the obvious symmetry between $J$ and $J'$, it is enough to
show the inclusion $L \subseteq J' + \varphi(L)$.

Let us abbreviate $M = {\mathfrak m}(A \oplus C)$, and
$I = J' + \varphi(L)$.
Since $\varphi(J) \subseteq J' \oplus L$, and
$J \subseteq {\mathfrak m}(A)^2$, it follows that
$\varphi(J) \subseteq J' \oplus (L \cap M^2) =
J' + ML + LM$.
Therefore, we have
$$L \subseteq J' + L = \varphi(J) + \varphi(L)
\subseteq I + ML + LM.$$
Substituting this upper bound for~$L$ into its right hand side, we
deduce the inclusion
$$L \subseteq I + M^2 L + MLM + L M^2.$$
Continuing in the same way, for every $n > 0$, we
have the inclusion
$$L \subseteq I + \sum_{k=0}^n M^k L M^{n-k} \subseteq I + M^{n+1}.$$
In view of \eqref{eq:closure}, it follows that $L$ is contained
in $\overline I$, the closure of~$I$ in $R\langle \langle A \oplus C \rangle \rangle$.
However, it is easy to see that $I = J' + \varphi(L)$ is closed in
$R\langle \langle A \oplus C\rangle \rangle$
(indeed, the closedness of~$I$ is equivalent to that of
$\varphi^{-1}(I) = \varphi^{-1}(J') + L$, and so,
by symmetry, it is enough to show that
$\varphi(J) + L$ is closed; but this is clear since
$\varphi(J) + L = p^{-1}(\psi(J))$ is the inverse image of the
closed ideal $\psi(J)$ of $R\langle \langle A \rangle \rangle$).
This completes the proofs
of Proposition~\ref{pr:cancel-trivial} and
Theorem~\ref{th:trivial-reduced-splitting}.
\end{proof}

\begin{definition}
\label{def:reduced-part}
We call the component $(A_{\rm red}, S_{\rm red})$ in the
decomposition \eqref{eq:decomposition} the \emph{reduced part} of
a QP $(A,S)$ (by Theorem~\ref{th:trivial-reduced-splitting}, it is
determined by $(A,S)$ up to right-equivalence).
\end{definition}

\begin{definition}
\label{def:2-acyclic}
We call a quiver $Q$ (as well as its arrow span~$A$) \emph{$2$-acyclic}
if it has no oriented $2$-cycles, i.e., satisfies the following
condition:
\begin{equation}
\label{eq:no-2-cycles}
\text{For every pair of vertices $i \neq j$, either $A_{i,j} = \{0\}$ or $A_{j,i} = \{0\}$.}
\end{equation}
\end{definition}

In the rest of this section we study the conditions on a QP
$(A,S)$ guaranteeing that its reduced part is $2$-acyclic.
We need some preparation.

For a quiver~$Q$ with the arrow span~$A$, let $\mathcal C =
{\mathcal C}(A)$ denote the set of cyclic paths on~$A$ up to
cyclical equivalence.
Thus, $\mathcal C$ is either empty (if $Q$ has no oriented cycles at all), or countable.
The space of potentials up to cyclical equivalence is naturally
identified with $K^{\mathcal C}$.
We say that a $K$-valued function on $K^{\mathcal C}$ is
\emph{polynomial} if it depends on finitely many components of a
potential~$S$ and can be expressed as a polynomial in these
components.
For a nonzero polynomial function~$F$, we denote by $U(F) \subset K^{\mathcal C}$
the set of all potentials~$S$ such that $F(S) \neq 0$.
By a regular function on $U(F)$ we mean a ratio of two polynomial
functions on $K^{\mathcal C}$ such that the denominator vanishes
nowhere on $U(F)$; in particular, any function of the form
$G/F^n$, where $G$ is a polynomial, is regular on $U(F)$.
If $A'$ is the arrow span of another quiver $Q'$, we say that a map
$K^{{\mathcal C}(A)} \to K^{{\mathcal C}(A')}$ is \emph{polynomial}
if its every component is a polynomial function; similarly, a map
$U(F) \to K^{{\mathcal C}(A')}$ is \emph{regular}
if its every component is a regular function on $U(F)$.

Now suppose that the arrow span $A$ satisfies \eqref{eq:no-loops},
and let $\{a_1, b_1, \dots, a_N, b_N\}$ be any maximal collection of
distinct arrows in~$Q$ such that $b_k a_k$ is a cyclic $2$-path
for $k = 1, \dots, N$.
Then the quiver obtained from~$Q$ by removing this
collection of arrows is clearly $2$-acyclic.
To such a collection we associate a nonzero polynomial function on $K^{{\mathcal C}(A)}$ given by
\begin{equation}
\label{eq:max-minor}
D_{a_1, \dots, a_N}^{b_1, \dots, b_N}(S) = \det (x_{b_q a_p})_{p,q = 1, \dots, N},
\end{equation}
where $x_{b_q a_p}$ is the sum of the coefficients of $b_q a_p$ and $a_p b_q$
in a potential~$S$, with the convention that $x_{b_q a_p} = 0$
unless $b_q a_p$ is a cyclic $2$-path.

\begin{proposition}
\label{pr:reduced-2-acyclic}
The reduced part $(A_{\rm red},S_{\rm red})$ of a QP $(A,S)$ is
$2$-acyclic if and only if $D_{a_1, \dots, a_N}^{b_1, \dots, b_N}(S) \neq 0$
for some collection of arrows as above.
Furthermore, if $A'$ is the arrow span of the quiver obtained from~$Q$ by
removing all arrows $a_1, b_1, \dots, a_N, b_N$, then there exists a regular map
$H: U(D_{a_1, \dots, a_N}^{b_1, \dots, b_N}) \to K^{{\mathcal C}(A')}$
such that, for any QP $(A,S)$ with $S \in U(D_{a_1, \dots, a_N}^{b_1, \dots, b_N})$, the
reduced part $(A_{\rm red},S_{\rm red})$ is right-equivalent to $(A', H(S))$.
\end{proposition}

The proof of Proposition~\ref{pr:reduced-2-acyclic} follows by
tracing the construction of $(A_{\rm red},S_{\rm red})$
given in the proof of Lemma~\ref{lem:killing-u-v}.
Note that we use the following convention.
If $A$ is $2$-acyclic from the start then the only collection
$\{a_1, b_1, \dots, a_N, b_N\}$ as above is the empty set; in this
case, the function $D_{a_1, \dots, a_N}^{b_1, \dots, b_N}$ is
understood to be equal to~$1$, and $H$ is just the identity mapping.

\section{Mutations of quivers with potentials}
\label{sec:mutations}

Let $(A,S)$ be a QP.
Suppose that a vertex $k \in Q_0$ does not belong to an oriented $2$-cycle.
In other words,~$k$ satisfies the following condition:
\begin{equation}
\label{eq:no-2-cycles-thru-k}
\text{For every vertex $i$, either $A_{i,k}$ or $A_{k,i}$ is zero.}
\end{equation}
Replacing $S$ if necessary with a cyclically equivalent
potential, we can also assume that
\begin{equation}
\label{eq:no-start-in-k}
\text{No cyclic path occurring in the expansion of~$S$
starts (and ends) at~$k$.}
\end{equation}
Under these conditions, we associate to $(A,S)$ a QP
$\widetilde \mu_k(A,S) = (\widetilde A, \widetilde S)$ on the
same set of vertices~$Q_0$.
We define the homogeneous components $\widetilde A_{i,j}$ as follows:
\begin{equation}
\label{eq:mu-k-A}
\widetilde A_{i,j} =
\begin{cases}
(A_{j,i})^\star & \text{if $i=k$ or $j=k$;} \\[.05in]
A_{i,j} \oplus A_{i,k} A_{k,j}
 & \text{otherwise;}
\end{cases}
\end{equation}
here the product $A_{i,k} A_{k,j}$ is understood as a subspace of
$A^2 \subseteq R\langle \langle A \rangle \rangle$.
Thus, the $R$-bimodule $\widetilde A$ is given by
\begin{equation}
\label{eq:tilde-A}
\widetilde A= \overline{e}_k A \overline{e}_k \oplus  A e_k A \oplus
(e_k A)^\star \oplus (A e_k)^\star,
\end{equation}
where we use the notation
\begin{equation}
\label{eq:overline-ek}
\overline{e}_k = 1 - e_k = \sum_{i \in Q_0 - \{k\}} e_i.
\end{equation}

We associate to $Q_1$
the set of arrows $\widetilde Q_1$
in the following way:
\begin{itemize}
\item Take all the arrows $c \in Q_1$ not incident to~$k$.
\item For each incoming arrow $a$ and outgoing arrow $b$ at $k$,
create a ``composite" arrow $[ba]$ corresponding to
the product $ba \in  \in A e_k A$.
\item Replace each incoming arrow $a$ (resp.~each outgoing arrow $b$) at~$k$ by
the corresponding arrow $a^\star$ (resp.~$b^\star$) oriented in the opposite way.
\end{itemize}
More formally, for $i=k$ or $j=k$, we set
\begin{equation}
\label{eq:arrows-thru-k-reversed}
\widetilde Q_1 \cap \widetilde A_{i,j} = \{a^\star \mid a \in Q_1
\cap A_{j,i}\}
\end{equation}
(the dual basis);
and for $i$ and $j$ different from~$k$, we define
\begin{equation}
\label{eq:arrows-not-thru-k}
\widetilde Q_1 \cap \widetilde A_{i,j} = (Q_1
\cap A_{i,j})\ \bigsqcup \ \{[ba] \mid b \in Q_1 \cap A_{i,k}, \,\,
a \in Q_1 \cap A_{k,j}\},
\end{equation}
where $[ba] \in \widetilde Q_1 \cap A_{i,k} A_{k,j}$ denotes the arrow
in $\widetilde Q_1$ associated with the product $ba$.

We now associate to $S$ the potential
$\widetilde \mu_k(S) = \widetilde S \in R\langle \langle \widetilde A \rangle \rangle$
given by
\begin{equation}
\label{eq:mu-k-S}
\widetilde S = [S] + \Delta_k,
\end{equation}
where
\begin{equation}
\label{eq:Tk}
\Delta_k = \Delta_k(A) = \sum_{a, b \in Q_1: \ h(a) = t(b) = k} [ba]a^\star b^\star,
\end{equation}
and $[S]$ is obtained by substituting $[a_p a_{p+1}]$ for each
factor $a_p a_{p+1}$ with $t(a_p) = h(a_{p+1}) = k$ of any cyclic path
$a_1 \cdots a_d$ occurring in the expansion of~$S$ (recall that none of these
cyclic paths starts at $k$).
It is easy to see that both $[S]$ and $\Delta_k$ do not depend on the
choice of a basis $Q_1$ of $A$.

The following proposition is immediate from the definitions.

\begin{proposition}
\label{pr:mu-tilde-trivial-summand}
Suppose a QP $(A,S)$ satisfies \eqref{eq:no-2-cycles-thru-k}
and \eqref{eq:no-start-in-k}, and a QP $(A',S')$ is such that
$e_k A' = A' e_k = \{0\}$.
Then we have
\begin{equation}
\widetilde \mu_k(A \oplus A',S+S') = \widetilde \mu_k(A,S) \oplus
(A',S').
\end{equation}
\end{proposition}

\begin{theorem}
\label{th:mu-tilde-preserves-isomorphisms}
The right-equivalence class of the QP $(\widetilde A, \widetilde S) = \widetilde \mu_k(A,S)$
is determined by the right-equivalence class of $(A,S)$.
\end{theorem}

\begin{proof}
Let $\widehat A$ be the finite-dimensional $R$-bimodule given by
\begin{equation}
\label{eq:hat-A}
\widehat A= A \oplus (e_k A)^\star \oplus (A e_k)^\star.
\end{equation}
The natural embedding $A \to \widehat A$ identifies
$R\langle\langle A\rangle\rangle$ with a closed subalgebra
in $R\langle\langle \widehat A\rangle\rangle$.
We also have a natural embedding $\widetilde A \to
R\langle\langle \widehat A\rangle\rangle$ (sending each arrow $[ba]$ to the product $ba$).
This allows us to identify
$R\langle\langle \widetilde A\rangle\rangle$ with another closed
subalgebra in $R\langle\langle \widehat A\rangle\rangle$, namely, with
the closure of the linear span of the paths
$\widehat a_1 \cdots \widehat a_d$ such that $\widehat a_1 \notin e_k A$
and $\widehat a_d \notin A e_k$.
Under this identification, the potential $\widetilde S$ given by
\eqref{eq:mu-k-S} and viewed as an element of
$R\langle\langle \widehat A\rangle\rangle$ is cyclically equivalent to
the potential
$$S + (\sum_{b \in Q_1 \cap A e_k}  b^\star b)
(\sum_{a \in Q_1 \cap  e_k A} a a^\star).$$

Taking this into account, we see that Theorem~\ref{th:mu-tilde-preserves-isomorphisms}
becomes a consequence of the following lemma.

\begin{lemma}
\label{lem:extension}
Every automorphism $\varphi$ of $R\langle\langle A\rangle\rangle$
can be extended to an automorphism $\widehat\varphi$ of
$R\langle\langle \widehat A \rangle\rangle$
satisfying
\begin{equation}
\label{eq:preserving-tildeA}
{\widehat \varphi}(R\langle\langle \widetilde A \rangle\rangle) =
R\langle\langle \widetilde A \rangle\rangle,
\end{equation}
and
\begin{equation}
\label{eq:preserving-quadratic-terms}
{\widehat \varphi}(\sum_{a \in Q_1 \cap  e_k A} a a^\star) =
\sum_{a \in Q_1 \cap  e_k A} a a^\star, \quad
{\widehat \varphi}(\sum_{b \in Q_1 \cap A e_k}  b^\star b) =
\sum_{b \in Q_1 \cap A e_k}  b^\star b.
\end{equation}
\end{lemma}

In order to extend $\varphi$ to an automorphism $\widehat\varphi$ of
$R\langle\langle \widehat A \rangle\rangle$, we need only to define
the elements $\widehat \varphi(a^\star)$ and $\widehat \varphi(b^\star)$ for all
arrows $a \in Q_1 \cap  e_k A$ and $b \in Q_1 \cap A e_k$.

We first deal with $\widehat \varphi(a^\star)$.
Let $Q_1 \cap e_k A = \{a_1, \dots, a_s\}$.
In view of Proposition~\ref{pr:automorphisms},
the action of~$\varphi$ on these arrows is given by
\begin{equation}
\label{eq:phi-on-a}
\begin{pmatrix}
\varphi(a_1) & \varphi(a_2) & \cdots & \varphi(a_s)\end{pmatrix}
=
\begin{pmatrix}
a_1 & a_2 & \cdots & a_s\end{pmatrix}
(C_0 + C_1),
\end{equation}
where:
\begin{itemize}
\item $C_0$ is an invertible $s \times s$ matrix with entries
in~$K$ such that its $(p,q)$-entry is~$0$ unless $t(a_p) = t(a_q)$;
\item $C_1$ is a $s \times s$ matrix whose $(p,q)$-entry belongs
to ${\mathfrak m}(A)_{t(a_p),t(a_q)}$.
\end{itemize}
Note that $C_0 + C_1$ is invertible, and its inverse is of the
same form: indeed, we have
$$(C_0 + C_1)^{-1} = (I + C_0^{-1} C_1)^{-1} C_0^{-1} =
(I + \sum_{n=1}^\infty (-1)^n (C_0^{-1} C_1)^n) C_0^{-1}.$$
Now we define the elements ${\widehat\varphi}(a_p^\star)$ by setting
$$
\begin{pmatrix}
{\widehat \varphi}(a_1^\star)\\
{\widehat\varphi}(a_2^\star)\\
\vdots\\
{\widehat\varphi}(a_s^\star)
\end{pmatrix}
=(C_0 + C_1)^{-1}\begin{pmatrix}
a_1^\star\\
a_2^\star\\
\vdots\\
a_s^\star
\end{pmatrix}.
$$
It follows that
\begin{align*}
{\widehat\varphi}(\sum_p a_p a_p^\star) &=
\begin{pmatrix}
{\widehat\varphi}(a_1) & {\widehat\varphi}(a_2) & \cdots & {\widehat\varphi}(a_s)\end{pmatrix}
\begin{pmatrix}
{\widehat\varphi}(a_1^\star)\\
{\widehat\varphi}(a_2^\star)\\
\vdots\\
{\widehat\varphi}(a_s^\star)
\end{pmatrix}\\
&=
\begin{pmatrix}
a_1 & a_2 & \cdots & a_s\end{pmatrix}
\begin{pmatrix}
a_1^\star\\
a_2^\star\\
\vdots\\
a_s^\star
\end{pmatrix}
=\sum_p a_p a_p^\star.
\end{align*}

For $b \in Q_1 \cap A e_k$, we define $\widehat \varphi(b^\star)$ in a
similar way.
Namely, let
$Q_1 \cap A e_k = \{b_1, \dots, b_t\}$.
As above, the action of~$\varphi$ on these arrows is given by
\begin{equation}
\label{eq:phi-on-b}
\begin{pmatrix}
\varphi(b_1)\\ \varphi(b_2)\\ \vdots \\ \varphi(b_t)\end{pmatrix}
=
(D_0 + D_1)\begin{pmatrix}
b_1 \\ b_2 \\ \vdots \\ b_t\end{pmatrix}
,
\end{equation}
where:
\begin{itemize}
\item $D_0$ is an invertible $t \times t$ matrix with entries
in~$K$ such that its $(p,q)$-entry is~$0$ unless $h(b_p) = h(b_q)$;
\item $D_1$ is a $t \times t$ matrix whose $(p,q)$-entry belongs
to ${\mathfrak m}(A)_{h(b_p),h(b_q)}$.
\end{itemize}
As above, we see that $D_0 + D_1$ is invertible, and its inverse is of the
same form.
Now we define the elements ${\widehat\varphi}(b_q^\star)$ by setting
$$
\begin{pmatrix}
{\widehat \varphi}(b_1^\star) &
{\widehat\varphi}(b_2^\star) &
\cdots &
{\widehat\varphi}(b_t^\star)
\end{pmatrix}
=\begin{pmatrix}
b_1^\star &
b_2^\star &
\cdots &
b_t^\star
\end{pmatrix} (D_0 + D_1)^{-1}.
$$
It follows that
\begin{align*}
{\widehat\varphi}(\sum_q b_q^\star b_q) &=
\begin{pmatrix}
{\widehat\varphi}(b_1)^\star & {\widehat\varphi}(b_2^\star) & \cdots & {\widehat\varphi}(b_t^\star)\end{pmatrix}
\begin{pmatrix}
{\widehat\varphi}(b_1)\\
{\widehat\varphi}(b_2)\\
\vdots\\
{\widehat\varphi}(b_t)
\end{pmatrix}\\
&=
\begin{pmatrix}
b_1^\star & b_2^\star & \cdots & b_t^\star\end{pmatrix}
\begin{pmatrix}
b_1\\
b_2\\
\vdots\\
b_t
\end{pmatrix}
=\sum_q b_q^\star b_q.
\end{align*}

The condition \eqref{eq:preserving-quadratic-terms} is then clearly
satisfied; the construction also makes clear that the automorphism $\widehat \varphi$
of $R\langle\langle \widehat A \rangle\rangle$ preserves
the subalgebra $R\langle\langle \widetilde A \rangle\rangle$.
As a consequence of Proposition~\ref{pr:automorphisms},
$\widehat \varphi$  restricts to an automorphism of
$R\langle\langle \widetilde A \rangle\rangle$, verifying
\eqref{eq:preserving-tildeA} and completing the proofs of
Lemma~\ref{lem:extension} and Theorem~\ref{th:mu-tilde-preserves-isomorphisms}.
\end{proof}

Note that even if a QP $(A,S)$ is assumed to be reduced,
the QP $\widetilde \mu_k(A,S) = (\widetilde A, \widetilde S)$ is
not necessarily reduced because the component
$[S]^{(2)} \in \widetilde A^2$ may be non-zero.
Combining Theorems~\ref{th:trivial-reduced-splitting} and
\ref{th:mu-tilde-preserves-isomorphisms}, we obtain the following
corollary.

\begin{corollary}
\label{cor:mutations-respect-isom}
Suppose a QP $(A,S)$ satisfies \eqref{eq:no-2-cycles-thru-k}
and \eqref{eq:no-start-in-k}, and let
$\widetilde \mu_k(A,S) = (\widetilde A, \widetilde S)$.
Let $(\overline A, \overline S)$ be a reduced QP such that
\begin{equation}
\label{eq:mutilde-mu}
(\widetilde A, \widetilde S) \cong (\widetilde A_{\rm triv}, \widetilde S^{(2)}) \oplus
(\overline A, \overline S)
\end{equation}
(see \eqref{eq:decomposition}).
Then the right-equivalence class of $(\overline A, \overline S)$ is determined by the
right-equivalence class of $(A,S)$.
\end{corollary}

\begin{definition}
\label{def:reduced-mutation}
In the situation of Corollary~\ref{cor:mutations-respect-isom},
we use the notation $\mu_k(A,S) = (\overline A, \overline S)$
and call the correspondence $(A,S) \mapsto \mu_k(A,S)$ the \emph{mutation at vertex}~$k$.
\end{definition}

Note that if a QP $(A,S)$ satisfies \eqref{eq:no-2-cycles-thru-k}
then the same is true for $\widetilde \mu_k(A,S)$ and for
$\mu_k(A,S)$.
Thus, the mutation $\mu_k$ is a well-defined transformation on the
set of right-equivalence classes of reduced QPs satisfying
\eqref{eq:no-2-cycles-thru-k}.
(With some abuse of notation, we sometimes denote
a right-equivalence class by the same symbol as any of its
representatives.)

\begin{example}
Consider the quiver $Q$ with vertices $\{1,2,3,4\}$ and arrows
$a:1\to 2$, $b:2\to 3$, $c:3\to 4$ and $d:4\to 1$:
$$
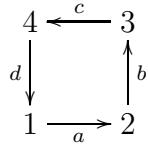
\xymatrix{
4\ar[d]_d & 3\ar[l]_c\\
1\ar[r]_a & 2\ar[u]_b
}
$$
Let $S = dcba$.
Let us perform the mutation at vertex $2$.
The arrow  $a$ is replaced by $e:=a^\star:2\to 1$,
and $b$ is replaced by $f:=b^\star:3\to 2$.
We also have a new arrow
$g := [ba]:1\to 3$.
So $\widetilde \mu_2(A)$ corresponds to the quiver with vertices $\{1,2,3,4\}$ and
arrows $c,d,e,f,g$:
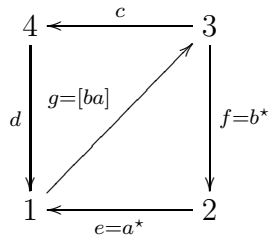
$$
\xymatrix{
4\ar[dd]_d & & 3\ar[ll]_c\ar[dd]^{f=b^\star} \\
\\
1\ar[uurr]^{g=[ba]} &  & 2\ar[ll]^{e=a^\star}
}
$$
The potential $\widetilde \mu_2(S) = \widetilde S$ is given by
$$\widetilde S= dcg + gef;$$
thus, $\widetilde \mu_2(A,S)$ is reduced, and we have $\widetilde \mu_2(A,S) = \mu_2(A,S)$.

Note that $\widetilde S$ does not satisfy condition
\eqref{eq:no-start-in-k} with respect to vertex $k = 3$
since the path $gef$ starts and ends at~$3$.
But we can fix this condition by replacing $\widetilde S$ with a cyclically equivalent
potential, say $S' = dcg + efg$.
Now let us mutate $(\widetilde A, S')$ at vertex $3$.
The arrows $c,f,g$ are replaced by
$c^\star:4\to 3$, $f^\star:2\to 3$ and $g^\star:3 \to 1$, respectively.
We also add new arrows $[cg]:1\to 4$ and $[fg]:1\to 2$.
Thus, $\widetilde \mu_3(\widetilde A,S')$ has arrows
$\{d,e,c^\star,f^\star,g^\star,[cg],[fg]\}$:
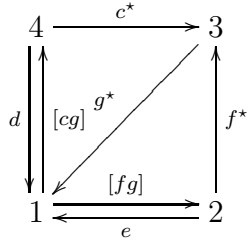
$$
\xymatrix{
4\ar[rr]^{c^\star}\ar@<-.5ex>[dd]_d & & 3\ar[lldd]_{g^\star}\\ \\
1\ar@<-.5ex>[uu]_{[cg]}\ar@<.5ex>[rr]^{[fg]} & &
2\ar[uu]_{f^\star}\ar@<.5ex>[ll]^{e}
}
$$
The potential $\widetilde \mu_3(S')$ is given by
$$
\mu_3 (S') = d[cg] + e[fg] + [fg]g^\star f^\star
+ [cg]g^\star c^\star.
$$
It is not reduced, so to obtain the reduced QP
$\mu_3(\widetilde A,S')$, we need to remove the trivial part
of $\widetilde \mu_3(\widetilde A,S')$.
The resulting quiver is as follows:
$$
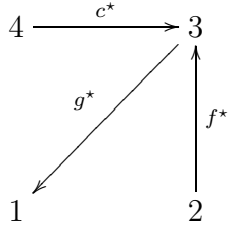
\xymatrix{
4\ar[rr]^{c^\star}
& & 3\ar[lldd]_{g^\star}\\ \\
1
& &
2\ar[uu]_{f^\star}
}
$$
Since it is acyclic (that is, has no oriented cycles),
the corresponding potential is~$0$.
\end{example}

Our next result is that every mutation is an involution.

\begin{theorem}
\label{th:mutation-involutive}
The correspondence $\mu_k: (A,S) \to (\overline A, \overline S)$
acts as an involution on the set of right-equivalence classes of reduced QPs
satisfying \eqref{eq:no-2-cycles-thru-k}, that is,
$\mu_k^2(A,S)$ is right-equivalent to $(A,S)$.
\end{theorem}

\begin{proof}
Let $(A, S)$ be a reduced QP satisfying \eqref{eq:no-2-cycles-thru-k}
and \eqref{eq:no-start-in-k}.
Let ${\widetilde \mu}_k (A,S)= ({\widetilde A}, {\widetilde S})$ and
${\widetilde \mu}_k^2 (A, S)= {\widetilde \mu}_k({\widetilde A}, {\widetilde S}) =
(\widetilde {\widetilde A}, \widetilde {\widetilde S})$.
In view of Theorem~\ref{th:trivial-reduced-splitting} and
Proposition~\ref{pr:mu-tilde-trivial-summand}, it is enough to show that
\begin{equation}
\label{eq:tilde-twice}
\text{$(\widetilde {\widetilde A}, \widetilde {\widetilde S})$ is
right-equivalent to $(A,S) \oplus (C,T)$, where $(C,T)$ is a trivial QP.}
\end{equation}
Using \eqref{eq:tilde-A} twice, and identifying $(e_k A)^\star$ with
$A^\star e_k$, and $(A e_k )^\star$ with
$e_k A^\star$, where $A^\star$ is the dual $R$-bimodule of $A$,
we conclude that
\begin{equation}
\label{eq:tilde-tilde-A}
\widetilde {\widetilde A} =
A \oplus  A e_k A \oplus
A^\star e_k A^\star.
\end{equation}
Furthermore, the basis of arrows in
$\widetilde {\widetilde A}$ consists of the original set of arrows $Q_1$ in $A$
together with the arrows $[ba] \in  A e_k A$ and
$[a^\star b^\star] \in A^\star e_k A^\star$
for $a \in Q_1 \cap  e_k A$ and $b \in Q_1 \cap A e_k$.

Remembering \eqref{eq:mu-k-S} and \eqref{eq:Tk}, we see that the
potential $\widetilde {\widetilde S}$ is given by
\begin{equation}
\label{eq:tilde-tilde-S}
\widetilde {\widetilde S} = [[S]] + [\Delta_k (A)]+\Delta_k ({\widetilde A})
= [S] + \sum_{a, b \in Q_1: \ h(a) = t(b) = k} ([ba][a^\star b^\star]
+ [a^\star b^\star] ba),
\end{equation}
hence it is cyclically equivalent to
\begin{equation}
\label{eq:S1}
S_1= [S] +  \sum_{a, b \in Q_1: \ h(a) = t(b) = k}
([ba] + ba)[a^\star b^\star]
\end{equation}
(recall that $[S]$ is obtained by substituting $[a_p a_{p+1}]$ for each
factor $a_p a_{p+1}$ with $t(a_p) = h(a_{p+1}) = k$ of any cyclic path
$a_1 \cdots a_d$ occurring in the path expansion of~$S$).
Let us abbreviate
$$(C,T) = (A e_k A \oplus A^\star e_k A^\star,
\sum_{a, b \in Q_1: \ h(a) = t(b) = k} [ba][a^\star b^\star]).$$
This is a trivial QP (cf. Proposition~\ref{pr:trivial-potential});
therefore to prove Theorem~\ref{th:mutation-involutive}
it suffices to show that
the QP $(\widetilde {\widetilde A}, S_1)$ given by
\eqref{eq:tilde-tilde-A} and \eqref{eq:S1} is right-equivalent to
$(A,S) \oplus (C,T)$.
We proceed in several steps.

{\bf Step 1:} Let $\varphi_1$ be the change of arrows
automorphism of
$R\langle\langle \widetilde {\widetilde A} \rangle\rangle$
(see Definition~\ref{def:automorphisms}) multiplying
each arrow $b \in Q_1 \cap A e_k$ by $-1$, and fixing the rest of
the arrows in $\widetilde {\widetilde A}$.
Then the potential $S_2 = \varphi_1(S_1)$ is given by
$$S_2= [S] +  \sum_{a, b \in Q_1: \ h(a) = t(b) = k}
([ba] - ba)[a^\star b^\star].$$

{\bf Step 2:}
Let $\varphi_2$ be the unitriangular automorphism of
$R\langle\langle \widetilde {\widetilde A} \rangle\rangle$
(see Definition~\ref{def:automorphisms})
sending each arrow $[ba] \in  A e_k A$ to $[ba] + ba$,
and fixing the rest of
the arrows in $\widetilde {\widetilde A}$.
Remembering the definition of $[S]$, it is easy to see that
the potential $\varphi_2(S_2)$ is cyclically equivalent
to a potential of the form
$$S_3 = S +  \sum_{a, b \in Q_1: \ h(a) = t(b) = k}
[ba]([a^\star b^\star] + f(a,b))$$
for some elements $f(a,b) \in {\mathfrak m}(A \oplus  A e_k A)^2$.

{\bf Step 3:}
Let $\varphi_3$ be the unitriangular automorphism of
$R\langle\langle \widetilde {\widetilde A} \rangle\rangle$
sending each arrow $[a^\star b^\star] \in A^\star e_k A^\star$ to
$[a^\star b^\star] - f(a,b)$,
and fixing the rest of
the arrows in $\widetilde {\widetilde A}$.
Then we have $\varphi_3(S_3) = S+T$.

Combining these three steps, we conclude that
the QP $(\widetilde {\widetilde A}, S_1)$
is right-equivalent to
$(\widetilde {\widetilde A}, S+T) = (A,S) \oplus (C,T)$, finishing
the proof of Theorem~\ref{th:mutation-involutive}.
\end{proof}

\section{Some mutation invariants}
\label{sec:mut-invariants}

In this section we fix a vertex~$k$ and study the effect of
the mutation~$\mu_k$ on the Jacobian algebra ${\mathcal P}(A,S)$.
We will use the following notation: for an $R$-bimodule $B$, denote
\begin{equation}
\label{eq:k-excluded}
B_{\hat k, \hat k} = \overline e_k B \overline e_k =
\bigoplus_{i,j \neq k} B_{i,j}
\end{equation}
(see \eqref{eq:overline-ek}).
Note that if $B$ is a (topological) algebra then
$B_{\hat k, \hat k}$ is a (closed) subalgebra of~$B$.

\begin{proposition}
\label{pr:k-excluded-mutation-invariant}
Suppose a QP $(A,S)$ satisfies \eqref{eq:no-2-cycles-thru-k}
and \eqref{eq:no-start-in-k}, and let
$(\widetilde A, \widetilde S) = \widetilde \mu_k(A,S)$ be given by
\eqref{eq:tilde-A} and \eqref{eq:mu-k-S}.
Then the algebras ${\mathcal P}(A,S)_{\hat k, \hat k}$
and ${\mathcal P}(\widetilde A,\widetilde S)_{\hat k, \hat k}$
are isomorphic to each other.
\end{proposition}

\begin{proof}
In view of \eqref{eq:tilde-A}, we have
\begin{equation}
\label{eq:tilde-A-k-excluded}
{\widetilde A}_{\hat k, \hat k} = A_{\hat k, \hat k}  \oplus  A e_k A.
\end{equation}
Thus, the algebra $R\langle \langle \widetilde A_{\hat k, \hat k}\rangle \rangle$
is generated by the arrows $c \in Q_1 \cap A_{\hat k, \hat k}$ and $[ba]$ for
$a \in Q_1 \cap e_k A$ and $b \in Q_1 \cap A e_k$.
The following fact is immediate from the definitions.

\begin{lemma}
\label{lem:[]-isomorphism}
The correspondence sending each $c \in Q_1 \cap A_{\hat k, \hat k}$ to itself,
and each generator $[ba]$ to $ba$ extends to an algebra isomorphism
$$R\langle \langle {\widetilde A}_{\hat k, \hat k} \rangle \rangle
\to {R\langle \langle A \rangle \rangle}_{\hat k, \hat k}.$$
\end{lemma}

Let $u \mapsto [u]$ denote the isomorphism
${R\langle \langle A \rangle \rangle}_{\hat k, \hat k}
\to R\langle \langle \widetilde A_{\hat k, \hat k}\rangle \rangle$
inverse of that in Lemma~\ref{lem:[]-isomorphism}.
It acts in the same way as
the correspondence $S \mapsto [S]$ in \eqref{eq:mu-k-S}:
$[u]$ is obtained by substituting $[a_p a_{p+1}]$ for each
factor $a_p a_{p+1}$ with $t(a_p) = h(a_{p+1}) = k$ of any path
$a_1 \cdots a_d$ occurring in the path expansion of~$u$.

\begin{lemma}
\label{lem:[]-Jacobian-epimorphism}
The correspondence $u \mapsto [u]$ induces an algebra epimorphism
\hfill \break
${\mathcal P}(A,S)_{\hat k, \hat k} \to
{\mathcal P}(\widetilde A,\widetilde S)_{\hat k, \hat k}$.
\end{lemma}

\begin{proof}
It is enough to prove the following two facts:
\begin{equation}
\label{eq:J-complements-k-excluded}
R\langle \langle {\widetilde A}\rangle \rangle_{\hat k, \hat k}
= R\langle \langle {\widetilde A}_{\hat k, \hat k} \rangle \rangle
+ J(\widetilde S)_{\hat k, \hat k};
\end{equation}
\begin{equation}
\label{eq:J-k-excluded-intersection}
[J(S)_{\hat k, \hat k}] \subseteq
R\langle \langle {\widetilde A}_{\hat k, \hat k} \rangle \rangle
\cap  J(\widetilde S)_{\hat k, \hat k}.
\end{equation}

To show \eqref{eq:J-complements-k-excluded}, we note that if a
path $\widetilde a_1 \cdots \widetilde a_d \in
R\langle \langle {\widetilde A}\rangle \rangle_{\hat k, \hat k}$
does not belong to
$R\langle \langle {\widetilde A}_{\hat k, \hat k} \rangle \rangle$
then it must contain one or more factors of the form $a^\star
b^\star$ with $h(a) = t(b) = k$.
In view of \eqref{eq:mu-k-S} and \eqref{eq:Tk}, we have
\begin{equation}
\label{eq:partial-ba}
a^\star b^\star =
\partial_{[ba]}\widetilde S - \partial_{[ba]}[S].
\end{equation}
Substituting this expression for each factor
$a^\star b^\star$, we see that
$\widetilde a_1 \cdots \widetilde a_d \in
R\langle \langle {\widetilde A}_{\hat k, \hat k} \rangle \rangle
+ J(\widetilde S)_{\hat k, \hat k}$, as desired.

To show \eqref{eq:J-k-excluded-intersection}, we note that
$J(S)_{\hat k, \hat k}$ is easily seen to be the closure of the ideal in
${R\langle \langle A \rangle \rangle}_{\hat k, \hat k}$ generated
by the elements $\partial_c S$ for all arrows $c \in Q_1$ with
$t(c) \neq k$ and $h(c) \neq k$, together with the elements
$(\partial_{a} S) a'$ for $a, a' \in Q_1 \cap e_k A$ , and
$b'(\partial_{b}S)$ for $b, b' \in Q_1 \cap A e_k $.
Let us apply the map $u \mapsto [u]$ to these generators.
First, we have:
\begin{equation}
\label{eq:[]-partial-c}
[\partial_c S] = \partial_c \widetilde S.
\end{equation}
With a little bit more work (using \eqref{eq:partial-ba}), we obtain
\begin{align}
\nonumber
[(\partial_{a}S)a'] &= \sum_{t(b) = k} (\partial_{[ba]}[S])[ba']\\
\label{eq:[]-partial-a}
&=  \sum_{t(b) = k} (\partial_{[ba]}\widetilde S - a^\star b^\star)[ba']\\
\nonumber
&= \sum_{t(b) = k} (\partial_{[ba]}\widetilde S)[ba'] -
a^\star \partial_{{a'}^\star}\widetilde S,
\end{align}
and
\begin{align}
\nonumber
[b'(\partial_{b}S)] &= \sum_{h(a) = k}[b'a] (\partial_{[ba]}[S])\\
\label{eq:[]-partial-b}
&=  \sum_{h(a) = k}[b'a] (\partial_{[ba]}\widetilde S - a^\star b^\star)\\
\nonumber
&= \sum_{h(a) = k}[b'a] (\partial_{[ba]}\widetilde S) -
(\partial_{{b'}^\star}\widetilde S) b^\star.
\end{align}
This implies the desired inclusion in
\eqref{eq:J-k-excluded-intersection}.
\end{proof}

To finish the proof of
Proposition~\ref{pr:k-excluded-mutation-invariant},
it is enough to show that the epimorphism in
Lemma~\ref{lem:[]-Jacobian-epimorphism} (let us denote it by $\alpha$)
is in fact an isomorphism.
To do this, we construct the left inverse algebra homomorphism
$\beta: {\mathcal P}(\widetilde A,\widetilde S)_{\hat k, \hat k} \to
{\mathcal P}(A,S)_{\hat k, \hat k}$ (so that $\beta \alpha$ is the identity
map on ${\mathcal P}(A,S)_{\hat k, \hat k}$).
We define $\beta$ as the composition of three maps.
First, we apply the epimorphism ${\mathcal P}(\widetilde A,\widetilde S)_{\hat k, \hat k}
\to {\mathcal P}(\widetilde{\widetilde A},\widetilde{\widetilde S})_{\hat k, \hat k}$
defined in the same way as $\alpha$.
Remembering the proof of Theorem~\ref{th:mutation-involutive} and
using the notation introduced there, we then apply the isomorphism
${\mathcal P}(\widetilde{\widetilde A},\widetilde{\widetilde S})_{\hat k, \hat k}
\to {\mathcal P}(A \oplus C, S + T)_{\hat k, \hat k}$ induced by
the automorphism $\varphi_3 \varphi_2 \varphi_1$ of
$R\langle \langle A \oplus C \rangle \rangle$.
Finally, we apply the isomorphism
${\mathcal P}(A \oplus C, S + T)_{\hat k, \hat k}
\to {\mathcal P}(A, S)_{\hat k, \hat k}$
given in Proposition~\ref{pr:jacobian-algebra-invariant}.

Since all the maps involved are algebra homomorphisms, it is
enough to check that $\beta \alpha$ fixes the generators
$p(c)$ and $p(ba)$ of ${\mathcal P}(A,S)_{\hat k, \hat k}$,
where  $p$ is the projection
$R\langle \langle A \rangle \rangle \to {\mathcal P}(A,S)$, and
$a, b, c$ have the same meaning as above.
This is done by direct tracing of the definitions.
\end{proof}

\begin{proposition}
\label{pr:fin-dim}
In the situation of
Proposition~\ref{pr:k-excluded-mutation-invariant},
if the Jacobian algebra ${\mathcal P}(A,S)$ is finite-dimensional
then so is ${\mathcal P}(\widetilde A,\widetilde S)$.
\end{proposition}

\begin{proof}
We start by showing that finite dimensionality of
${\mathcal P}(A,S)$ follows from a seemingly weaker
condition.

\begin{lemma}
\label{lem:fin-dim-reduction}
Let $J \subseteq {\mathfrak m}(A)$ be a closed ideal
in $R\langle \langle A \rangle \rangle$.
Then the quotient algebra $R\langle \langle A \rangle \rangle/J$
is finite dimensional provided the subalgebra
$R\langle \langle A \rangle \rangle_{\hat k, \hat k}/J_{\hat k, \hat k}$
is finite dimensional.
In particular, the Jacobian algebra ${\mathcal P}(A,S)$ is finite-dimensional
if and only if so is the subalgebra ${\mathcal P}(A,S)_{\hat k, \hat k}$.
\end{lemma}

\begin{proof}
Similarly to \eqref{eq:k-excluded},
for an $R$-bimodule $B$, we denote
$$B_{k, \hat k} = e_k B \overline e_k =
\bigoplus_{j \neq k} B_{k,j}, \quad
B_{\hat k, k} = \overline e_k B e_k =
\bigoplus_{i \neq k} B_{i,k}.$$
We need to show that if
$R\langle \langle A \rangle \rangle_{\hat k, \hat k}/J_{\hat k, \hat k}$
is finite dimensional then so is each of the spaces
$R\langle \langle A \rangle \rangle_{k, \hat k}/J_{k, \hat k}$,
$R\langle \langle A \rangle \rangle_{\hat k, k}/J_{\hat k, k}$ and
$R\langle \langle A \rangle \rangle_{k, k}/J_{k, k}$.
Let us treat $R\langle \langle A \rangle \rangle_{k, k}/J_{k, k}$;
the other two cases are done similarly (and a little simpler).

Let
$$Q_1 \cap A_{k,\hat k} = \{a_1, \dots, a_s\}, \quad
Q_1 \cap A_{\hat k, k} = \{b_1, \dots, b_t\}.$$
We have
$$R\langle \langle A \rangle \rangle_{k, k} = K e_k \oplus
\bigoplus_{\ell, m} a_\ell R\langle \langle A \rangle \rangle_{\hat k, \hat k} b_m.$$
It follows that there is a surjective map
$\alpha: K \times \Mat_{s \times t}(R\langle \langle A \rangle \rangle_{\hat k, \hat k})
\to R\langle \langle A \rangle \rangle_{k, k}/J_{k, k}$ given by
$$\alpha(c, C) = p(c e_k +
\begin{pmatrix}
a_1 & a_2 & \cdots & a_s
\end{pmatrix} C
\begin{pmatrix}
b_1\\
b_2\\
\vdots\\
b_t
\end{pmatrix}),$$
where $\Mat_{s \times t}(B)$ stands for the space of $s \times t$
matrices with entries in~$B$, and $p$ is the projection
$R\langle \langle A \rangle \rangle \to R\langle \langle A \rangle \rangle/J$.
The kernel of $\alpha$ contains the space $\Mat_{s \times t}(J_{\hat k, \hat k})$,
hence $R\langle \langle A \rangle \rangle_{k, k}/J_{k, k}$ is
isomorphic to a quotient of the finite-dimensional space
$K \times \Mat_{s \times t}(R\langle \langle A \rangle \rangle_{\hat k, \hat k}/J_{\hat k, \hat k})$.
Thus, $R\langle \langle A \rangle \rangle_{k, k}/J_{k, k}$ is
finite dimensional, as desired.
\end{proof}

To finish the proof of Proposition~\ref{pr:fin-dim},
suppose that ${\mathcal P}(A,S)$ is finite dimensional.
Then ${\mathcal P}(\widetilde A,\widetilde S)_{\hat k, \hat k}$
is finite dimensional by Proposition~\ref{pr:k-excluded-mutation-invariant}.
Applying Lemma~\ref{lem:fin-dim-reduction} to the QP
$(\widetilde A,\widetilde S)$, we conclude that
${\mathcal P}(\widetilde A,\widetilde S)$ is finite dimensional, as desired.
\end{proof}

Remembering \eqref{eq:mutilde-mu} and using
Proposition~\ref{pr:jacobian-algebra-invariant}, we see that
Propositions~\ref{pr:k-excluded-mutation-invariant} and
\ref{pr:fin-dim} have the following corollary.

\begin{corollary}
\label{cor:mutation-preserves-fin-dim}
Suppose $(A,S)$ is a reduced QP satisfying \eqref{eq:no-2-cycles-thru-k},
and let $(\overline A, \overline S) = \mu_k(A,S)$ be a reduced QP
obtained from $(A,S)$ by the mutation at~$k$.
Then the algebras ${\mathcal P}(A,S)_{\hat k, \hat k}$
and ${\mathcal P}(\overline A,\overline S)_{\hat k, \hat k}$
are isomorphic to each other, and
${\mathcal P}(A,S)$ is finite-dimensional
if and only if so is ${\mathcal P}(\overline A,\overline S)$.
\end{corollary}

We see that the class of QPs with finite dimensional Jacobian
algebras is invariant under mutations.
Let us now present another such class.

\begin{definition}
\label{def:deff-space}
For every QP $(A,S)$, we define its
\emph{deformation space} $\deff(A,S)$ by
\begin{equation}
\label{eq:deff-space}
\deff(A,S) = \Tr({\mathcal P}(A, S))/R
\end{equation}
(see Definitions~\ref{def:cyclic-stuff} and \ref{def:trace-space}).
\end{definition}

\begin{remark}
\label{rem:deformations}
Definition~\ref{def:deff-space} can be motivated as follows
(we keep the following arguments informal although with some work they can be made rigorous).
Let $G = \Aut(R\langle\langle A\rangle\rangle)$
be the group of algebra automorphisms of
$R\langle\langle A\rangle\rangle$ (acting as the identity on~$R$).
Using Proposition~\ref{pr:automorphisms}, we can think of~$G$ as
an infinite dimensional algebraic group.
In view of Definition~\ref{def:trace-space}, $G$
acts naturally on the trace space $\Tr(R\langle\langle A\rangle\rangle)$.
Remembering Definition~\ref{def:AS-isomorphism}, it is natural
to think of the deformation space of a QP $(A,S)$
as the normal space at $\pi(S)$ of the orbit $G \cdot \pi(S)$ in
the ambient space $\pi({\mathfrak m}(A)^2)$
(recall that $\pi$ stands for the natural projection
$R\langle\langle A\rangle\rangle \to \Tr(R\langle\langle A\rangle\rangle)$).
Arguing as in Lemma~\ref{lem:noncommutative-Taylor}, we conclude
that the infinitesimal action of (the Lie algebra of)~$G$ on
 $\pi({\mathfrak m}(A)^2)$ is by the transformations
$$\pi(u) \mapsto \pi(\sum_{k=1}^N b_k \partial_{a_k} u),$$
where $Q_1 = \{a_1, \dots, a_N\}$ is the set of arrows, and
$b_k \in {\mathfrak m}(A)_{h(a_k), t(a_k)}$
(this is well defined in view of Proposition~\ref{pr:cyclic-equivalence}).
This makes it natural to identify the tangent space at $\pi(S)$ of $G \cdot \pi(S)$
with $\pi(J(S))$, hence to identify the corresponding
normal space with $\pi({\mathfrak m}(A))/\pi(J(S))$, or
equivalently, with the space $\deff(S)$ given by
\eqref{eq:deff-space}.
\end{remark}

\begin{proposition}
\label{pr:Tr-mutation-invariant}
In the situation of Proposition~\ref{pr:k-excluded-mutation-invariant},
deformation spaces $\deff(\widetilde A, \widetilde S)$
and $\deff(A, S)$ are isomorphic to each other.
\end{proposition}

\begin{proof}
In view of Proposition \ref{pr:cyclical-equiv-thru-trace},
$\deff(A,S)$ is isomorphic to
$\Tr({\mathcal P}(A, S)_{\hat k, \hat k})/R_{\hat k, \hat k}$.
Therefore, our assertion is immediate from
Proposition~\ref{pr:k-excluded-mutation-invariant}.
\end{proof}

\begin{definition}
\label{def:rigid-QP}
We call a QP $(A,S)$ \emph{rigid} if $\deff(A,S) = \{0\}$, i.e.,
if $\Tr({\mathcal P}(A, S)) = R$.
\end{definition}

Combining Propositions~\ref{pr:jacobian-algebra-invariant}
and \ref{pr:Tr-mutation-invariant}, we obtain the following corollary.

\begin{corollary}
\label{cor:mutation-preserves-rigidity}
If a reduced QP $(A,S)$ satisfies \eqref{eq:no-2-cycles-thru-k}
and is rigid, then the QP $(\overline A, \overline S) = \mu_k(A,S)$
is also rigid.
\end{corollary}

Some examples of rigid and non-rigid QPs will be given in
Section~\ref{sec:QP-matrix-mutations}.

\section{Nondegenerate QPs}
\label{sec:generic}

If we wish to be able to apply to a reduced QP  $(A,S)$ the mutation at
\emph{every} vertex of $Q_0$, the $R$-bimodule $A$ must satisfy
\eqref{eq:no-2-cycles-thru-k} at all vertices.
Thus, the arrow span $A$ must be $2$-acyclic (see
Definition~\ref{def:2-acyclic}).
Such an arrow span $A$ can be encoded by a skew-symmetric integer
matrix $B = B(A) = (b_{i,j})$ with rows and columns labeled by $Q_0$, by
setting
\begin{equation}
\label{eq:B-thru-A}
b_{i,j} = \dim A_{i,j} - \dim A_{j,i}.
\end{equation}
Indeed, the dimensions of the components $A_{i,j}$ are recovered
from~$B$ by
\begin{equation}
\label{eq:A-thru-B}
\dim A_{i,j} =[b_{i,j}]_+,
\end{equation}
where we use the notation
\begin{equation}
\label{eq:x+}
[x]_+ = \max(x,0).
\end{equation}

\begin{proposition}
\label{pr:A-B-mutation}
Let $(A,S)$ be a $2$-acyclic reduced QP, and
suppose that the reduced QP $\mu_k(A,S) = (\overline A, \overline S)$
obtained from $(A,S)$ by the mutation at some vertex~$k$
(see Definition~\ref{def:reduced-mutation}) is also $2$-acyclic.
Let $B(A) = (b_{i,j})$ and $B(\overline A) = (\overline b_{i,j})$
be the skew-symmetric integer matrices given by \eqref{eq:B-thru-A}.
Then we have
\begin{equation}
\label{eq:B-mutation}
\overline b_{i,j} =
\begin{cases}
-b_{i,j} & \text{if $i=k$ or $j=k$;} \\[.05in]
b_{i,j} + [b_{i,k}]_+ \ [b_{k,j}]_+ - [-b_{i,k}]_+\ [-b_{k,j}]_+
 & \text{otherwise.}
\end{cases}
\end{equation}
\end{proposition}

\begin{proof}
First we note that by Proposition~\ref{pr:trivial-potential},
if $(C,T)$ is a trivial QP then $\dim C_{i,j} = \dim C_{j,i}$
for all $i, j$.
In view of \eqref{eq:mutilde-mu}, this implies that
\begin{equation}
\label{eq:B-via-tildeA}
\overline b_{i,j} = \dim \overline A_{i,j} - \dim \overline A_{j,i} =
\dim \widetilde A_{i,j} - \dim \widetilde A_{j,i},
\end{equation}
where $(\widetilde A, \widetilde S) = \widetilde \mu_k(A,S)$.
Using \eqref{eq:mu-k-A}, we obtain
$$\dim \widetilde A_{i,j} =
\begin{cases}
\dim A_{j,i} & \text{if $i=k$ or $j=k$;} \\[.05in]
\dim A_{i,j} + \dim A_{i,k} \dim A_{k,j}
 & \text{otherwise.}
\end{cases}$$
To obtain \eqref{eq:B-mutation}, it remains to substitute
these expressions into \eqref{eq:B-via-tildeA} and
use \eqref{eq:A-thru-B}.
\end{proof}

An easy calculation using the obvious identity
$x=[x]_+ - [-x]_+$ shows that the second case in \eqref{eq:B-mutation}
can be rewritten in several equivalent ways as follows:
\begin{align*}
\overline b_{i,j} &= b_{i,j} + {\rm sgn}(b_{i,k}) \ [b_{i,k}b_{k,j}]_+\\
&= b_{i,j} + [-b_{i,k}]_+\,b_{k,j} +b_{i,k} [b_{k,j}]_+\\
& = b_{i,j} + \displaystyle\frac{|b_{i,k}| b_{k,j} + b_{i,k} |b_{k,j}|}{2}.
\end{align*}
It follows that the transformation $B \mapsto \overline B$ given
by \eqref{eq:B-mutation} coincides with the \emph{matrix mutation}
at~$k$ which plays a crucial part in the theory of cluster algebras
(cf.~\cite[(4.3)]{ca1}, \cite[(2.2), (2.5)]{ca4}).

We see that the mutations of $2$-acyclic QPs provide a
natural framework for matrix mutations.
With some abuse of notation, we denote by $\mu_k(A)$ the
$2$-acyclic $R$-bimodule such that the skew-symmetric matrix
$B(\mu_k(A))$ is obtained from $B(A)$ by the mutation at~$k$;
thus, $\mu_k(A)$ is determined by~$A$ up to an isomorphism.

Note that the matrix mutations at arbitrary vertices can be iterated
indefinitely, while the $2$-acyclicity condition \eqref{eq:no-2-cycles}
can be destroyed by a QP mutation.
We will study QPs for which this does not happen.

\begin{definition}
\label{def:nondegenerate}
Let $k_1,\dots,k_\ell\in Q_0$ be a finite sequence of vertices such that
$k_p \neq k_{p+1}$ for $p = 1, \dots, \ell-1$.
We say that a QP $(A,S)$ is $(k_\ell,\cdots,k_1)$-\emph{nondegenerate}
if all the QPs $(A,S), \mu_{k_1}(A,S),  \mu_{k_2} \mu_{k_1}(A,S),
\dots, \mu_{k_\ell}\cdots \mu_{k_1}(A,S)$ are $2$-acyclic (hence well-defined).
We say that $(A,S)$ is \emph{nondegenerate} if it is
$(k_\ell,\dots,k_1)$-nondegenerate for every sequence of vertices
as above.
\end{definition}

To state our next result recall the terminology introduced before
Proposition~\ref{pr:reduced-2-acyclic}.
In particular, for a given quiver with the arrow span~$A$, the QPs
on~$A$ are identified with the elements of $K^{{\mathcal C}(A)}$.

\begin{proposition}
\label{pr:finite-nondegeneracy}
Suppose that the base field $K$ is infinite, $Q$ is a $2$-acyclic
quiver with the arrow span~$A$, a sequence of vertices
$k_1,\dots,k_\ell$ is as in Definition~\ref{def:nondegenerate},
and $A' = \mu_{k_\ell}\cdots \mu_{k_1}(A)$.
Then there exist a non-zero polynomial function
$F: K^{{\mathcal C}(A)} \to K$ and a regular
map $G: U(F) \to K^{{\mathcal C}(A')}$
such that every QP $(A,S)$ with $S \in U(F)$ is
$(k_\ell,\dots,k_1)$-nondegenerate, and, for any QP $(A,S)$ with $S \in U(F)$,
the QP $\mu_{k_\ell}\cdots \mu_{k_1}(A,S)$ is right-equivalent to $(A', G(S))$.
\end{proposition}

\begin{proof}
We proceed by induction on $\ell$.
First let us deal with the case $\ell=1$, that is, with a single
mutation $\mu_k$.
Recall that $\mu_k(A,S) = (\overline A, \overline S)$ is the reduced part of the QP
$\widetilde \mu_k(A,S) = (\widetilde A, \widetilde S)$ given by
\eqref{eq:mu-k-A} and \eqref{eq:mu-k-S}.
It is clear from the definition that
$\widetilde S = \widetilde G(S)$
for a polynomial map
$\widetilde G: K^{{\mathcal C}(A)} \to K^{{\mathcal C}(\widetilde A)}$.
Now let us apply Proposition~\ref{pr:reduced-2-acyclic} to the
quiver with the arrow span $\widetilde A$.
We see that there exists a polynomial function of the form
$D_{c_1, \dots, c_N}^{d_1, \dots, d_N}$ on $K^{{\mathcal C}(\widetilde A)}$
(see \eqref{eq:max-minor},
where we have changed the notation for the arrows to avoid the
notation conflict with Section~\ref{sec:mutations})
such that the reduced part $(\overline A, \overline S)$ of a QP
$(\widetilde A, \widetilde S)$ is $2$-acyclic whenever
$\widetilde S \in U(D_{c_1, \dots, c_N}^{d_1, \dots, d_N})$.
Furthermore, for $\widetilde S \in U(D_{c_1, \dots, c_N}^{d_1, \dots, d_N})$,
the QP $(\overline A, \overline S)$ is right-equivalent to $(A',H(\widetilde S))$
for some regular map $H: U(D_{c_1, \dots, c_N}^{d_1, \dots, d_N}) \to K^{{\mathcal C}(A')}$,
where $A' = \mu_k(A)$.
We now define a polynomial function $F: K^{{\mathcal C}(A)} \to K$
and a regular map $G: U(F) \to K^{{\mathcal C}(A')}$ by setting
\begin{equation}
\label{eq:D-3-cycles}
F = D_{c_1, \dots, c_N}^{d_1, \dots, d_N} \circ \widetilde G,
\quad G = H \circ \widetilde G.
\end{equation}
To finish the argument for $\ell = 1$, it remains to show that $F$ is not
identically equal to zero.
But this is clear from the definitions \eqref{eq:max-minor} and
\eqref{eq:mu-k-S}, since the oriented $2$-cycles in $\widetilde A$
(up to cyclical equivalence) are of the form $c[ba]$ and so are in a bijection with the oriented
$3$-cycles $cba$ in $A$ that pass through $k$.

Now assume that $\ell \geq 2$, and that our assertion holds if we
replace~$\ell$ by $\ell - 1$.
Let $A_1 = \mu_{k_1}(A)$, so $A' = \mu_{k_\ell}\cdots \mu_{k_2}(A_1)$.
By the inductive assumption, there exist a non-zero polynomial function
$F': K^{{\mathcal C}(A_1)} \to K$ and a regular
map $G': U(F') \to K^{{\mathcal C}(A')}$
such that, for any QP $(A_1,S_1)$ with $S_1 \in U(F')$,
the QP $\mu_{k_\ell}\cdots \mu_{k_2}(A_1,S_1)$ is right-equivalent to $(A', G'(S_1))$.
Also by the already established case $\ell = 1$, there exists a non-zero polynomial function
$F'': K^{{\mathcal C}(A_1)} \to K$ such that, for any QP $(A_1,S_1)$ with $S_1 \in U(F'')$,
the QP $\mu_{k_1}(A_1,S_1)$ is $2$-acyclic, hence is
right-equivalent to some QP on $A$.
Since the base field~$K$ is assumed to be infinite, we have $U(F')
\cap U(F'') \neq \emptyset$.
Choose $S_1^{(0)} \in U(F') \cap U(F'')$, and let $(A,S_0) = \mu_k(A_1,S_1^{(0)})$.
By Theorem~\ref{th:mutation-involutive}, we have $\mu_k(A,S_0) = (A_1,S_1^{(0)})$.
By the above argument for $\ell = 1$, there exist a nonzero
polynomial function $F_1: K^{{\mathcal C}(A)} \to K$ and a regular map
$G_1: U(F_1) \to K^{{\mathcal C}(A_1)}$ (of the type \eqref{eq:D-3-cycles})
such that $\mu_k(A,S) = (A_1, G_1(S))$ for $S \in U(F_1)$.
In particular, we have $G_1(S_0) = S_1^{(0)}$ implying that
$F' \circ G_1$ is a nonzero polynomial function on $K^{{\mathcal C}(A)}$.
It follows that the nonzero polynomial function $F(S) = F_1(S) F'(G_1(S))$ and the
regular map $G = G' \circ G_1: U(F) \to K^{{\mathcal C}(A')}$ are
well-defined and satisfy all the required conditions.
This completes the proof of
Proposition~\ref{pr:finite-nondegeneracy}.
\end{proof}

\begin{corollary}
\label{cor:nondegenerate-exists}
For every $2$-acyclic arrow span~$A$, there exists a countable
family ${\mathcal F}$ of nonzero polynomial functions on
$K^{{\mathcal C}(A)}$ such that the QP $(A,S)$ is nondegenerate
whenever $S \in \bigcap_{F \in {\mathcal F}} U(F)$.
In particular, if the base field $K$ is uncountable, then
there exists a nondegenerate QP on~$A$.
\end{corollary}

\begin{proof}
By Proposition~\ref{pr:finite-nondegeneracy},
for every sequence $k_1, \dots,k_\ell$ as in Definition~\ref{def:nondegenerate},
there exists a nonzero polynomial function
$F_{k_1,\dots,k_\ell}$ on $K^{{\mathcal C}(A)}$ such that
a QP $(A,S)$ is $(k_\ell,\dots,k_1)$-nondegenerate for
$S \in U(F_{k_1,\dots,k_\ell})$.
These functions form a desired countable family~$\mathcal F$.

It remains to prove that $\bigcap_{F \in {\mathcal F}} U(F) \neq \emptyset$ provided
$K$ is uncountable.
If $A$ is acyclic, i.e., has no oriented cycles, then $K^{{\mathcal C}(A)} = \{0\}$, and each of
the functions in ${\mathcal F}$ is just a nonzero constant, so there is nothing to
prove; no assumption on $K$ is needed here.
If $A$ has at least one oriented cycle then the set ${\mathcal C}(A)$ is
countable (recall that it consists of cyclic paths up to cyclical equivalence).
Thus, we can realize $K^{{\mathcal C}(A)}$ as the polynomial ring
$K[X_1, X_2, \dots]$ in countably many indeterminates.
Since $K$ is uncountable, we can choose $x_1$ so that $F(x_1)\neq 0$ for all
$F \in {\mathcal F}\cap K[X_1]$.
Then we choose $x_2$ so that $F(x_1,x_2)\neq 0$ for all
$F \in {\mathcal F}\cap K[X_1,X_2]$.
Continuing like this, we find a sequence $x_1,x_2,\dots$ such that $F(x_1,x_2,\dots)\neq 0$
for all $F \in {\mathcal F}$.
\end{proof}

\section{Rigid QPs}
\label{sec:QP-matrix-mutations}

\begin{proposition}
\label{pr:rigidity-no-2-cycles}
Every rigid reduced QP $(A,S)$ is $2$-acyclic.
\end{proposition}

\begin{proof}
First note that the definition of rigidity can be conveniently
restated as follows:
\begin{eqnarray}
\label{eq:rigidity-working-def}
&\text{a QP $(A,S)$ is rigid if and only if every potential}\\
\nonumber
&\text{on~$A$ is cyclically equivalent to an element of $J(S)$.}
\end{eqnarray}

Now suppose for the sake of contradiction that for some $i \neq j$ both components
$A_{i,j}$ and $A_{j,i}$ are non-zero.
Choose non-zero elements $a \in A_{i,j}$ and $b \in A_{j,i}$.
Remembering the definition of the Jacobian ideal
(see Definition~\ref{def:cyclic-stuff}), it is easy to see that
the cyclic part of $J(S)$ is contained in ${\mathfrak m}(A)^3$.
It follows that $ab$ is \emph{not} cyclically equivalent to an element of $J(S)$,
in contradiction with \eqref{eq:rigidity-working-def}.
\end{proof}

Combining Proposition~\ref{pr:rigidity-no-2-cycles}
with Corollary~\ref{cor:mutation-preserves-rigidity},
we obtain the following result.

\begin{corollary}
\label{cor:rigid-is-nondegenerate}
Any rigid QP is nondegenerate.
\end{corollary}

Let us now give some examples.

\begin{example}
\label{ex:acyclic}
Recall that a skew-symmetric integer matrix~$B$ is \emph{acyclic}
if the corresponding directed graph (with an arrow $i \to j$
associated with each entry $b_{i,j} > 0$) has no oriented cycles.
If the matrix $B(A)$ given by \eqref{eq:B-thru-A} is acyclic, then
$R\langle \langle A \rangle \rangle_{\rm cyc} = \{0\}$, and so the
only QP associated with $A$ is $(A,0)$, which is clearly rigid.

Now suppose that $A$ is $2$-acyclic, and that
$B(A)$ is not necessarily acyclic but is mutation
equivalent to an acyclic matrix (i.e., can be transformed to an
acyclic matrix by a sequence of mutations).
As a consequence of Corollary~\ref{cor:mutation-preserves-rigidity}
and Theorem~\ref{th:mutation-involutive}, there exists a potential $S$ such that
$(A,S)$ is a rigid reduced QP; moreover, $(A,S)$ is unique up
to right-equivalences.
\end{example}

\begin{example}
\label{ex:0-potential}
For $A$ arbitrary, the deformation space of a QP $(A,0)$ is
naturally identified with the space of potentials modulo cyclical
equivalence, hence it is infinite-dimensional provided $A$ has at least one
oriented cycle.
\end{example}

\begin{example}[Cyclic triangle]
\label{ex:simple-triangle-any-S}
Let $Q$ be the quiver with three vertices $1,2,3$ and
three arrows $a: 1 \to 2$, $b: 2 \to 3$ and
$c: 3 \to 1$:
$$
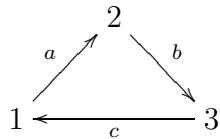
\xymatrix{
& 2\ar^{b}[rd] & \\
1\ar[ru]^a & & 3\ar[ll]^c
}
$$
An arbitrary potential $S$ is cyclically equivalent to the one of
the form $S = F(cba)$, where $F \in K[[t]]$ is a formal power series.
The deformation space $\deff(A,S)$ is naturally isomorphic to the
quotient space of $t K[[t]]$ modulo the ideal generated by $t dF/dt$.
If ${\rm char} K = 0$, and $n \geq 1$ is the smallest exponent such that
$t^n$ appears in~$F$, then $\dim \deff(A,S) = n-1$.
In particular, $(A,S)$ is rigid if and only if $n=1$.

Now let us consider the QP $(\widetilde A, \widetilde S) =
\widetilde \mu_2 (A,S)$; in view of
\eqref{eq:arrows-thru-k-reversed},
\eqref{eq:arrows-not-thru-k}
and \eqref{eq:mu-k-S}, $\widetilde A$ has four arrows
$a^\star, b^\star, c, [ba]$, and
$$\widetilde S = F(c[ba]) + [ba] a^\star b^\star.$$
Thus, if $n \geq 2$ then $(\widetilde A, \widetilde S)$ is reduced
and so is equal to $\mu_2(A,S) = (\overline A, \overline S)$.
Since $\mu_2(A,S)$ has an oriented $2$-cycle formed by the arrows
$c$ and $[ba]$, the mutations at vertices $1$ and $3$ cannot be applied.
We see that the QP $(A,F(cba))$ is degenerate for $n \geq 2$.
\end{example}

\begin{example}[Double cyclic triangle]
\label{ex:double-triangle}
Now consider the quiver with three vertices $1,2,3$ and
six arrows $a_1, a_2: 1 \to 2$, $b_1, b_2: 2 \to 3$ and
$c_1, c_2: 3 \to 1$:
$$
\xymatrix{
& 2\ar@<.5ex>[rd]^{b_1}\ar@<-.5ex>[rd]_{b_2} & \\
1\ar@<.5ex>[ru]^{a_1}\ar@<-.5ex>[ru]_{a_2} & & 3\ar@<-.5ex>[ll]_{c_2}\ar@<.5ex>[ll]^{c_1}}
$$

Any potential~$S$ on~$A$ is cyclically equivalent to the one whose
degree $3$ component belongs to the $8$-dimensional space
$A^3_{1,1} = A_{1,3} A_{3,2} A_{2,1}$.
It is known that the diagonal action of the group $GL_2^3$
on $K^2 \otimes K^2 \otimes K^2$ has seven orbits, see e.g.,
\cite[Chapter 14, Example 4.5]{gkz}.
Thus, by performing a change of arrows automorphism, we can assume
that the degree~$3$ component of~$S$ is one of the representatives
of these orbits.
An easy case-by-case analysis shows that no potential can give rise to
a rigid QP on~$A$.

For instance, let
\begin{equation}
\label{eq:S-double-triangle}
S = c_1 b_1 a_1 + c_2 b_2 a_2.
\end{equation}
Then $J(S)$ is the closure of the ideal in $R\langle \langle A \rangle \rangle$
generated by six elements
$$c_1 b_1, b_1 a_1, a_1 c_1, c_2 b_2, b_2 a_2, a_2 c_2.$$
One checks easily that the cyclic path $c_1 b_2 a_1 c_2 b_1 a_2$
is not cyclically equivalent to an element of $J(S)$, hence
$(A,S)$ is not rigid.

Now let us compute $\mu_2(A,S)$.
Again setting $(\widetilde A, \widetilde S) =
\widetilde \mu_2 (A,S)$, we see that
$\widetilde A$ has ten arrows
$$a_1^\star, a_2^\star, b_1^\star, b_2^\star, c_1, c_2,
[b_1 a_1], [b_1 a_2], [b_2 a_1], [b_2 a_2],$$
and
$$\widetilde S = c_1 [b_1 a_1] + c_2 [b_2 a_2] +
\sum_{i,j=1}^2 [b_i a_j] a_j^\star b_i^\star.$$
To obtain the splitting \eqref{eq:decomposition} of
$(\widetilde A, \widetilde S)$, we apply the automorphism
$\varphi$ of $R\langle \langle \widetilde A \rangle \rangle$
fixing all arrows except $c_1$ and $c_2$, and such that
$\varphi (c_i) = c_i - a_i^\star b_i^\star$.
An easy check shows that $\mu_2(A,S) = (\overline A, \overline S)$
can be described as follows: $\overline A$ is $6$-dimensional with
the arrows $a_1^\star, a_2^\star, b_1^\star, b_2^\star,
[b_1 a_2], [b_2 a_1]$, and
$$\overline S = [b_1 a_2] a_2^\star b_1^\star +
[b_2 a_1] a_1^\star b_2^\star.$$
Thus, the mutated QP $(\overline A, \overline S)$ can be
obtained from the initial QP $(A,S)$ by a renumbering of the vertices.
This implies that one can apply to $(A,S)$ unlimited mutations at
arbitrary vertices, so $(A,S)$ is a non-rigid, nondegenerate QP.
\end{example}

\begin{example}
\label{ex:triangular-grid}
For each $n \geq 0$, let us consider the following quiver $Q(n)$,
which we refer to as the \emph{triangular grid of order}~$n$.
The vertex set of $Q(n)$ is
$$Q(n)_0 = \{(p,q,r) \in \Z_{\geq 0}^3 \mid p+q+r = n\};$$
and there is a single arrow $(p_1,q_1,r_1) \to(p_2,q_2,r_2)$
if and only if $(p_2,q_2,r_2) - (p_1,q_1,r_1)$ is one of the three
vectors $(-1,1,0), (0,-1,1), (1,0,-1)$.
Thus, the 
vertices of $Q(n)$ form a regular triangular grid
with $n^2$ cyclically oriented unit triangles.
For example, the quiver $Q(4)$ is
$$
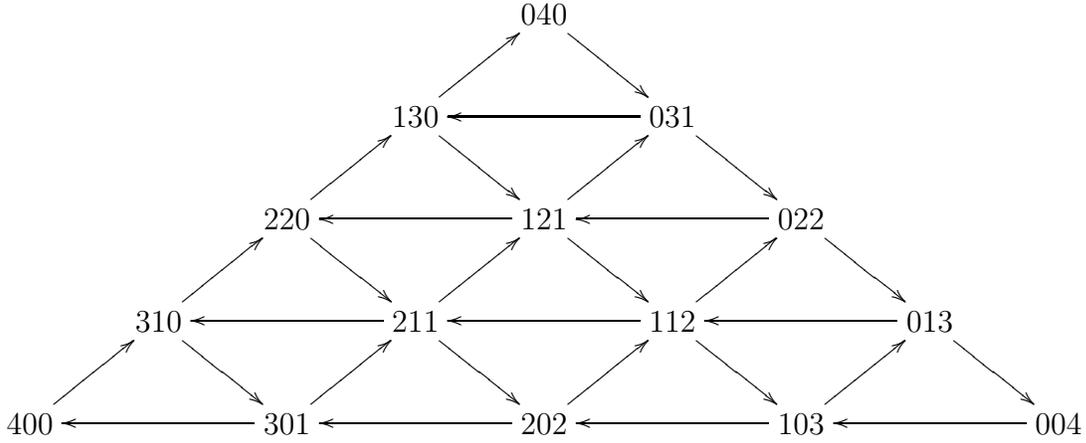
\xymatrix{
& & & & 040 \ar[rd] & & & &\\
& & & 130 \ar[ru]\ar[rd] & & 031 \ar[ll]\ar[rd] & & &\\
& & 220 \ar[ru]\ar[rd] & & 121 \ar[ll]\ar[ru]\ar[rd] & & 022 \ar[ll] \ar[rd] & &\\
& 310 \ar[ru]\ar[rd] & & 211 \ar[ll]\ar[ru]\ar[rd] & & 112\ar[ll]\ar[ru]\ar[rd] & & 013\ar[ll]\ar[rd] &\\
 400 \ar[ru] & & 301 \ar[ll]\ar[ru] & & 202\ar[ll]\ar[ru] & & 103\ar[ll]\ar[ru] & &004 \ar[ll]
}
$$


Let~$A=A(n)$ be the arrow span of $Q(n)$, and
let~$a \in A$ (resp.~$b \in A$, $c \in A$) denote the sum of all
arrows of $Q(n)$ that are parallel translates of $(-1,1,0)$
(resp. $(0,-1,1)$, $(1,0,-1)$).
Thus, every interior vertex~$i$ has three incoming arrows
$e_i a, e_i b, e_i c$ and three outgoing arrows
$a e_i, b e_i, c e_i$.
Every path of length~$d$ can be uniquely written as
$a_d \cdots a_2 a_1 e_j$, where each $a_s$ is one of the
elements $a, b, c$, and $j$ is a vertex; this expression is
non-zero if and only if the polygonal line obtained by attaching
to the vertex~$j$ the vectors corresponding to $a_1, a_2, \dots,
a_d$ (in this order) is contained in our grid.

Define a potential $S \in A^3$ by setting
$$S = cba - bca.$$
Then the Jacobian ideal $J(S)$ is generated by the elements
$$(cb -bc)e_j, (ac -ca)e_j, (ba -ab)e_j$$
for all vertices~$j$.
It follows that the image of the path $a_d \cdots a_2 a_1 e_j$
in the Jacobian algebra ${\mathcal P}(A,S)$ does not change under
any permutation of the factors $a_1, \dots, a_d$.
In particular, we see that ${\mathcal P}(A,S)$ is spanned by the
images of the paths $c^k b^\ell a^m e_j$ for all vertices~$j$ and
all $k, \ell, m$ such that $0 \leq k, \ell, m \leq n$; hence
${\mathcal P}(A,S)$ is finite-dimensional.
By a similar argument, it is easy to see that $(A,S)$ is rigid.
Indeed, every potential on $A$ is cyclically equivalent to an
element of the closure of the span of the elements
$(cba)^m e_j$ for all vertices $j$ and all $m \geq 0$.
Denoting by $p$ the projection
$R\langle \langle A \rangle \rangle \to \Tr({\mathcal P}(A,S))$, we see
that the rigidity of $(A,S)$ follows from the fact that
$p(cba e_j)=0$ for all~$j$.
Now if $ae_j \neq 0$ and $h(a e_j) = i$ then we have
$$p(cba e_j) = p(acbe_i) = p(cab e_i) = p(cba e_i).$$
Continuing in the same way, we obtain that, for any $m \geq 1$
such that $a^m e_j \neq 0$, we have
$p(cba e_j) = p(cba e_k)$, where $k$ is the end-point of the
path $a^m e_j$.
Taking $m$ the largest such that $a^m e_j \neq 0$, we conclude
that $p(cba e_j) = 0$, as desired.
\end{example}

As shown in \cite{kr}, the quiver $Q(3)$ in Example~\ref{ex:triangular-grid}
is \emph{not} mutation-equivalent to an acyclic one.
So there exist QPs with
finite-dimensional Jacobian algebras (and also rigid QPs), which
are not mutation-equivalent to acyclic ones.

We now describe a procedure to obtain new QPs with
finite-dimensional Jacobian algebras (and new rigid QPs)
from old ones.

\begin{definition}
\label{def:sub-QP}
For a QP $(A,S)$ and a subset $I$ of the vertex set $Q_0$,
we define the \emph{restriction} of $(A,S)$ to $I$ as the
QP $(A|_I, S|_I)$ given by
$$A|_I = \bigoplus_{i,j \in I} A_{i,j}$$
and
$$S|_I = \psi_I(S),$$
where $\psi_I: R\langle \langle A \rangle \rangle \to
R\langle \langle A|_I \rangle \rangle$ is the algebra homomorphism
such that $\psi_I(a) = a$ for $a \in A|_I$, and
$\psi_I(b) = 0$ for any arrow $b$ not belonging to $A|_I$.
\end{definition}

\begin{proposition}
\label{pr:QP-restriction}
The homomorphism $\psi_I$ induces an epimorphism
of Jacobian algebras
${\mathcal P}(A,S) \to {\mathcal P}(A|_I,S|_I)$ and an epimorphism
of deformation spaces $\deff(A,S) \to \deff(A|_I,S|_I)$.
Therefore, if ${\mathcal P}(A,S)$ is finite-dimensional, or if
$(A,S)$ is rigid, then the same is true for $(A|_I,S|_I)$.
\end{proposition}

\begin{proof}
Remembering \eqref{eq:cyclic-derivative}, it is easy to see that
$\psi_I(\partial_a S) = \partial_a \psi_I(S)$ for any arrow
$a \in  A|_I$, and $\psi_I(\partial_b S) = 0$ for any arrow
$b$ not belonging to $A|_I$.
Therefore, we have $\psi_I(J(S)) = J(\psi_I(S))$, implying all the assertions.
\end{proof}

\begin{corollary}
\label{cor:submatrix-rigid}
Suppose that $A$ and $A'$ are $2$-acyclic, and
there is a rigid QP $(A,S)$ on~$A$.
Let $B(A)$ and $B(A')$ be the corresponding
skew-symmetric integer matrices given by \eqref{eq:B-thru-A}.
Suppose that $B(A')$ can be obtained by a simultaneous permutation of
rows and columns from some principal submatrix of a matrix
mutation-equivalent to $B(A)$.
Then there exists a rigid QP $(A',S')$ on~$A'$.
\end{corollary}

\begin{proof}
In view of \eqref{eq:B-thru-A}, the matrix $B(A|_I)$ is the
principal submatrix of $B(A)$ involving rows and columns from~$I$.
Therefore, our assertion follows by combining Proposition~\ref{pr:QP-restriction}
with Corollary~\ref{cor:mutation-preserves-rigidity} and
Proposition~\ref{pr:A-B-mutation}.
\end{proof}

We conclude this section with a combinatorial application of
Corollary~\ref{cor:submatrix-rigid}.

\begin{corollary}
\label{cor:double-triangle-not-in-grid}
Let $B = B(A(n))$ be the matrix associated with the
triangular grid of some order~$n$
(see Example~\ref{ex:triangular-grid}).
Then none of the matrices mutation-equivalent to~$B$
contains
$$B' =
\begin{pmatrix}
0 & 2 & -2\\
-2 & 0 & 2\\
2 & -2 & 0
\end{pmatrix}$$
as a principal submatrix.
\end{corollary}

\begin{proof}
Note that $B' = B(A')$, where $A'$ is the quiver in
Example~\ref{ex:double-triangle}.
We have seen that there exists a rigid QP on~$A$,
but not on~$A'$.
Thus our statement is a special case of
Corollary~\ref{cor:submatrix-rigid}.
\end{proof}

\begin{remark}
\label{rem:triangular-grid}
Using the results in \cite[Section~2.6]{ca3} (cf. also \cite[Theorem~1]{gls}),
it is easy to see that the quiver $Q(n)$ in Example~\ref{ex:triangular-grid}
is naturally associated with the cluster algebra structure
in the coordinate ring of the base affine space of the group $SL_{n+3}$.
We have been informed by J.~Schr\"oer that, following his
suggestion, A.~Seven has shown that
a skew-symmetric matrix $B'$ associated with an arbitrary tree
appears as a principal submatrix in some matrix
mutation-equivalent to the matrix $B(A(n))$ for some~$n$.
J.~Schr\"oer also informed us that Corollary~\ref{cor:double-triangle-not-in-grid}
has been also proved by B.~Keller (using a different method).
\end{remark}

\section{Relation to cluster-tilted algebras}
\label{sec:Dynkin-rigid}

Let $Q$ be a quiver with the vertex span~$R$ and
the arrow span $A$.
Assume that $Q$ is $2$-acyclic.
Let $B(A)$ be the corresponding skew-symmetric integer matrix
given by \eqref{eq:B-thru-A}.
As shown in \cite{ca2}, the matrix $B(A)$ gives rise to a
cluster algebra of finite type if and only if $Q$ is
mutation-equivalent to a Dynkin quiver (that is, an arbitrary orientation of a
Dynkin diagram of one of the types $A_n$, $D_n$, $E_6$, $E_7$, or $E_8$).
In particular, as we have seen in Example~\ref{ex:acyclic},
there is a rigid reduced QP $(A,S)$, and it is unique up
to a right-equivalence; in fact, (the right-equivalence class of)
$(A,S)$ is obtained by a sequence of mutations from a QP
$(A_0,0)$, where $A_0$ is associated to a Dynkin quiver.
We now present an explicit choice of such a potential~$S$.
To do this, we need some preparation.

First note that, according to \cite[Theorem~1.8]{ca2},
every quiver mutation-equivalent to a Dynkin quiver has no
multiple edges, that is $|b_{i,j}| \leq 1$ for every entry of $B(A)$.
In other words, we have:
\begin{equation}
\label{eq:simply-laced}
\text{$\dim A_{i,j} \leq 1$ for all $i$ and $j$.}
\end{equation}
Therefore, we can unambiguously denote by
$a_{i,j}$ the only arrow in a non-zero subspace $A_{i,j}$.
We will also use the convention that $a_{i,j} = 0$ whenever
$A_{i,j} = \{0\}$.

Second, we use the following terminology:
a \emph{chordless cycle} in (the underlying graph of)~$Q$
is a $d$-subset of vertices $I \subseteq Q_0$
that can be bijectively labeled by the elements of
$\Z/d\Z$ so that the edges between them
are precisely $\{i, i+1\}$ for $i \in \Z/d\Z$.
According to \cite[Proposition~9.7]{ca2}, if $Q$ is mutation-equivalent to a Dynkin quiver
then the arrows of every chordless cycle in~$Q$ must be cyclically oriented.
In terms of the arrow span~$A$, this condition can be stated as follows:
\begin{align}
\label{eq:chordless}
&\text{For any chordless $d$-cycle~$I$,
there exists a bijection $\nu: \Z/d\Z \to I$ such}\\
\nonumber
&\text{that the set of arrows in the restriction $A|_I$
is $\{a_{\nu(i),\nu(i+1)} \mid i \in \Z/d\Z\}$}.
\end{align}
(see Definition~\ref{def:sub-QP}).
Note that the choice of a bijection $\nu$ in \eqref{eq:chordless} is
unique up to a cyclic shift.

We call a potential~$S$ on~$A$ \emph{primitive} if it has the form
\begin{equation}
\label{eq:chordless-potential}
S = \sum_I x_I a_{\nu(1),\nu(2)} \cdots a_{\nu(d-1),\nu(d)}
a_{\nu(d),\nu(1)},
\end{equation}
where the (finite) sum is over all chordless cycles~$I$ in~$Q$,
for each~$I$ there is a bijection~$\nu$ chosen as in
\eqref{eq:chordless}, and the coefficients
$x_I$ are some non-zero elements of the base field~$K$.


\begin{proposition}
\label{pr:finite-type-potential}
If a quiver~$Q$ with the arrow span~$A$ is mutation-equivalent to a Dynkin quiver,
and $S$ is a primitive potential on~$A$, then the QP $(A,S)$ is rigid.
\end{proposition}

To prove Proposition~\ref{pr:finite-type-potential}, we impose on $Q$ a
weaker assumption that its arrow span $A$ satisfies \eqref{eq:no-2-cycles}, \eqref{eq:simply-laced}
and \eqref{eq:chordless}.
Choose some vertex~$k$, and (as in Section~\ref{sec:generic})
let $\mu_k(A)$ denote the $2$-acyclic $R$-bimodule such that the skew-symmetric matrix
$B(\mu_k(A))$ is obtained from $B(A)$ by the mutation at~$k$.
It is easy to see that $\mu_k(A)$ satisfies \eqref{eq:simply-laced}
but not necessarily \eqref{eq:chordless}.
In view of Corollary~\ref{cor:mutation-preserves-rigidity},
the assertion of Proposition~\ref{pr:finite-type-potential}
is a consequence of the following lemma.

\begin{lemma}
\label{lem:chordless-mutation-invariant}
Suppose that the arrow span $A$ of a quiver $Q$
satisfies \eqref{eq:no-2-cycles}, \eqref{eq:simply-laced}
and \eqref{eq:chordless}, and that $\mu_k(A)$ also satisfies
\eqref{eq:chordless} for some vertex~$k$.
Let $S$ be a primitive potential on~$A$.
Then the QP $(\overline A, \overline S)= \mu_k(A,S)$ is right-equivalent
to a QP on $\mu_k(A)$ with a primitive potential.
\end{lemma}

\begin{proof}
Let $\widetilde \mu_k(A,S) = (\widetilde A, \widetilde S)$ be the QP
given by \eqref{eq:arrows-thru-k-reversed}, \eqref{eq:arrows-not-thru-k}
and \eqref{eq:mu-k-S}.
Denote by $In(k)$ (resp. $Out(k)$) the set of vertices $j$
(resp.~$i$) such that $\dim A_{k,j} = 1$ (resp. $\dim A_{i,k} = 1$).
In view of \eqref{eq:chordless}, every arrow~$a$ with both ends in
$In(k) \cup Out(k)$ has $h(a) \in In(k)$ and $t(a) \in Out(k)$.
We denote the set of these arrows by $Q'_1$.
The arrows of~$\widetilde A$ can be unambiguously denoted as follows:
\begin{itemize}
\item $\widetilde a_{i,j} = a_{i,j}$ for all $i,j$ different
from~$k$ and such that $a_{i,j} \neq 0$.
\item $\widetilde a_{i,j} = [a_{i,k} a_{k,j}]$ for
all $i \in Out(k), j \in In(k)$.
\item $\widetilde a_{j,k} = a_{k,j}^\star$ for $j \in In(k)$.
\item $\widetilde a_{k,i} = a_{i,k}^\star$ for $i \in Out(k)$.
\end{itemize}

We can split $S$ into the sum of four terms
\begin{equation}
\label{eq:S-4-parts}
S = S_1 + S_2 + S_3 + S_4,
\end{equation}
where
\begin{itemize}
\item $S_1$ involves (oriented) $3$-cycles $a_{i,k} a_{k,j}
a_{j,i}$.
\item $S_2$ involves chordless $d$-cycles through~$k$ with $d \geq 4$;
\item $S_3$ involves chordless cycles having an arrow $a_{j,i} \in Q'_1$ but not passing through~$k$;
\item $S_4$ involves chordless cycles having no arrows with both
ends in $In(k) \cup \{k\} \cup Out(k)$.
\end{itemize}
Using \eqref{eq:chordless}, it is easy to see that every chordless
cycle~$I$ involved in $S_2$ or $S_3$ has exactly one common point with
each of the sets $In(k)$ and $Out(k)$.
Also note that every term in $S_1$ or $S_3$ contains exactly one arrow
from $Q'_1$, while none of the terms in $S_2$ or $S_4$ contain
such arrows.
Remembering \eqref{eq:mu-k-S}, we write the potential $\widetilde S$
as follows:
\begin{equation}
\label{eq:tilde-S-split}
\widetilde S = [S_1] + [S_2] + [S_3] + [S_4] + \Delta_k.
\end{equation}
We have
$$[S_1] = \sum_{a_{j,i} \in Q'_1} x_{\{i,j,k\}} \widetilde a_{i,j} a_{j,i},$$
and this is the degree~$2$ component $\widetilde S^{(2)}$ of $\widetilde S$.
In view of \eqref{eq:mutilde-mu}, the arrows in $\overline A$ are
obtained from those in $\widetilde A$ by removing all arrows
$a_{j,i} \in Q'_1$ and their opposites $\widetilde a_{i,j}$.

Inspecting the other terms in \eqref{eq:tilde-S-split},
it is easy to see that
$$[S_2] = \overline S_3, \quad [S_4] = \overline S_4,$$
$$\Delta_k = \overline S_1 + \sum_{a_{j,i} \in Q'_1}
\widetilde a_{i,j} \widetilde a_{j,k} \widetilde a_{k,i},$$
$$[S_3] = \sum_{a_{j,i} \in Q'_1} f_{i,j} a_{j,i},$$
where $f_{i,j} = \partial_{a_{j,i}} S_3$,
and the terms $\overline S_1, \overline S_3$ and $\overline S_4$
have the same meaning as in \eqref{eq:S-4-parts} with $A$ replaced
by $\overline A$.
Let $\varphi$ be the automorphism of $R \langle \langle \widetilde
A \rangle \rangle$ acting on arrows as follows:
$$\varphi(a_{j,i}) = x_{\{i,j,k\}}^{-1}(a_{j,i} - \widetilde a_{j,k} \widetilde a_{k,i}),
\quad \varphi(\widetilde a_{i,j}) = \widetilde a_{i,j} - x_{\{i,j,k\}}^{-1}f_{i,j}$$
for every $a_{j,i} \in Q'_1$, and $\varphi$ fixes the rest of the arrows.
Then we have
\begin{align*}
\varphi(\widetilde S) &= \overline S_1 + \overline S_3 + \overline S_4
+ \varphi(\sum_{a_{j,i} \in Q'_1}
(x_{\{i,j,k\}} \widetilde a_{i,j} a_{j,i} + \widetilde a_{i,j} \widetilde a_{j,k} \widetilde a_{k,i}
+ f_{i,j} a_{j,i}))\\
&= \overline S_1 + \overline S_3 + \overline S_4 +
\sum_{a_{j,i} \in Q'_1} (\widetilde a_{i,j}a_{j,i}
 - x_{\{i,j,k\}}^{-1} f_{i,j}\widetilde a_{j,k} \widetilde a_{k,i}).
\end{align*}
We see that the degree~$2$ component of $\varphi(\widetilde S)$ is
$$\varphi(\widetilde S)^{(2)} = \sum_{a_{j,i} \in Q'_1} \widetilde a_{i,j}a_{j,i},$$
and a moment's reflection shows that
$$-\sum_{a_{j,i} \in Q'_1}x_{\{i,j,k\}}^{-1} f_{i,j}\widetilde a_{j,k} \widetilde
a_{k,i}$$
can be viewed as the component $\overline S_2$ of a primitive
potential on~$\overline A$.
We conclude that $\mu_k(A,S)$ is right-equivalent to
$(\overline A, \overline S_1 + \overline S_2 + \overline S_3 + \overline
S_4)$, finishing the proof.
\end{proof}

We conclude this section by briefly discussing a connection between
Jacobian algebras of rigid QPs and cluster-tilted algebras
introduced in \cite{BMR1}.
We refer to \cite{BMR1} for precise definitions; roughly speaking,
cluster-tilted algebras are endomorphism rings of tilting objects
in cluster categories.
One can associate such an algebra to any quiver~$Q$ which is mutation-equivalent
to a Dynkin quiver.
As shown in \cite[Theorem~4.1]{CCS} (for type~$A$) and
\cite[Theorems~4.1, 4.2]{BMR2}, the cluster-tilted algebra associated to~$Q$ is
isomorphic to the path algebra of~$Q$ modulo an explicitly described ideal of
relations.
By inspection, this ideal is exactly the Jacobian ideal of a
primitive potential~$S$ given by \eqref{eq:chordless-potential}.
Thus, Proposition~\ref{pr:finite-type-potential}
has the following corollary, which shows that the Jacobian algebras of QPs can be viewed as
generalizations of cluster-tilted algebras.

\begin{corollary}
\label{cor:jacobian-cluster-tilted}
If a quiver~$Q$ with the arrow span~$A$ is mutation-equivalent to a Dynkin quiver,
then the Jacobian algebra ${\mathcal P}(A,S)$ of a rigid QP
on~$A$ (explicitly given by \eqref{eq:chordless-potential})
is isomorphic to the cluster-tilted algebra associated to~$Q$.
\end{corollary}

\section{Decorated representations and their mutations}
\label{sec:reps}

The following definition is inspired by \cite{mrz}.

\begin{definition}
\label{def:dec-rep}
A \emph{decorated representation} of a QP $(A,S)$
is a pair ${\mathcal M}=(M,V)$, where $V$ is  a finite-dimensional (left)
$R$-module, and  $M$ is a finite-dimensional
$R \langle\langle A \rangle \rangle$-module annihilated by $J(S)$.
\end{definition}

Equivalently, $M$ is a finite-dimensional ${\mathcal P}(A,S)$-module.
We will sometimes write $\mathcal M = (A,S,M,V)$ and refer to
such a quadruple as a \emph{QP-representation}.

We have $M = \bigoplus_{i \in Q_0}  M_i$ and $V = \bigoplus_{i \in Q_0} V_i$,
where $M_i = e_i M$ and $V_i = e_i V$.
With some abuse of notation, for $u \in R \langle\langle A \rangle \rangle$
or $u \in {\mathcal P}(A,S)$, we denote the multiplication
operator $m \mapsto um$ on $M$ simply as $u: M \to M$; we write $u
= u_M$ if we need to emphasize the dependence of this operator on~$M$.
In particular, for each arrow $a \in A$, we have $a: M_{t(a)} \to M_{h(a)}$,
and $a|_{M_i} = 0$ for $i \neq t(a)$.

Note that every finite-dimensional
$R \langle\langle A \rangle \rangle$-module~$M$ is
\emph{nilpotent}, i.e.,~$M$ is annihilated by ${\mathfrak m}^n$ for $n \gg 0$.
We thank Bill Crawley-Boevey for pointing this out to us; the
following argument was also suggested by him.
The above claim is a special case of the following more general
fact: if a ring $U$ with unit is complete in the $I$-adic topology
for some two-sided ideal~$I$, and $M$ is a $U$-module of finite
length~$n$, then $I^n M = \{0\}$.
Indeed, the element $1+x$ is invertible in $U$ for any $x \in I$,
since it has inverse $1-x+x^2-x^3+ \cdots$.
Thus $I$ is contained in the Jacobson radical $J$ (since $J$ is the set of
$x \in U$ such that $1+xy$ is invertible for all $y \in U$).
Thus $IS = \{0\}$ for any simple $U$-module $S$ (since $J$ is the intersection of
annihilators of all simple modules).
Now if $M$ has composition series
$$\{0\} = M_0 \subset M_1 \subset \cdots \subset M_n = M,$$
then for all $i \geq 1$, we have $I(M_i/M_{i-1}) = \{0\}$, so $IM_i \subseteq M_{i-1}$.
It follows that $I^n M = \{0\}$, as claimed.

The aim of this section is to extend the definition of QP-mutations
in Corollary~\ref{cor:mutations-respect-isom}
and Definition~\ref{def:reduced-mutation} to the level of
QP-representations, and to prove a representation-theoretic
extension of Theorem~\ref{th:mutation-involutive}.
To do this, we first introduce right-equivalence for
QP-representations.

\begin{definition}
\label{def:right-equiv-reps}
Let $(A,S)$ and $(A',S')$ be QPs on the same set of vertices, and
let ${\mathcal M}=(M,V)$ (resp.~${\mathcal M}'=(M',V')$) be a
decorated representation of $(A,S)$ (resp. of~$(A',S')$).
A \emph{right-equivalence} between ${\mathcal M}$  and ${\mathcal M}'$
is a triple $(\varphi, \psi, \eta)$, where:
\begin{itemize}
\item $\varphi: R \langle \langle A \rangle \rangle \to
R \langle \langle A' \rangle \rangle$ is a
right-equivalence between $(A,S)$ and $(A',S')$
(see Definition~\ref{def:AS-isomorphism});
\item $\psi: M \to M'$ is a vector space isomorphism such that
$\psi \circ u_M = \varphi(u)_{M'} \circ \psi$ for
$u \in R \langle \langle A \rangle \rangle$;
\item $\eta: V \to V'$ is an isomorphism of $R$-modules.
\end{itemize}
\end{definition}

\begin{remark}
\label{rem:right-equiv-reps}
The usual notion of isomorphism of decorated representations
${\mathcal M}=(M,V)$ and ${\mathcal M}'=(M',V')$
of the same QP $(A,S)$ (namely, that $M$ and $M'$ are isomorphic ${\mathcal P}(A,S)$-modules,
and $V$ and $V'$ are isomorphic $R$-modules) is a special case of
right-equivalence with the $\varphi$-component being the identity.
The right-equivalence seems to be more relevant for
applications to cluster algebras.
To illustrate, consider the \emph{Kronecker quiver} $Q$ with two
vertices $1$ and $2$, and two arrows $a$ and $b$ from $1$ to $2$.
Since $Q$ has no oriented cycles, it has only one QP $(A,S)$:
the one with $S=0$.
Thus, a decorated representation ${\mathcal M}=(M,V)$ with $V = 0$
is just a usual representation of the quiver $Q$, i.e., it
consists of two vector spaces $M_1$ and $M_2$, and two linear maps
$a$ and $b$ from $M_1$ to $M_2$.
Such representations were classified by Kronecker.
In particular, he proved that, for every $n \geq 1$, the
isomorphism classes of indecomposable $Q$-representations
with $\dim M_1 = \dim M_2 = n$ are naturally parameterized
by the projective line.
However all these representations are right-equivalent to each other.
This is more in sync with the fact that, as discussed in
\cite{calzel}, all these representations give rise to the same
element of the corresponding cluster algebra.
\end{remark}

We now present a representation-theoretic extension of
Theorem~\ref{th:trivial-reduced-splitting}.
Let $\mathcal M = (A,S,M,V)$ be a QP-representation,
and let $\varphi: R \langle \langle A_{\rm red} \oplus C \rangle \rangle \to
R \langle \langle A \rangle \rangle$ be a right-equivalence
of the QPs $(A_{\rm red}, S_{\rm red}) \oplus (C, T)$
and $(A,S)$, where $(A_{\rm red}, S_{\rm red})$ is a reduced QP,
and $(C,T)$ is a trivial QP,
see Theorem~\ref{th:trivial-reduced-splitting}.
We define a $R \langle \langle A_{\rm red} \rangle \rangle$-module
$M'$ by setting $M'=M$ as a $K$-vector space with the action of
$R \langle \langle A_{\rm red} \rangle \rangle$ given by
$u_{M'} = \varphi(u)_M$.
In view of Proposition~\ref{pr:jacobian-algebra-invariant},
this makes a quadruple $\mathcal M_{\rm red} = (A_{\rm red}, S_{\rm red}, M',V)$
a well-defined QP-representation.

\begin{definition}
\label{def:reduced-part-rep}
We call the QP-representation $\mathcal M_{\rm red}$ given by the above construction
the \emph{reduced part} of $\mathcal M$.
\end{definition}

This terminology is justified by the following.

\begin{proposition}
\label{pr:rep-splitting}
The right-equivalence class of $\mathcal M_{\rm red}$
is determined by the right-equivalence class of $\mathcal M$.
\end{proposition}

\begin{proof}
To get more in sync with the notation in
Proposition~\ref{pr:cancel-trivial}, we change our notation
a little bit and assume that $\mathcal M$ is a decorated
representation of a QP $(A \oplus C, S+T)$, where $(A,S)$
is a reduced QP, and $(C,T)$ a trivial one.
Let $\varphi$ be an auto-right-equivalence
of $(A \oplus C, S + T)$, that is, an algebra
automorphism of $R \langle \langle A \oplus C \rangle \rangle$
such that $\varphi(S + T)$ is cyclically
equivalent to $S + T$.
To prove Proposition~\ref{pr:rep-splitting}, it suffices to show
the following:
\begin{eqnarray}
\label{eq:phi'}
&\text{there exists an algebra automorphism $\varphi'$ of
$R \langle \langle A \rangle \rangle$ such that}\\
\nonumber
&\text{$\varphi'(S)$ is cyclically
equivalent to $S$, and
$\varphi'(u)_M = \varphi(u)_M$ for
$u \in R \langle \langle A \rangle \rangle$.}
\end{eqnarray}

As in the proof of Proposition~\ref{pr:cancel-trivial},
let~$L$ denote the closure of the two-sided ideal in
$R \langle \langle A \oplus C \rangle \rangle$ generated by~$C$.
Recall from \eqref{eq:L-splits} that we have
$$R \langle \langle A \oplus C \rangle \rangle =
R \langle \langle A \rangle \rangle \oplus L,$$
and
$$J(S+T) = J(S) \oplus L$$
(in the last equality, $J(S)$ is understood as the Jacobian ideal
of~$S$ in $R \langle \langle A \rangle \rangle$).
In particular, we have $u_M = 0$ for $u \in L$.

Let $\psi: R\langle \langle A \rangle \rangle \to R\langle \langle A \rangle \rangle$
denote the restriction to $R\langle \langle A \rangle \rangle$
of the composition $p \circ \varphi$, where $p$ is the projection of
$R\langle \langle A \oplus C \rangle \rangle$ onto $R\langle \langle A \rangle \rangle$ along $L$.
Then we have $\psi(u)_M = \varphi(u)_M$ for
$u \in R\langle \langle A \rangle \rangle$.

According to \eqref{eq:psi-does-the-job},
$\psi$ is an algebra automorphism
of $R\langle \langle A \rangle \rangle$.
Furthermore, using \eqref{eq:psi-does-the-job} in conjunction with
Proposition~\ref{pr:up-to-square}, we conclude that
$J(\psi(S)) = J(S)$, and that there exists an algebra automorphism
$\eta$ of $R\langle \langle A \rangle \rangle$ such that
$\eta(\psi(S))$ is cyclically equivalent to~$S$, and
$\eta(u) - u \in J(S)$ for $u \in R\langle \langle A \rangle \rangle$.
Taking $\varphi' = \eta \circ \psi$, we see that
$$\varphi'(u)_M = \eta(\psi(u))_M = \psi(u)_M = \varphi(u)_M$$
for $u \in R\langle \langle A \rangle \rangle$.
Thus $\varphi'$ satisfies all the conditions in \eqref{eq:phi'},
and we are done.
\end{proof}

We turn to the definition of \emph{mutations} for a
QP-representation $\mathcal M = (A,S,M,V)$.
We fix a vertex~$k$ satisfying \eqref{eq:no-2-cycles-thru-k}, and let $a_1, \dots, a_s$ be all arrows with
$h(a_p) = k$, and $b_1, \dots, b_t$ be all arrows with
$t(b_q) = k$.
We denote
\begin{equation}
\label{eq:in-out}
M_{\rm in} = \bigoplus_{p = 1}^s M_{t(a_p)}, \quad
M_{\rm out} = \bigoplus_{q = 1}^t M_{h(b_q)}.
\end{equation}
Let $\alpha = \alpha_M: M_{\rm in}  \to M_k$ and $\beta = \beta_M: M_k \to M_{\rm out} $ be
$K$-linear maps given in the matrix form by
\begin{equation}
\label{eq:alpha-beta}
\alpha  = \begin{pmatrix}
a_1 & a_2 & \cdots & a_s\end{pmatrix}, \quad
\beta = \begin{pmatrix}
b_1\\
b_2\\
\vdots\\
b_t
\end{pmatrix}.
\end{equation}
We also introduce a $K$-linear map $\gamma = \gamma_M: M_{\rm out} \to M_{\rm in} $ as follows.
Replacing $S$ if necessary by a cyclically equivalent potential,
we may assume that $S \in {R\langle \langle A \rangle \rangle}_{\hat k, \hat k}$
(see \eqref{eq:k-excluded}).
We identify ${R\langle \langle A \rangle \rangle}_{\hat k, \hat k}$
with $R\langle \langle {\widetilde A}_{\hat k, \hat k} \rangle \rangle$
as in Lemma~\ref{lem:[]-isomorphism}.
This allows us to define the component
$\gamma_{p,q}: M_{h(b_q)} \to M_{t(a_p)}$ of $\gamma$ by
setting
\begin{equation}
\label{eq:gamma-mu-nu}
\gamma_{p,q} = \partial_{[b_q a_p]} S.
\end{equation}
Thus, we have constructed the following triangle of linear maps:
\begin{equation}
\label{eq:triangle}
\xymatrix{
& M_k\ar@<.5ex>[rd]^{\beta} &\\
M_{\rm in}
\ar@<.5ex>[ru]^{\alpha} & & M_{\rm out}\ar@<.5ex>[ll]^{\gamma}
}.
\end{equation}

\begin{lemma}
\label{lem:triangle-compositions}
We have $\alpha \gamma = 0$ and $\gamma \beta = 0$.
\end{lemma}

\begin{proof}
The $q$-th component of $\alpha \gamma$ is equal to
$$\sum_{p} a_p \partial_{[b_q a_p]}S = \partial_{b_q}S;$$
therefore, $\alpha \gamma = 0$, since $M$ is a representation of ${\mathcal P}(A,S)$.
Similarly, the $p$-th component of $\gamma \beta$ is equal to
$$\sum_{q} (\partial_{[b_q a_p]}S) b_q = \partial_{a_p}S,$$
implying that $\gamma \beta = 0$.
\end{proof}

Now let $(\widetilde A, \widetilde S)$ be given by
\eqref{eq:tilde-A} and \eqref{eq:mu-k-S}.
We associate to a QP-representation $\mathcal M = (A,S,M,V)$
the QP-representation $\widetilde \mu_k ({\mathcal M})=
(\widetilde A, \widetilde S,\overline M, \overline V)$ defined as follows.
First, we set
\begin{equation}
\label{eq:M-unchanged}
\overline M_i = M_i, \quad  \overline V_i = V_i \quad (i \neq k).
\end{equation}
We define $\overline M_k$ and $\overline V_k$ by
\begin{equation}
\label{eq:new-Mk}
\overline M_k = \frac{\ker \gamma}{{\rm im}\ \beta} \oplus
{\rm im}\ \gamma \oplus \frac{\ker \alpha}{{\rm im}\ \gamma}
\oplus V_k, \quad  \overline V_k =
\frac{\ker \beta}{\ker \beta \cap {\rm im}\ \alpha}\ .
\end{equation}

We now define the action on $\overline M$ of all arrows in $\widetilde A$
(recall that they are given by \eqref{eq:arrows-thru-k-reversed}
and \eqref{eq:arrows-not-thru-k}).
Thus, for every such arrow~$c$, we need to define a linear map
$c_{\overline M}: \overline M_{t(c)} \to \overline M_{h(c)}$.

First, we set
$$c_{\overline M} = c_{M}$$
for every arrow~$c$
not incident to~$k$, and
$$[b_q a_p]_{\overline M} =  {(b_q a_p)}_M$$
for all $p$ and $q$.

To define the action of the remaining arrows $a_p^\star$ and
$b_q^\star$, we assemble them into the operators
$$\overline \alpha  = \begin{pmatrix}
b_1^\star & b_2^\star & \cdots & b_t^\star \end{pmatrix}$$
and
$$\overline \beta = \begin{pmatrix}
a_1^\star \\
a_2^\star \\
\vdots\\
a_s^\star
\end{pmatrix}$$
in the same way as in \eqref{eq:alpha-beta}.
Thus, we need to define linear maps
$$
\overline \alpha: M_{\rm out} = \overline M_{\rm in}  \to  \overline M_k,
\quad \overline \beta:  \overline M_k \to \overline M_{\rm out} = M_{\rm in}.$$
We will use the following notational convention: whenever we have
a pair $U_1 \subseteq U_2$ of vector spaces,
denote by $\iota:U_1 \to U_2$ the inclusion map, and by $\pi: U_2 \to U_2/U_1$
the natural projection.
We now introduce the following \emph{splitting data}:
\begin{align}
\label{eq:rho}
& \text{Choose a linear map $\rho: M_{\rm out} \to \ker \gamma$
such that $\rho \iota = {\rm id}_{\ker \gamma}$.}\\
\label{eq:sigma}
&\text{Choose a linear map $\sigma: \ker \alpha / {\rm im}\ \gamma
\to \ker \alpha$ such that $\pi \sigma = {\rm id}_{\ker \alpha/ {\rm im}\ \gamma}$.}
\end{align}
Then we define:
\begin{equation}
\label{eq:alpha-beta-mutated}
\overline \alpha =
\begin{pmatrix}
- \pi \rho\\
- \gamma\\
0\\
0
\end{pmatrix}, \quad
\overline \beta  =
\begin{pmatrix}
0 & \iota & \iota \sigma & 0\end{pmatrix}.
\end{equation}

Having defined the action of all arrows in $\widetilde A$ on $\overline M$, we can view
$\overline M$ as a module over the path
algebra $R \langle \widetilde A \rangle$.
The property that $M$ is annihilated by ${\mathfrak m}(A)^n$ for $n \gg
0$ implies that $\overline M$ is annihilated by $\widetilde A^n$ for $n \gg 0$.
This allows us to view $\overline M$ as a module over the completed path
algebra $R \langle \langle \widetilde A \rangle \rangle$.

\begin{proposition}
\label{pr:mutation-well-defined}
The above definitions make $\widetilde \mu_k ({\mathcal M})=(\overline M, \overline V)$
a decorated representation of $(\widetilde A, \widetilde S)$.
\end{proposition}

\begin{proof}
We only need to show that
$(\partial_c \widetilde S)_{\overline M} = 0$ for every arrow~$c$
in $\widetilde A$.
If~$c$ is not incident to~$k$, the desired statement follows from
\eqref{eq:[]-partial-c}.
If~$c$ is one of the arrows $[b_q a_p]$, then, in view of
\eqref{eq:partial-ba} and \eqref{eq:gamma-mu-nu}, it is enough to
show that
$$a_p^\star b_q^\star  + \gamma_{p,q} = 0$$
for all $p$ and $q$.
In other words, we need to show that
$\overline \beta \overline \alpha = - \gamma$;
But this follows at once by multiplying the row and column given
by \eqref{eq:alpha-beta-mutated}.

It remains to show that $(\partial_{a_p^\star} \widetilde S)_{\overline M} = 0$
and $(\partial_{b_q^\star} \widetilde S)_{\overline M} = 0$ for all $p$ and $q$.
We first deal with the first equality.
Remembering \eqref{eq:mu-k-S} and \eqref{eq:Tk}, we see that
$$(\partial_{a_p^\star} \widetilde S)_{\overline M} =
(\sum_q b_q^\star [b_q a_p])_{\overline M} =
(\sum_q (b_q^\star)_{\overline M} (b_q)_M) (a_p)_M.$$
Thus it suffices to show that
$$\sum_q (b_q^\star)_{\overline M} (b_q)_M = 0,$$
or equivalently, $\overline \alpha \beta = 0$.
In view of \eqref{eq:alpha-beta-mutated}, we have
$$\overline \alpha \beta = \begin{pmatrix}
- \pi \rho \beta\\
- \gamma \beta\\
0\\
0
\end{pmatrix} = 0,$$
as desired (the equality $\pi \rho \beta = 0$ is immediate from
the definitions, while $\gamma \beta = 0$ by
Lemma~\ref{lem:triangle-compositions}).

The remaining equality $(\partial_{b_\nu^\star} \widetilde S)_{\overline M} = 0$
is proved in a similar way.
We have
$$(\partial_{b_q^\star} \widetilde S)_{\overline M} =
(\sum_p [b_q a_p] a_p^\star )_{\overline M} =
(b_q)_M \sum_p (a_p)_M (a_p^\star)_{\overline M}.$$
Thus, it suffices to observe that
$$\alpha \overline \beta = \begin{pmatrix}
0 & \alpha \iota & \alpha \iota \sigma & 0\end{pmatrix} = 0.$$
\end{proof}

\begin{example}
\label{ex:sources-sinks}
Let us illustrate the definition of the operation $\widetilde \mu_k$
by a special case where the vertex~$k$ is a sink or a source.
First suppose~$k$ is a sink, that is, there are no arrows~$b$
with $t(b) = k$.
Then we have $M_{\rm out} = \{0\}$, hence $\beta = 0$ and $\gamma = 0$.
Thus, \eqref{eq:new-Mk} simplifies to
$$\overline M_k = \ker \alpha \oplus V_k, \quad  \overline V_k =
{\rm coker}\ \alpha\ .$$
The arrow span $\widetilde A$ is obtained from~$A$ by reversing
all arrows~$a$ with $h(a) = k$, that is, replacing every such
arrow~$a$ with $a^\star$.
Thus, $k$ becomes a source for~$\widetilde A$, hence
$\overline M_{in} = \{0\}$ and $\overline \alpha = 0$.
The choice of splitting data \eqref{eq:rho} and
\eqref{eq:sigma} becomes immaterial, and the second equality in
\eqref{eq:alpha-beta-mutated} simplifies to
$$\overline \beta  =
\begin{pmatrix}
\iota & 0\end{pmatrix}.$$
Note that we have $\widetilde A_{\hat k, \hat k} = A_{\hat k, \hat k}$,
and the potential
$\widetilde S \in R\langle \langle \widetilde A_{\hat k, \hat k}\rangle \rangle$
is naturally identified with~$S$.

The case where~$k$ is a source is completely similar.
In this case we have  $M_{\rm in} = \{0\}$, hence $\alpha = 0$ and $\gamma = 0$.
It follows that
$$\overline M_k = {\rm coker}\  \beta \oplus V_k, \quad  \overline V_k =
\ker \beta\ ,$$
and the map $\overline \alpha: \overline M_{\rm in} = M_{\rm out}
\to \overline M_k$ is given by
$$\overline \alpha  =
\begin{pmatrix}
-\pi\\ 0\end{pmatrix}.$$

In both cases, $\widetilde \mu_k$ coincides with the ``extended reflection functor"
introduced in \cite{mrz}; furthermore, if we ignore the decorations (and the potentials),
it becomes the classical Bernstein-Gelfand-Ponomarev reflection
functor at~$k$, see \cite{bgp}.
\end{example}

Now we return to the case of an arbitrary vertex~$k$.
\begin{proposition}
\label{pr:independence-of-splitting}
The isomorphism class of the decorated representation $\widetilde \mu_k ({\mathcal M})$ does not depend
on the choice of the splitting data \eqref{eq:rho} -- \eqref{eq:sigma}.
\end{proposition}

\begin{proof}
We have the following freedom in choosing the splitting data: one
can replace $\rho: M_{\rm out} \to \ker \gamma$
with $\rho' = \rho + \xi \gamma$ for some linear map
$\xi: {\rm im}\ \gamma \to \ker \gamma$, and replace
$\sigma: \ker \alpha / {\rm im}\ \gamma \to \ker \alpha$
by $\sigma' = \sigma + \eta$ for some linear map
$\eta: \ker \alpha / {\rm im}\ \gamma \to {\rm im}\ \gamma$.
Let $\overline \alpha'$ and $\overline \beta'$ be the
maps obtained by replacing $\rho$ with $\rho'$, and
$\sigma$ with $\sigma'$ in \eqref{eq:alpha-beta-mutated}.
It is enough to show that $\psi \overline \alpha =
\overline \alpha'$ and $\overline \beta' \psi = \overline \beta$
for some linear automorphism $\psi: \overline M_k \to  \overline M_k$.
Decomposing $\overline M_k$ as in \eqref{eq:new-Mk}, we define
$\psi$ as the block-triangular matrix
$$
\psi = \begin{pmatrix}
I & \pi \xi & 0 & 0\\
0 & I & -\eta & 0\\
0 & 0 & I & 0\\
0 & 0 & 0 & I
\end{pmatrix},
$$
where $I$ stands for the identity transformation.
The invertibility of $\psi$ is obvious, and the desired
equalities $\psi \overline \alpha =
\overline \alpha'$ and $\overline \beta' \psi = \overline \beta$
are checked by direct matrix multiplication.
\end{proof}

\begin{proposition}
\label{pr:mutation-acts-on-right-equiv}
The right-equivalence class of the representation $\widetilde \mu_k ({\mathcal M})$
is determined by the right-equivalence class of~${\mathcal M}$.
\end{proposition}

\begin{proof}
Let
$\varphi$ be an automorphism of
$R \langle \langle A \rangle \rangle$,
and let ${\mathcal M}'=(A,\varphi(S),M',V')$ be
the QP-representation defined as follows:
$V' = V$ and $M' = M$ as $R$-modules,
while the
$R \langle \langle A \rangle \rangle$-actions
in $M$ and $M'$ are related by
\begin{equation}
\label{eq:M-M'}
u_M = \varphi(u)_{M'} \quad (u \in R \langle \langle A \rangle \rangle)
\end{equation}
(note that \eqref{eq:M-M'} indeed defines a representation of
$(A,\varphi(S))$ in view of
Proposition~\ref{pr:automorphism-respects-jacobian}).
To prove Proposition~\ref{pr:mutation-acts-on-right-equiv}, it
suffices to show that the representations $\widetilde \mu_k({\mathcal M})$
and $\widetilde \mu_k({\mathcal M}')$ are right-equivalent.

We retain all the above notation related to ${\mathcal M}$ and
$\widetilde \mu_k({\mathcal M})$; in particular, $\alpha, \beta$ and $\gamma$
stand for the linear maps in the triangle \eqref{eq:triangle}.
Let $\alpha', \beta'$ and $\gamma'$ denote the corresponding maps
for the representation ${\mathcal M}'$.
Our first order of business is to relate these maps to
$\alpha, \beta$ and $\gamma$.

We can write the action of
$\varphi$ on the arrows $a_1, \dots, a_s$ as follows:
\begin{equation}
\label{eq:phi-on-a-2}
\begin{pmatrix}
\varphi(a_1) & \varphi(a_2) & \cdots & \varphi(a_s)\end{pmatrix}
=
\begin{pmatrix}
a_1 & a_2 & \cdots & a_s\end{pmatrix}
C,
\end{equation}
where $C = C_0 + C_1$ is an invertible $s \times s$ matrix
as in \eqref{eq:phi-on-a}.
Similarly, the action of
$\varphi$ on the arrows $b_1, \dots, b_s$ can be written as
\begin{equation}
\label{eq:phi-on-b-2}
\begin{pmatrix}
\varphi(b_1)\\ \varphi(b_2)\\ \vdots \\ \varphi(b_t)\end{pmatrix}
= D \begin{pmatrix}
b_1 \\ b_2 \\ \vdots \\ b_t\end{pmatrix},
\end{equation}
where $D = D_0 + D_1$ is an invertible $t \times t$ matrix
as in \eqref{eq:phi-on-b}.
Therefore we have
\begin{align}
\label{eq:alpha-alpha'}
\alpha  &= \begin{pmatrix}
a_1 & a_2 & \cdots & a_s\end{pmatrix}_M =
\begin{pmatrix}
\varphi(a_1) & \varphi(a_2) & \cdots &
\varphi(a_s)\end{pmatrix}_{M'}\\
\nonumber
&= (\begin{pmatrix}
a_1 & a_2 & \cdots & a_s\end{pmatrix} C)_{M'} =
\alpha' C_{M'},
\end{align}
and similarly,
\begin{equation}
\label{eq:beta-beta'}
\beta   = D_{M'} \beta';
\end{equation}
here $C_{M'}$ (resp. $D_{M'}$) is understood as an
$R$-module automorphism of $M'_{\rm in} = M_{\rm in}$
(resp. of $M'_{\rm out} = M_{\rm out}$).

Turning to the maps $\gamma$ and $\gamma'$, we claim that they are
related by
\begin{equation}
\label{eq:gamma-gamma'}
\gamma' = C_{M'} \gamma D_{M'}.
\end{equation}
To see this, we use \eqref{eq:gamma-mu-nu} and \eqref{eq:chain-rule} to write
\begin{align}
\label{eq:gamma'-pq-1}
\gamma'_{p,q} &= (\partial_{[b_q a_p]} \varphi(S))_{M'}\\
\nonumber
& = (\sum_{c} \Delta_{[b_q a_p]} (\varphi(c)) \square
\varphi(\partial_{c}S))_{M'},
\end{align}
where the sum is over all arrows~$c$ in
${\widetilde A}_{\hat k, \hat k}$.
If~$c$ is one of the arrows in~$A$ then by \eqref{eq:M-M'} we have
$$\varphi(\partial_{c}S)_{M'} = (\partial_{c}S)_{M} = 0;$$
remembering the definition \eqref{eq:square}, we see that~$c$ does
not contribute to \eqref{eq:gamma'-pq-1}.
Thus, we have
\begin{equation}
\label{eq:gamma'-pq-2}
\gamma'_{p,q} = (\sum_{p', q'} \Delta_{[b_q a_p]} (\varphi(b_{q'} a_{p'})) \square
\varphi(\partial_{[b_{q'} a_{p'}]}S))_{M'}.
\end{equation}
Remembering \eqref{eq:delta-xi}, and using \eqref{eq:phi-on-a-2}
and \eqref{eq:phi-on-b-2}, we see that the summand with $(p',q') = (p,q)$
in \eqref{eq:gamma'-pq-2} contains among its terms the $(p,q)$-entry of
the matrix
$$(C \varphi(\partial_{[b_{q} a_{p}]}S) D)_{M'} =
C_{M'} \gamma D_{M'}.$$
Thus, to prove \eqref{eq:gamma-gamma'}, it remains to show that
the rest of the terms in \eqref{eq:gamma'-pq-2} add up to~$0$.
Again using the definitions \eqref{eq:delta-xi} and
\eqref{eq:square}, we can rewrite the rest of the sum in
\eqref{eq:gamma'-pq-2} as $S_1 + S_2$, where
$$S_1 = (\sum_{p', q'} \Delta_{[b_q a_p]} (\varphi(b_{q'})) \square
\varphi(a_{p'} \cdot \partial_{[b_{q'} a_{p'}]}S))_{M'},$$
$$S_2 = (\sum_{p', q'} \Delta_{[b_q a_p]} (\varphi(a_{p'})) \square
\varphi(\partial_{[b_{q'} a_{p'}]}S \cdot b_{q'}))_{M'}.$$
It remains to observe that
$$S_1 = (\sum_{q'} \Delta_{[b_q a_p]} (\varphi(b_{q'})) \square
\varphi(\partial_{b_{q'}}S))_{M'} = 0$$
since $\varphi(\partial_{b_{q'}}S)_{M'} = (\partial_{b_{q'}}S)_{M} =
0$;
and similarly,
$$S_2 = (\sum_{p'} \Delta_{[b_q a_p]} (\varphi(a_{p'})) \square
\varphi(\partial_{a_{p'}}S))_{M'} = 0.$$

In view of \eqref{eq:alpha-alpha'}, \eqref{eq:beta-beta'} and
\eqref{eq:gamma-gamma'}, we have
\begin{eqnarray}
\nonumber
\ker \alpha =& C_{M'}^{-1}(\ker \alpha'),
\quad {\rm im}\ \alpha &= {\rm im}\ \alpha',\\
\label{eq:kernels-images}
\ker \beta =& \ker \beta',
\quad \quad \quad \,\,\,{\rm im}\ \beta &=  D_{M'}({\rm im}\ \beta'),\\
\nonumber
\ker \gamma =& D_{M'}(\ker \gamma'),
\quad {\rm im}\ \gamma &= C_{M'}^{-1}({\rm im}\ \gamma').
\end{eqnarray}
Recall that the spaces $\overline M$ and $\overline V$
in the decorated representation $\widetilde \mu_k ({\mathcal M})=(\overline M, \overline V)$ of
$(\widetilde A, \widetilde S)$ are given by \eqref{eq:M-unchanged}
and \eqref{eq:new-Mk}.
We express the decorated representation
$\widetilde \mu_k ({\mathcal M}') =(\overline {M'}, \overline {V'})$ of
$(\widetilde A, \widetilde {\varphi(S)})$ in the same way, with
the maps $\alpha, \beta$ and $\gamma$ replaced by $\alpha', \beta'$ and $\gamma'$.
In particular, we have $\overline {V'} = \overline {V}$, and
$\overline {M'}_i = \overline {M}_i = M_i$ for $i \neq k$.
To specify the actions of $R \langle \langle A \rangle \rangle$ in
$\overline M$ and $\overline {M'}$, we need to choose the splitting data
$(\rho, \sigma)$ and $(\rho', \sigma')$ as in
\eqref{eq:rho} and \eqref{eq:sigma}.
Note that, in view of \eqref{eq:kernels-images}, we can choose
\begin{equation}
\label{eq:twisted-splitting-data}
\rho' = D_{M'}^{-1} \rho D_{M'}, \quad
\sigma' = C_{M'} \sigma C_{M'}^{-1};
\end{equation}
here with some abuse of notation we use the same notation
$C_{M'}^{-1}$ for the isomorphism $\ker \alpha' \to \ker \alpha$
and the induced isomorphism $\ker \alpha' / {\rm im}\ \gamma'
\to \ker \alpha / {\rm im}\ \gamma$.

Everything is now in place for defining the desired
right-equivalence $(\widehat \varphi, \psi, \eta)$
between $\widetilde \mu_k ({\mathcal M})$  and $\widetilde \mu_k ({\mathcal M}')$
(see Definition \ref{def:right-equiv-reps}).
First of all, we define
$\widehat \varphi: R \langle \langle \widetilde A \rangle \rangle \to
R \langle \langle \widetilde A \rangle \rangle$ as the
right-equivalence between $(\widetilde A, \widetilde S)$ and
$(\widetilde {A},\widetilde {\varphi(S)}$ constructed in
the proof of Lemma~\ref{lem:extension}.
In particular, we have
$$
\begin{pmatrix}
{\widehat \varphi}(a_1^\star)\\
{\widehat\varphi}(a_2^\star)\\
\vdots\\
{\widehat\varphi}(a_s^\star)
\end{pmatrix}
= C^{-1}\begin{pmatrix}
a_1^\star\\
a_2^\star\\
\vdots\\
a_s^\star
\end{pmatrix},
\quad
\begin{pmatrix}
{\widehat \varphi}(b_1^\star) &
{\widehat\varphi}(b_2^\star) &
\cdots &
{\widehat\varphi}(b_t^\star)
\end{pmatrix}
=\begin{pmatrix}
b_1^\star &
b_2^\star &
\cdots &
b_t^\star
\end{pmatrix} D^{-1}.
$$

Next we define $\psi: \overline M \to \overline {M'}$ as the
identity map on $\oplus_{i \neq k} \overline M_i =
\oplus_{i \neq k} M_i = \oplus_{i \neq k} \overline {M'}_i$, and
the restriction $\psi|_{\overline M_k} : \overline M_k \to \overline {M'}_k$ given by the
block-diagonal matrix
\begin{equation}
\label{eq:psi-Mk}
\psi|_{\overline M_k} = \begin{pmatrix}
D_{M'}^{-1} & 0 & 0 & 0\\
0 & C_{M'} & 0 & 0\\
0 & 0 & C_{M'} & 0\\
0 & 0 & 0 & I
\end{pmatrix}
\end{equation}
(this is well-defined in view of \eqref{eq:kernels-images}).
Finally, we define $\eta: \overline V \to \overline {V'}$ simply
as the identity map.

The only thing to check is the equality
$\psi \circ c_{\overline M} = \widehat \varphi(c)_{\overline {M'}} \circ \psi$
for any arrow $c \in \widetilde A$.
And the only case that may require some consideration is when $c$ is one of the arrows
$a_p^\star$ or $b_q^\star$.
Unraveling the definitions, it suffices to show that
$$\overline \beta = C_{M'}^{-1} \overline {\beta'} \circ
\psi|_{\overline M_k}, \quad
\psi|_{\overline M_k} \circ \overline \alpha
= \overline {\alpha'} D_{M'}^{-1}.$$
But this is an immediate consequence of the definitions
\eqref{eq:psi-Mk} and \eqref{eq:alpha-beta-mutated}
(we also need an analogue of \eqref{eq:alpha-beta-mutated}
for the maps  $\overline {\beta'}$ and $\overline {\alpha'}$,
using the splitting data \eqref{eq:twisted-splitting-data}).
This completes the proof of
Proposition~\ref{pr:mutation-acts-on-right-equiv}.
\end{proof}

Note that in the above treatment of the operation
$\mathcal M \mapsto \widetilde \mu_k (\mathcal M)$
for a QP-representation $\mathcal M = (A,S,M,V)$,
the QP $(A,S)$ was not assumed to be reduced.
Recall from Proposition~\ref{pr:rep-splitting}
that we have a well-defined operation
$\mathcal M \mapsto {\mathcal M}_{\rm red}$
on (right-equivalence classes of) QP-representations.
The following property is immediate from definitions.

\begin{proposition}
\label{pr:mu-tilde-of-reduced-rep}
Let $(A,S)$ be a QP satisfying \eqref{eq:no-2-cycles-thru-k}.
Then, for every representation $\mathcal M$ of $(A,S)$,
the representation $\widetilde \mu_k (\mathcal M)_{\rm red}$ is
right-equivalent to $\widetilde \mu_k (\mathcal M_{\rm red})_{\rm red}$.
\end{proposition}

Recall that, according to Corollary~\ref{cor:mutations-respect-isom}
and Definition~\ref{def:reduced-mutation},
the correspondence $\widetilde \mu_k: (A,S) \mapsto
\widetilde \mu_k(A,S) = (\widetilde A, \widetilde S)$
gives rise to the \emph{mutation}
$(A,S) \mapsto \mu_k(A,S) = (\overline A, \overline S)$,
which is a well-defined bijective transformation on the set of
right-equivalence classes of reduced QPs satisfying
\eqref{eq:no-2-cycles-thru-k}.
Here $(\overline A, \overline S)$ is the reduced part
of $(\widetilde A, \widetilde S)$.
Now for every QP-representation $\mathcal M = (A,S,M,V)$ of a
reduced QP $(A,S)$ we define
\begin{equation}
\label{eq:rep-mutation}
\mu_k (\mathcal M) = \widetilde \mu_k (\mathcal M)_{\rm red};
\end{equation}
thus, $\mu_k (\mathcal M)$ is a decorated representation
$(\overline A, \overline S, \overline M, \overline V)$ of a reduced QP
$(\overline A, \overline S)$.
Combining Propositions~\ref{pr:rep-splitting} and
\ref{pr:mutation-acts-on-right-equiv}, we obtain
the following important corollary.

\begin{corollary}
\label{cor:rep-mutations-respect-right-equiv}
The correspondence $\mathcal M \mapsto \mu_k (\mathcal M)$
is a well-defined transformation on the
set of right-equivalence classes of decorated representations
of reduced QPs satisfying
\eqref{eq:no-2-cycles-thru-k}.
\end{corollary}

We refer to the transformation $\mathcal M \mapsto \mu_k (\mathcal M)$
in Corollary~\ref{cor:rep-mutations-respect-right-equiv}
as the \emph{mutation at vertex}~$k$.
With some abuse of terminology, we will talk about mutations of
decorated representations (rather than their right-equivalence classes).

The following result naturally extends
Theorem~\ref{th:mutation-involutive}.

\begin{theorem}
\label{th:rep-mutation-involutive}
The mutation $\mu_k$ of decorated representations is an
involution; that is, for every decorated representation
$\mathcal M$ of a reduced QP $(A,S)$
satisfying \eqref{eq:no-2-cycles-thru-k},
the decorated representation $\mu_k^2 (\mathcal M)$
of a QP $\mu_k^2(A,S)$ is right-equivalent to $\mathcal M$.
\end{theorem}

\begin{proof}
In view of Proposition~\ref{pr:mu-tilde-of-reduced-rep},
$\mu_k^2 (\mathcal M)$ is right-equivalent to
$\widetilde \mu_k^2 (\mathcal M)_{\rm red}$.
Therefore, is suffices to show that
$\widetilde \mu_k^2 (\mathcal M)_{\rm red}$ is right-equivalent
to $\mathcal M$.

We write the QP-representation $\widetilde \mu_k^2 (\mathcal M)$ as
$\widetilde \mu_k^2 (\mathcal M)=
(\widetilde {\widetilde A}, \widetilde {\widetilde S},
\overline {\overline M}, \overline {\overline V})$.
The QP $(\widetilde {\widetilde A}, \widetilde {\widetilde S})$
is given by \eqref{eq:tilde-tilde-A} and \eqref{eq:tilde-tilde-S}.
In particular, there is a natural embedding of~$A$ into
$\widetilde {\widetilde A}$ identifying $A$ with the reduced part
$\widetilde {\widetilde A}_{\rm red}$.
Furthermore, as shown in the proof of
Theorem~\ref{th:mutation-involutive}, an automorphism of
$R \langle \langle \widetilde {\widetilde A} \rangle \rangle$
that establishes the right-equivalence in \eqref{eq:tilde-twice}
can be chosen so that it restricts to an automorphism
$\varphi_0: R \langle \langle A \rangle \rangle \to R \langle \langle A \rangle \rangle$
acting as follows:
\begin{align}
\label{eq:phi0}
&\text{$\varphi_0$ multiplies each of the arrows $b_1, \dots, b_t$ by $-1$,}\\
\nonumber
&\text{and fixes the rest of
the arrows in~$A$.}
\end{align}
In view of Definition~\ref{def:reduced-part-rep},
the QP-representation
$\mu_k^2 (\mathcal M) = \widetilde \mu_k^2 (\mathcal M)_{\rm red}$
can be realized as $(A,S,M',V')$, where $M' = \overline {\overline M}$
and $V' = \overline {\overline V}$ as vector spaces, and the
action of $R \langle \langle A \rangle \rangle$ in $M'$ is given by
\begin{equation}
\label{eq:mu-square-action}
u_{M'} = \varphi_0(u)_{\overline {\overline M}} \quad \quad (u \in R \langle \langle A \rangle \rangle).
\end{equation}
To prove Theorem~\ref{th:rep-mutation-involutive}, it suffices to
show that the decorated representation $(M',V')$ of $(A,S)$ is
isomorphic to $(M,V)$.

We first compute $M' = \overline {\overline M}$
and $V' = \overline {\overline V}$ as vector spaces.
According to \eqref{eq:M-unchanged}, we have
$$M'_i = \overline {\overline M}_i = {\overline M}_i = M_i, \quad
V'_i = \overline {\overline V}_i = {\overline V}_i = V_i$$
for all $i \neq k$.
As for the spaces $M'_k$ and $V'_k$, they
are given as in \eqref{eq:new-Mk}, with the maps $\alpha, \beta$,
and $\gamma$ replaced by $\overline \alpha, \overline \beta$,
and $\overline \gamma$, respectively.
Recall that $\overline \alpha$ and $\overline \beta$ are given by
\eqref{eq:alpha-beta-mutated}.
As for $\overline \gamma$, by applying the definition \eqref{eq:gamma-mu-nu}
to the potential $\widetilde S$ given by \eqref{eq:mu-k-S} and
\eqref{eq:Tk}, we see that
\begin{equation}
\label{eq:overline-gamma}
\overline \gamma = \beta \alpha.
\end{equation}
As a direct consequence of the definitions, we conclude that
\begin{eqnarray}
\nonumber
&\ker \overline \alpha = {\rm im}\ \beta,
\quad {\rm im}\ \overline \alpha =
\frac{\ker \gamma}{{\rm im}\ \beta} \oplus
{\rm im}\ \gamma \oplus \{0\}
\oplus \{0\},\\
\label{eq:overline-kernels-images}
&\ker \overline \beta =
\frac{\ker \gamma}{{\rm im}\ \beta} \oplus
\{0\} \oplus \{0\}
\oplus V_k,
\quad \,\,\,{\rm im}\ \overline \beta = \ker \alpha ,\\
\nonumber
&\ker \overline \gamma = \ker(\beta \alpha),
\quad {\rm im}\ \overline \gamma = {\rm im}\ (\beta \alpha).
\end{eqnarray}
It follows that
$$V'_k = \overline {\overline V}_k =
\frac{\ker \overline \beta}{\ker \overline \beta
\cap {\rm im}\ \overline \alpha} = V_k,$$
and so $V' = V$, as desired.
We also have
$$M'_k  = \frac{\ker(\beta \alpha)}{\ker \alpha} \oplus
{\rm im}\ (\beta \alpha) \oplus \frac{{\rm im}\ \beta}{{\rm im}\ (\beta \alpha)}
\oplus \frac{\ker \beta}{\ker \beta \cap {\rm im}\ \alpha}\ .$$
We now make the following easy observations:
\begin{itemize}
\item the map~$\alpha$ induces an isomorphism
$\ker(\beta \alpha)/\ker \alpha  \to \ker \beta \cap {\rm im}\ \alpha$;
\item the map~$\beta$ induces an isomorphism
${\rm im}\ \alpha/(\ker \beta \cap {\rm im}\ \alpha)
\to {\rm im}\ (\beta \alpha)$;
\item the map~$\beta$ induces an isomorphism
$M_k/(\ker \beta + {\rm im}\ \alpha)
\to {\rm im}\ \beta/{\rm im}\ (\beta \alpha)$.
\item there is a natural isomorphism
$\ker \beta/(\ker \beta \cap {\rm im}\ \alpha) \to
(\ker \beta + {\rm im}\ \alpha)/{\rm im}\ \alpha$.
\end{itemize}
Using these isomorphisms, we can identify~$M'_k$ with
the space
\begin{equation}
\label{eq:M'k-rewritten}
M'_k = (\ker \beta \cap {\rm im}\ \alpha) \oplus
\frac{{\rm im}\ \alpha}{\ker \beta \cap {\rm im}\ \alpha}
\oplus \frac{M_k}{\ker \beta + {\rm im}\ \alpha}
\oplus \frac{\ker \beta + {\rm im}\ \alpha}{{\rm im}\ \alpha}\ .
\end{equation}

To describe the action of $R \langle \langle A \rangle \rangle$ in
$M'$, we need only to describe the maps $\alpha': M_{\rm in} \to M'_k$
and $\beta': M'_k \to M_{\rm out}$ constructed in the same way
as in \eqref{eq:alpha-beta}.
As in \eqref{eq:alpha-beta-mutated}, the definition of $\alpha'$
and $\beta'$ involves splitting data \eqref{eq:rho} - \eqref{eq:sigma}.
In view of \eqref{eq:overline-kernels-images}, in the current situation
the splitting data take the following form:
\begin{align*}
& \text{Choose a linear map $\overline \rho: M_{\rm in} \to \ker (\beta \alpha)$
such that $\overline \rho \iota = {\rm id}_{\ker (\beta \alpha)}$.}\\
&\text{Choose a linear map $\overline \sigma: {\rm im}\ \beta /
{\rm im}\ (\beta \alpha) \to {\rm im}\ \beta$ such that
$\pi \overline \sigma = {\rm id}_{{\rm im}\ \beta /{\rm im}\ (\beta \alpha)}$.}
\end{align*}
Adapting \eqref{eq:alpha-beta-mutated} to the current situation
(in particular, realizing $M'_k$ as in \eqref{eq:M'k-rewritten}),
we see that the maps $\alpha'$ and $\beta'$ take the following
form:
\begin{equation}
\label{eq:alpha'-beta'}
\alpha' =
\begin{pmatrix}
- \alpha \overline \rho\\
- \pi \alpha\\
0\\
0
\end{pmatrix}, \quad
\beta'  =
\begin{pmatrix}
0 & - \beta & - \iota \overline \sigma \beta & 0\end{pmatrix};
\end{equation}
here with some abuse of notation we denote by the same symbol
$\beta$ the two maps ${\rm im}\ \alpha/\ker \beta \cap {\rm im}\ \alpha
\to M_{\rm out}$ and $M_k/(\ker \beta + {\rm im}\ \alpha) \to
{\rm im}\ \beta / {\rm im}\ (\beta \alpha)$ induced by~$\beta$.
Note that the appearance of the minus sign in $\beta'$ is caused
by the minus sign in \eqref{eq:phi0}.

To complete the proof of Theorem~\ref{th:rep-mutation-involutive},
it remains to construct an isomorphism of vector spaces
$\psi: M'_k \to M_k$ such that
\begin{equation}
\label{eq:psi-conditions}
\psi \alpha' = \alpha, \quad \quad \beta \psi = \beta'.
\end{equation}
To do this, notice that the four direct summands in \eqref{eq:M'k-rewritten}
are the factors in the filtration
$$\{0\} \subseteq \ker \beta \cap {\rm im}\ \alpha \subseteq
{\rm im}\ \alpha \subseteq \ker \beta + {\rm im}\ \alpha
\subseteq M_k.$$
Choose some sections
\begin{align*}
&\sigma_1: {\rm im}\ \alpha/(\ker \beta \cap {\rm im}\ \alpha)
\to {\rm im}\ \alpha,\\
&\sigma_2: (\ker \beta + {\rm im}\ \alpha)/{\rm im}\ \alpha
\to (\ker \beta + {\rm im}\ \alpha),\\
&\sigma_3: M_k/(\ker \beta + {\rm im}\ \alpha) \to M_k
\end{align*}
for the three factors of this filtration, so that they satisfy:
$${\rm im}\ \sigma_1 = \alpha (\ker \overline \rho), \quad \quad
{\rm im}\ \sigma_2 \subseteq \ker \beta, \quad \quad
{\rm im}\ (\beta \sigma_3) =  {\rm im}\ \overline \sigma.$$
Now define an an isomorphism $\psi: M'_k \to M_k$ by setting
$$\psi =
\begin{pmatrix}
- \iota & - \iota \sigma_1 & - \iota \sigma_3 & - \iota \sigma_2
\end{pmatrix}.$$
The equalities \eqref{eq:psi-conditions} are checked by a direct
inspection, finishing the proof.
\end{proof}

Note that there is an obvious way to define direct sums for decorated
representations of a given QP $(A,S)$.
Hence we can talk about \emph{indecomposable} QP-representations.
Clearly, the right-equivalence relation respects direct sums and
indecomposability.
It is also immediate from the definitions that any mutation $\mu_k$
of QP-representations sends direct sums to direct sums.
Combining this with Theorem~\ref{th:rep-mutation-involutive}, we
obtain the following corollary.

\begin{corollary}
\label{cor:rep-mutations-respect-idecomposables}
Any mutation $\mu_k$ is an involution on the
set of right-equivalence classes of indecomposable decorated representations
of reduced QPs satisfying
\eqref{eq:no-2-cycles-thru-k}.
\end{corollary}

We call a QP-representation $(A,S,M,V)$ \emph{positive} if $V = \{0\}$.
Thus, indecomposable positive representations are just
indecomposable ${\mathcal P}(A,S)$-modules.
In particular, for every vertex~$k$, the \emph{simple} representation
${\mathcal S}_k(A,S)$ is the indecomposable positive representation of $(A,S)$ such that
$\dim M_i = \delta_{i,k}$.
We denote by ${\mathcal S}_k^-(A,S)$ the indecomposable representation $(M,V)$
of $(A,S)$ such that $M = \{0\}$ and $\dim V_i = \delta_{i,k}$.
We refer to ${\mathcal S}_k^-(A,S)$ as the \emph{negative simple} representation at~$k$.
The following proposition is immediate from the definitions.

\begin{proposition}
\label{pr:positive-negative-mutations}
Any indecomposable QP-representation
is either positive, or negative simple.
If $\mu_k(A,S) = (\overline A, \overline S)$ then we have
\begin{equation}
\label{eq:mutation-simples}
\mu_k({\mathcal S}_k(A,S)) = {\mathcal S}_k^-(\overline A, \overline S),
\quad \mu_k({\mathcal S}_k^-(A,S)) = {\mathcal S}_k(\overline A, \overline S);
\end{equation}
and this is the only mutation that interchanges positive and
negative indecomposable representations.
\end{proposition}

\section{Some three-vertex examples}
\label{sec:band}

In this section we illustrate the action of mutations on
QP-representations by some examples dealing with three-vertex quivers.
All the representations $(M,V)$ considered below will be positive,
i.e., $V = \{0\}$.

\begin{example}
\label{ex:A3-cyclic-triangle}
Let $Q$ be the quiver with three vertices $1,2,3$ and
two arrows $a: 1 \to 2$ and $b: 2 \to 3$:
$$
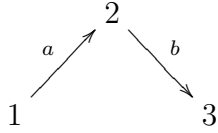
\xymatrix{
& 2\ar^{b}[rd] & \\
1\ar[ru]^a & & 3 
}
$$
Since $Q$ is acyclic, the only QP on it is $(A,0)$.
We have $\mu_2(A,0) = (\overline A, \overline S)$, where
$\overline A$ is the arrow span of the quiver $\overline Q$ given
by
$$
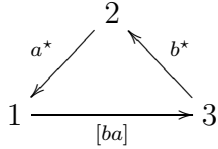
\xymatrix{
& 2\ar_{a^\star}[ld] & \\
1\ar[rr]_{[ba]} & & 3\ar[lu]_{b^\star}
}
$$
and $\overline S = b^\star [ba] a^\star$.
Thus, positive representations of $(A,0)$ are the representations of the quiver~$Q$,
while positive representations of $(\overline A, \overline S)$ are
the representations of the quiver $\overline Q$ satisfying
the relations
\begin{equation}
\label{eq:double-products-0}
b^\star [ba] =   [ba] a^\star =  a^\star b^\star = 0.
\end{equation}
In view of Corollary~\ref{cor:rep-mutations-respect-idecomposables}
and Proposition~\ref{pr:positive-negative-mutations}, the mutation
$\mu_2$ establishes a  bijection between the set of right-equivalence
classes of indecomposable positive representations of $(A,0)$
different from the simple representation ${\mathcal S}_2$, and the
same set for $(\overline A, \overline S)$.
Since $Q$ is a Dynkin quiver of type $A_3$, by Gabriel's theorem,
an indecomposable positive representation $M$ of $(A,0)$ is uniquely
up to an isomorphism determined by its dimension vector
${\bf dim}\, M = (\dim M_1, \dim M_2, \dim M_3)$, and these
dimension vectors are the positive roots of type $A_3$
(note that in this case, the right-equivalence classes are the same
as isomorphism classes).
Computing the images of these representations under $\mu_2$, we obtain the
correspondence between the dimension vectors given in
Table~\ref{tab:A3-cyclic-triangle}.
\begin{table}[ht]
\begin{center}
\begin{tabular}{|c|c|c|c|c|c|c|}
\hline
&&&&&\\[-.1in]
${\bf dim} \, M$ & (1,0,0) & (0,0,1) & (1,1,0) & (0,1,1) & (1,1,1) \\
\hline
&&&&&\\[-.1in]
${\bf dim} \,\mu_2(M)$ & (1,1,0) & (0,1,1) & (1,0,0) & (0,0,1) & (1,0,1) \\[.05in]
\hline
\end{tabular}
\end{center}
\medskip
\caption{Indecomposable representations for $A_3$ and the cyclic triangle.}
\label{tab:A3-cyclic-triangle}
\end{table}
\vspace{-.1in}
We conclude that an indecomposable positive representation of $(\overline A, \overline S)$ is uniquely
up to right-equivalence determined by its dimension vector, and
these dimension vectors are given in the second line of
Table~\ref{tab:A3-cyclic-triangle}, with the exception of ${\bf dim} \,
{\mathcal S}_2 = (0,1,0)$.
\end{example}

\begin{example}
\label{ex:A2-1}
Now let $Q$ be the quiver with three vertices $1,2,3$ and
three arrows $a: 1 \to 2$, $b: 2 \to 3$, and $c: 1 \to 3$:
$$
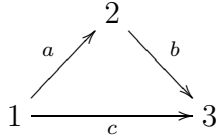
\xymatrix{
& 2\ar^{b}[rd] & \\
1\ar[ru]^a \ar[rr]_c& & 3 
}
$$
Again, the only QP on $Q$ is $(A,0)$.
We have $\mu_2(A,0) = (\overline A, \overline S)$, where
$\overline A$ is the arrow span of the quiver $\overline Q$ given
by
$$
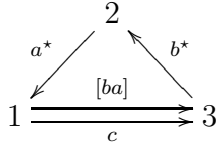
\xymatrix{
& 2\ar_{a^\star}[ld] & \\
1 \ar@<-.5ex>[rr]_{c}\ar@<.5ex>[rr]^{[ba]} & & 3\ar[lu]_{b^\star}
}
$$
and $\overline S = b^\star [ba] a^\star$.
Again, positive representations of $(A,0)$ are the representations of the quiver~$Q$,
while positive representations of $(\overline A, \overline S)$ are
the representations of the quiver $\overline Q$ satisfying
the relations \eqref{eq:double-products-0}.

We consider indecomposable positive representations of $(A,0)$
with the dimension vector $(n,n,n)$ for some $n \geq 1$.
Assume that $K$ is algebraically closed.
Since $Q$ is an extended Dynkin quiver of type $A^{(1)}_2$, and
$(n,n,n)$ is an isotropic imaginary root, by Kac's extension of
Gabriel's theorem, the isomorphism classes of indecomposable
$Q$-representations of this dimension form a $1$-parametric family.
An easy check shows that these representations break into three
right-equivalence classes.
Their representatives can be described as follows.
For each of them we have $M_1 = M_2 = M_3 = K^n$, and two of the
maps $a_M, b_M, c_M$ are equal to the identity map $I$, while the
third one is the nilpotent Jordan block~$N$.
If $a_M = N$ (resp. $b_M = N$, $c_M = N$) then we denote the
corresponding $Q$-representation by $M(a)$ (resp. $M(b)$, $M(c)$).
In view of \eqref{eq:new-Mk}, if $M$ is one of these
representations then $\mu_2(M) = \overline M$ is positive, and we have
$$\overline M_2 = {\rm coker}\ b_M \oplus \ker a_M$$
(note that since $S = 0$, we have $\gamma = 0$).
It follows that $\overline {M(c)}$ has dimension vector $(n,0,n)$,
with the maps $[ba]_{\overline {M(c)}}\ , c_{\overline {M(c)}}: K^n
\to K^N$ given by $[ba]_{\overline {M(c)}} = I, c_{\overline {M(c)}}= N$.
Also both representations $\overline {M(a)}$ and $\overline {M(b)}$
have dimension vector $(n,1,n)$.
In each of them, the arrows $[ba]$ and $c$ act as
$[ba] = N, c = I$.
We also have $b^\star_{\overline {M(a)}} = 0$, while
the map $a^\star_{\overline {M(a)}}: K \to K^n$ has
${\rm im}\ a^\star_{\overline {M(a)}} = \ker N$; similarly,
$a^\star_{\overline {M(b)}} = 0$, while
the map $b^\star_{\overline {M(b)}}: K^n \to K$ has
$\ker b^\star_{\overline {M(a)}} = {\rm im}\  N$.
\end{example}

\begin{example}
\label{ex:double-cyclic-triangle}
Our last example deals with the QP $(A,S)$ from
Example~\ref{ex:double-triangle}.
Thus, the quiver in question has three vertices $1,2,3$ and
six arrows $a_1, a_2: 1 \to 2$, $b_1, b_2: 2 \to 3$ and
$c_1, c_2: 3 \to 1$; and the potential~$S$ is given by
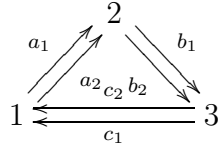
\eqref{eq:S-double-triangle}.
$$
\xymatrix{
& 2\ar@<.5ex>[rd]^{b_1}\ar@<-.5ex>[rd]_{b_2} & \\
1\ar@<.5ex>[ru]^{a_1}\ar@<-.5ex>[ru]_{a_2} & & 3\ar@<-.5ex>[ll]_{c_2}\ar@<.5ex>[ll]^{c_1}}
$$
To specify a positive representation $M$ of $(A,S)$, we need to define three vector
spaces $M_1, M_2, M_3$, and six linear maps
$(a_1)_M, (a_2)_M: M_1 \to M_2$, $(b_1)_M, (b_2)_M: M_2 \to M_3$,
and $(c_1)_M, (c_2)_M: M_3 \to M_1$.
In our case, $J(S)$ is the closure of the ideal in $R\langle \langle A \rangle \rangle$
generated by six elements
$$c_1 b_1, b_1 a_1, a_1 c_1, c_2 b_2, b_2 a_2, a_2 c_2.$$
Thus, all the compositions $(c_1)_M (b_1)_M, \dots, (a_2)_M
(c_2)_M$ must be equal to~$0$.

We first consider the indecomposable positive
representation $M$ of $(A,S)$ given by:
\begin{equation}
\label{eq:band-1-0}
M_1 = M_2 = K, \quad M_3 = 0; \quad (a_1)_M = (a_2)_M = 1.
\end{equation}
Let us compute $\mu_2(M) = (\overline M, \overline V)$.
First of all, the QP $\mu_2(A,S) = (\overline A, \overline S)$
was computed in Example~\ref{ex:double-triangle}: recall that the
arrows in $\overline A$ are $a_1^\star, a_2^\star, b_1^\star, b_2^\star,
[b_1 a_2], [b_2 a_1]$, and the potential $\overline S$ is given by
$$\overline S = [b_1 a_2] a_2^\star b_1^\star +
[b_2 a_1] a_1^\star b_2^\star.$$
To compute $\overline M$ and  $\overline V$, we apply
\eqref{eq:M-unchanged} and \eqref{eq:new-Mk}
to the triangle \eqref{eq:triangle} given by
$$M_{\rm in} = K^2, \,\, M_k = M_2 = K, \,\, M_{\rm out} = \{0\},
\,\, \alpha =
\begin{pmatrix}
1 & 1 \end{pmatrix}$$
(so we have $\beta = 0$ and $\gamma = 0$).
It follows that $\overline V = \{0\}$, i.e., $\mu_2(M)$ is
positive; we also have
$\overline M_1 = M_1 = K$, $\overline M_3 = M_3 = \{0\}$,
and
$$\overline M_2 = \ker \alpha = K \cdot \begin{pmatrix}
1 \\ -1 \end{pmatrix}$$
(this is the third term in the decomposition of $\overline M_k$ in
\eqref{eq:new-Mk}).
Since $M_3 = 0$, the arrows $b_1^\star, b_2^\star,
[b_1 a_2], [b_2 a_1]$ act as $0$ in $\overline M$.
As for $a_1^\star$ and  $a_2^\star$, their action is given by
the second equality in \eqref{eq:alpha-beta-mutated}
(note that the choice of a splitting \eqref{eq:sigma}
is immaterial here).
Namely, identifying $\overline M_2$ with $K$ via choosing
$\begin{pmatrix}
1 \\ -1 \end{pmatrix}$ as the standard basis vector,
we obtain
$$(a_1^\star)_{\overline M} = 1, \quad (a_2^\star)_{\overline M} =
-1$$
as maps $\overline M_2 = K \to K = \overline M_1$.

Note that the resulting representation $\mu_2 (M)$ can be
conveniently described as follows: by renumbering the vertices of
our quiver via
\begin{equation}
\label{eq:renumbering-vertices-1}
1'= 2, \quad 2'=1, \quad 3'=3,
\end{equation}
and setting
\begin{equation}
\label{eq:renumbering-arrows-1}
a'_1 = -a_2^\star, \,\,a'_2 = a_1^\star, \,\,b'_1 = [b_1a_2],
\,\,b'_2 = [b_2a_1],\,\, c'_1 = -b_1^\star,\,\, c'_2 = b_2^\star,
\end{equation}
the representation $\mu_2 (M)$ gets identified
with the initial representation $M$ of the initial
QP $(A,S)$.

The mutation $\mu_1(M)$ can be computed in a similar way.
But since we have already computed the QP $\mu_2(A,S) = (\overline A,
\overline S)$, we find it more convenient to renumber the vertices via
$1'= 3$, $2'=1$, $3'=2$, so that $\mu_1(M)$ gets
identified with $\mu_2({M}')$, where ${M}'$
is given by:
\begin{equation}
\label{eq:band-0-1}
M'_1 = 0, \quad M'_2 = M'_3 = K, ; \quad (b_1)_{M'} = (b_2)_{M'} = 1.
\end{equation}
Now the triangle \eqref{eq:triangle} is given by
$$M'_{\rm in} = \{0\}, \,\, M'_k = M'_2 = K, \,\, M'_{\rm out} = K^2,
\,\, \beta =
\begin{pmatrix}
1 \\ 1 \end{pmatrix}$$
(so we have $\alpha = 0$ and $\gamma = 0$).
It follows that $\mu_2({M}')$ is positive, and we have
$\overline {M'}_1 = M'_1 = \{0\}$, $\overline {M'}_3 = M'_3 = K$,
and
$$\overline {M'}_2 = K^2/K \cdot \begin{pmatrix}
1 \\ 1 \end{pmatrix}$$
(this is the first term in the decomposition of $\overline M_k$ in
\eqref{eq:new-Mk}).
Since $M'_1 = 0$, the arrows $a_1^\star, a_2^\star,
[b_1 a_2], [b_2 a_1]$ act as $0$ in $\overline {M'}$.
As for $b_1^\star$ and  $b_2^\star$, their action is given by
the first equality in \eqref{eq:alpha-beta-mutated}
(note that the choice of a splitting \eqref{eq:rho}
is immaterial here).
Namely, identifying $\overline {M'}_2$ with $K$ via choosing
$\pi(\begin{pmatrix}
1 \\ 0 \end{pmatrix}) = - \pi(\begin{pmatrix}
0 \\ 1 \end{pmatrix})$ as the standard basis vector,
we obtain
$$(b_1^\star)_{\overline {M'}} = -1, \quad (b_2^\star)_{\overline {M'}} = 1$$
as maps $\overline {M'}_3 = K \to K = \overline {M'}_2$.

As above, by renumbering the vertices of
our quiver via
\begin{equation}
\label{eq:renumbering-vertices-2}
1'= 1, \quad 2'=3, \quad 3'=2,
\end{equation}
and setting
\begin{equation}
\label{eq:renumbering-arrows-2}
a'_1 = [b_2a_1], \,\,a'_2 = [b_1a_2], \,\,b'_1 = b_2^\star,
\,\,b'_2 = -b_1^\star,\,\, c'_1 = a_1^\star,\,\, c'_2 = -a_2^\star,
\end{equation}
the resulting representation $\mu_2 ({M}')$ gets identified
with the initial representation ${M}'$.

We now include the representations $M$ and ${M}'$
given by \eqref{eq:band-1-0} and \eqref{eq:band-0-1}
into a family of positive representations of $(A,S)$ defined as
follows: for every pair of nonnegative integers $(m,n) \neq
(0,0)$, we define the positive representation $M = M(m,n)$
of $(A,S)$ by setting
\begin{equation}
\label{eq:band-general-spaces}
M_1 = K^m, \,\, M_2 = K^{m+n}, \,\, M_3 = K^n
\end{equation}
and
\begin{align}
\label{eq:band-general-maps}
&(a_1)_M =\begin{pmatrix} I_m \\ 0  \end{pmatrix}, \quad
(a_2)_M =\begin{pmatrix} 0 \\ I_m   \end{pmatrix}, \quad
(b_1)_M =\begin{pmatrix} 0 & I_n \end{pmatrix},\\
\nonumber
&(b_2)_M =\begin{pmatrix} I_n & 0 \end{pmatrix}, \quad
(c_1)_M = 0, \quad (c_2)_M = 0,
\end{align}
where $I_n$ is the $n\times n$ identity matrix.

We refer to the representations $M(m,n)$ as well as
those obtained from them by renumbering the vertices
as \emph{band representations}; they are a special case of band
modules studied in \cite{BR,F} in the context of
string algebras.
Note that both representations $M$ and ${M}'$
treated above are indeed special cases of band representations: we have
$M = M(1,0)$, ${M}' = M(0,1)$.
By a direct generalization of the above computations, we obtain the following
proposition.

\begin{proposition}\
\label{pr:band}
\begin{enumerate}
\item If $m \geq n$ then after renumbering of vertices as in \eqref{eq:renumbering-vertices-1}
and the change of arrows as in \eqref{eq:renumbering-arrows-1},
the representation $\mu_2(M(m,n))$ can be identified
with $M(m-n,n)$.
\item If $m \leq n$ then after renumbering of vertices as in \eqref{eq:renumbering-vertices-2}
and the change of arrows as in \eqref{eq:renumbering-arrows-2},
the representation $\mu_2(M(m,n))$ can be identified
with $M(m,n-m)$.
\end{enumerate}
\end{proposition}

Remembering Theorem~\ref{th:rep-mutation-involutive}, we obtain
the following corollary.

\begin{corollary}\
\label{cor:band-mu1-3}
\begin{enumerate}
\item After renumbering of vertices as in \eqref{eq:renumbering-vertices-1},
the representation $\mu_1(M(m,n))$ becomes
right-equivalent to $M(m+n,n)$.
\item After renumbering of vertices as in \eqref{eq:renumbering-vertices-2},
the representation $\mu_3(M(m,n))$ becomes
right-equivalent to $M(m,m+n)$.
\end{enumerate}
\end{corollary}

\begin{corollary}
\label{cor:band-reps-mutation-closed}
The class of band representations is
closed under mutations.
\end{corollary}

Note that if we iterate the mutations in Proposition~\ref{pr:band},
the pair $(m,n)$ gets transformed according to the Euclid algorithm for
finding $\gcd(m,n)$.
Thus, after a sequence of mutations (and appropriate renumberings
of vertices), every $M(m,n)$ can be transformed into $M(\gcd(m,n),0)$.
Since $M(d,0)$ is obviously isomorphic to the direct sum
of~$d$ copies of $M(1,0)$, by backtracking this sequence of
mutations, we obtain the following well-known corollary.

\begin{corollary}
\label{cor:band-reps-indecomposable}
The representation $M(m,n)$ is indecomposable if and
only if $m$ and $n$ are relatively prime.
Furthermore, if $\gcd(m,n)=d$ then $M(m,n)$ is
right-equivalent to the direct sum
of~$d$ copies of $M(m/d,n/d)$.
\end{corollary}
\end{example}

\begin{remark}
\label{rem:string-modules}
By the same methods as above, one can compute all the mutations
for another family of representations of the QP $(A,S)$ in
Example~\ref{ex:double-cyclic-triangle}: \emph{string modules}
introduced and studied in \cite{BR,F}.
\end{remark}

\section{Some open problems}
\label{sec:open-problems}

Here we collect some natural questions that we find important for
better understanding of QPs and their representations.
In what follows, suppose that $(A,S)$ is a reduced QP with the
Jacobian algebra ${\mathcal P}(A,S)$.
Let ${\mathcal M}(A,S)$ denote the category of
finite dimensional ${\mathcal P}(A,S)$-modules.
Suppose also that $k\in Q_0$ is a vertex satisfying
\eqref{eq:no-2-cycles-thru-k}, so that the mutated reduced QP
$\mu_k(A,S)$ is well-defined.


\begin{question}
Is the isomorphism class of ${\mathcal P}(A,S)$
determined by the equivalence class of the category ${\mathcal M}(A,S)$?
\end{question}

\begin{question}
Is the isomorphism class of ${\mathcal P}(\mu_k(A,S))$
determined by the isomorphism class of ${\mathcal P}(A,S)$?
\end{question}

\begin{question}
Is the category ${\mathcal M}(\mu_k(A,S))$
determined up to equivalence by ${\mathcal M}(A,S)$?
\end{question}

Note that the right-equivalence class of $(A,S)$ is \emph{not} determined by the
isomorphism class of the Jacobian algebra ${\mathcal P}(A,S)$.
In fact, we can construct a QP $(A,S)$ which is \emph{not}
right-equivalent to $(A,cS)$ for some nonzero $c \in K$, while we
obviously have ${\mathcal P}(A,S) = {\mathcal P}(A,cS)$ (the possibility of
such an example was brought to our attention by Bill Crawley-Boevey).

We conclude with the following intriguing question.

\begin{question}
Is there a proper analogue of the cluster category for a non-acyclic quiver with potential?
\end{question}

\section{Appendix. Proof of Lemma~\ref{lem:tr-IJ}}
\label{sec:topological-appendix}

We include Lemma~\ref{lem:tr-IJ} into a more general setup.
We call a $K$-vector space $V$ a \emph{$C$-space} (for the lack of a better
term) if $V$ has an increasing filtration $\{0\} = V_0\subseteq V_1\subseteq
\cdots$ such that all $V_n$ are finite dimensional, and $V =
\bigcup_{n \geq 0} V_n$.
(Equivalently, $V$ is either finite dimensional, or it has
countable dimension.)
The class of $C$-spaces is clearly closed under taking subspaces,
quotient spaces, finite direct sums, and finite tensor products.
We always consider $C$-spaces equipped with discrete topology; in
particular, this applies to the base field~$K$.

We refer to the dual space $V^\star$ of a $C$-space~$V$ as a
\emph{$D$-space} (the dual is understood as the space of all
linear forms $V \to K$).
Most of the properties of $D$-spaces discussed below are undoubtedly
well-known; for the convenience of the reader, we provide a
self-contained treatment.

\begin{example}
\label{ex:RA-D-space}
The complete path algebra $R\langle \langle A
\rangle \rangle$ can be naturally viewed as a $D$-space $V^\star$,
corresponding to the $C$-space $V = \oplus_{d=0}^\infty
{(A^d)}^\star$, and the filtration $(V_n)$ given by
$$V_n = \oplus_{d=0}^{n-1} {(A^d)}^\star \quad (n \geq 1).$$
\end{example}

For a subspace $W$ of $V$, we denote by
$W^\perp \subset V^\star$ its orthogonal complement, that is,
$$W^{\perp}=\{f\in V^\star\mid f(W)=0\}.$$
We make $V^\star$ into a topological vector space by taking the
sets $V_n^\perp$ for all $n \geq 0$ as a basic system of open neighborhoods of~$0$.
In particular, in Example~\ref{ex:RA-D-space}, we have
$V_n^\perp = {\mathfrak m}(A)^n$, so the $D$-space topology on
$R\langle \langle A \rangle \rangle$ coincides with the topology
introduced in Section~\ref{sec:path-algebras}.

Since every $v \in V$ belongs to some~$V_n$, a sequence
$f_1, f_2, \dots$ converges in $V^\star$ if and only if,
for every $v \in V$, the sequence $(f_k(v))$
stabilizes as $k \to \infty$.
This implies in particular that $W^\perp$ is a closed subspace
of $V^\star$ for every subspace $W$ of $V$.
In fact, the converse is also true.

\begin{lemma}
\label{lem:closed-perp}
A vector subspace $Z$
of $V^\star$ is closed if and only if $Z = W^{\perp}$ for some subspace $W$ of $V$.
\end{lemma}

\begin{proof}
Let $Z$ be a vector subspace of $V^\star$.
Let
$$\text{$W = \{v \in V \mid f(v) = 0$ for $f \in Z\}$.}$$
It suffices to show that $W^\perp$ is contained in the closure~$\overline Z$ of~$Z$.
Let $f \in W^\perp$.
Restricting $f$ to each finite-dimensional subspace $V_n$ of $V$, we
conclude that $f \mid_{V_n} = h_n \mid_{V_n}$ for some $h_n \in
Z$.
Thus, the sequence $h_1, h_2, \dots,$ converges to~$f$, implying
that $f \in \overline Z$, as required.
\end{proof}

In view of Lemma~\ref{lem:closed-perp}, for every closed subspace
$Z$ of $V^\star$, the spaces $Z$ and $V^\star/Z$ can be naturally viewed
as $D$-spaces: indeed, we have
$$Z = W^\perp = (V/W)^\star, \quad
V^\star/Z  = V^\star/W^\perp = W^\star$$
for some subspace $W$ of $V$.
The following lemma is immediate from the definitions.

\begin{lemma}
\label{lem:induced-topology}
For every closed subspace $Z \subseteq V^\star$, the
$D$-space topologies on $Z$ and $V^\star/Z$ coincide with the topologies induced
from~$V^\star$.
In particular, the embedding $Z \to V^\star$ and the projection
$V^\star \to V^\star/Z$ are continuous.
\end{lemma}

\begin{lemma}
\label{lem:sum}
If $Z_1$ and $Z_2$ are closed subspaces of $V^\star$, then
$Z_1+Z_2$ is a closed subspace of $V^\star$ as well.
\end{lemma}

\begin{proof}
By Lemma~\ref{lem:closed-perp}, $Z_1 = W_1^\perp$ and
$Z_2 = W_2^\perp$ for some subspaces $W_1$ and $W_2$ of $V$.
Choosing some direct complements of $W_1 \cap W_2$ in $W_1$ and
$W_2$, and a direct complement of $W_1 + W_2$ in $V$, it is easy
to see that
$$Z_1 + Z_2 = W_1^\perp + W_2^\perp = (W_1 \cap W_2)^\perp,$$
proving that $Z_1 + Z_2$ is closed.
\end{proof}

\begin{lemma}
\label{lem:continuous-linear-map}
Let $U$ and $V$ be $C$-spaces, and $U^\star$ and $V^\star$ be the
corresponding $D$-spaces.
A linear map $\alpha: U^\star \to V^\star$ is continuous if and
only if $\alpha = \beta^\star$ for some linear map $\beta: V \to U$.
\end{lemma}

\begin{proof}
First let us show that $\alpha = \beta^\star$ is continuous.
By the definition, it is enough to show that, for every $n$, there
exists an index~$k$ such that $U_k^\perp \subset \alpha^{-1}(V_n^\perp)$.
Since the subspace $\beta(V_n) \subset U$ is finite dimensional,
it is contained in some $U_k$, implying the desired inclusion
$U_k^\perp \subset \alpha^{-1}(V_n^\perp)$.

Conversely, suppose $\alpha: U^\star \to V^\star$ is a continuous
linear map.
Let $v \in V$.
Then the linear form $f \mapsto \alpha(f)(v)$ is a continuous
linear map $U^\star \to K$, and so its kernel is a closed subspace
of $U^\star$.
Using Lemma~\ref{lem:closed-perp}, we conclude that there exists
a unique $u \in U$ such that $\alpha(f)(v) = f(u)$ for all
$f \in U^\star$.
The correspondence $v \mapsto u$ is the desired linear map $\beta: V \to U$
such that $\alpha = \beta^\star$.
\end{proof}

\begin{lemma}
\label{lem:closed-map}
Any continuous linear map of
$D$-spaces $\alpha: U^\star \to V^\star$ sends closed vector
subspaces of $U^\star$ to closed vector subspaces of $V^\star$.
\end{lemma}

\begin{proof}
Let $Z\subseteq U^\star$ be a closed vector subspace.
By Lemma~\ref{lem:closed-perp}, $Z = W^\perp$ for some vector
subspace $W \subset U$.
Also by Lemma~\ref{lem:continuous-linear-map}, we have
$\alpha = \beta^\star$ for a linear map $\beta: V \to U$.
The definitions imply that $\alpha(Z) =
\beta^\star (W^\perp) = (\beta^{-1}(W))^\perp$, hence
$\alpha(Z)$ is  a closed subspace of $V^\star$, as claimed.
\end{proof}

We will call a $D$-space $V^\star$ a \emph{$D$-algebra} if it has
a structure of an associative $K$-algebra such that
$V_m^\perp V_n^\perp \subset V_{m+n}^\perp$ for all $m, n \geq 0$.
In particular, $R\langle \langle A \rangle \rangle$ is a
$D$-algebra.

\begin{lemma}
\label{lem:If}
If $I_1, \dots, I_N$ are closed subspaces
in a $D$-algebra $V^\star$, then the subspace
$I_1 f_1 + \cdots + I_N f_N$ is closed for every
$f_1, \dots, f_N \in V^\star$.
In particular, finitely generated left ideals in $V^\star$ are closed.
\end{lemma}

\begin{proof}
By the definition of a $D$-algebra, the operator of right multiplication with any $f \in V^\star$ is continuous.
Thus each subspace $I_k f_k$ is closed by Lemma~\ref{lem:closed-map}, and our assertion
follows from Lemma~\ref{lem:sum}.
\end{proof}

Recall from Definition~\ref{def:trace-space}, that the trace space
of a $D$-algebra $V^\star$ is the quotient
$\Tr (V^\star) = V^\star/\{V^\star,V^\star\}$, where $\{V^\star,V^\star\}$
is the closure of the vector subspace in~$V^\star$ spanned by all commutators.
We denote by $\pi: V^\star \to \Tr(V^\star)$ the canonical projection.
By Lemma~\ref{lem:induced-topology}, $\pi$ is continuous with respect
to the $D$-space topologies.

In view of Proposition~\ref{pr:cyclical-equiv-thru-trace},
the assertion of Lemma~\ref{lem:tr-IJ} is a special case of the
following.

\begin{lemma}
\label{lem:tr-IJ-general}
Let $I$ be a closed (two-sided) ideal of a $D$-algebra $V^\star$,
and $J$ be the closure of an ideal generated by
finitely many elements $f_1, f_2, \dots, f_N$.
Then the subspace $\pi(I J) \subseteq \Tr (V^\star)$ is equal to
$\pi(I f_1 + \cdots + I f_N)$.
\end{lemma}

\begin{proof}
Let $J^0$ be the ideal generated by $f_1, f_2, \dots, f_N$, that
is, the linear span of elements of the form $uf_kv$ with $u, v \in
V^\star$ and $k = 1, \dots, N$.
Thus the ideal $IJ^0$ is the linear span of elements of the form $guf_kv$
with $g \in I$.
By the definition, we have $\pi(guf_kv) = \pi (vguf_k)$, and so
$\pi(I J^0) = \pi(I f_1 + \cdots + I f_N)$.
Since $I J^0$ is dense in $IJ$, it follows that
$\pi(I J^0)$ is dense in $\pi(IJ)$.
On the other hand, the subspace
$\pi(I f_1 + \cdots + I f_N) \subseteq \Tr (V^\star)$
is closed by Lemmas~\ref{lem:If} and \ref{lem:closed-map}.
We conclude that $\pi(I J^0) = \pi(I f_1 + \cdots + I f_N) =
\pi(IJ)$, as required.
\end{proof}
\section*{Acknowledgments}

We thank Victor Ginzburg for helpful comments and bibliographic
guidance, Bill Crawley-Boevey for valuable remarks,
Daniel Labardini Fragoso for careful reading
of the manuscript and many useful suggestions and
Christopher Herzog for making us aware of and discussing
the use of quivers with potentials in theoretical physics.
After this paper was completed, we have been informed by Maxim
Kontsevich that he has rediscovered some of the results in
Sections~\ref{sec:mut-invariants} and \ref{sec:generic}
in the context of  $A_\infty$-categories.

\end{document}